\documentclass{amsart} 

\usepackage{color}
\usepackage{tikz}
\usepackage{amsmath} 
\usepackage{amssymb}
\usepackage{hyperref} 
\usepackage{enumerate}
\colorlet{symbols}{black} % Change this to change color of symbols
\tikzset{
	dot/.style={circle,fill=symbols,draw=symbols,inner sep=0pt,minimum size=0.4pt},
	basic/.style={draw=symbols},
	>=stealth,
	}

\input shuffle

\renewcommand{\d}{\partial} 
\newcommand{\ov}{\overline}

\renewcommand{\Re}{{\rm Re}\,}
\renewcommand{\Im}{{\rm Im}\,}
\newcommand{\loc}{{\rm loc}\,}
\renewcommand{\sup}{{\rm sup}\,}

\newcommand{\Supp}{{\rm Supp}\,}
 
\newcommand{\sign}{\text{sign}}

\def\Tc{T_{\hbox{\tiny c}}}

\def\sechL{\sech_{\hbox{\tiny\tiny{L}}}}

\def\cH{\mathcal{H}}

\def\C{\mathop{\mathbb C\kern 0pt}\nolimits}
\def\N{\mathop{\mathbb N\kern 0pt}\nolimits}
\def\R{\mathop{\mathbb R\kern 0pt}\nolimits}
\def\Z{\mathop{\mathbb Z\kern 0pt}\nolimits}
\def\bq{\mathop{\mathbf q\kern 0pt}\nolimits}
\def\bpsi{\mathop{\mathbf{\Psi}\kern 0pt}\nolimits}

\def\da{\,{\rm d}a\,}

\def\dt{\,{\rm d}t\,}
\def\dy{\,{\rm d}y\,}
\def\dz{\,{\rm d}z\,}
\def\dx{\,{\rm d}x\,}   
\def\dm{\,{\rm d}m\,}  
\def\dxi{\,{\rm d}\xi\,} 
\def\dt'{\,{\rm d}t'\,}
\def\dtau{\,{\rm d}\tau\,}

\def\ds'{\,{\rm d}s'\,}
\def\sech{\,{\rm sech}\,}
\def\tanh{\,{\rm tanh}\,}

\def\cC{\mathcal{C}}

\def\cE{\mathcal{E}}

\def\Ic{\mathcal{I}_{\hbox{\small cut}}}
\def\cM{\mathcal{M}}

\def\cU{\mathcal{U}}
\def\cV{\mathcal{V}}
\def\cP{\mathcal{P}}
\def\cR{\mathcal{R}}
\def\cS{\mathcal{S}}

\newtheorem{defi}{Definition}[section]
\newtheorem{thm}{Theorem}[section]
\newtheorem{lem}{Lemma}[section]
\newtheorem{rmk}{Remark}[section]

\newtheorem{prop}{Proposition}[section]

\title[Conserved energies for the Gross-Pitaevskii equation]
{Conserved energies for the one dimensional Gross-Pitaevskii equation} 

 \author[H. Koch]{Herbert Koch}
\address [H. Koch]% 
 {Mathematical Institute, University Bonn, Endenicher Allee 60, 53115 Bonn, Germany}
\email{koch@math.uni-bonn.de}

 \author[X. Liao]{Xian Liao}
\address [X. Liao]% 
 {Institute of Analysis, Karlsruhe Institute for Technology, Englerstrasse 2, 76131 Karlsruhe, Germany}
\email{xian.liao@kit.edu} 

\date{}

\begin{document} 
\maketitle
\begin{abstract}
We prove the global-in-time well-posedness of the one dimensional Gross-Pitaevskii equation in the energy space, which is a complete metric space equipped with a newly introduced metric and with the energy norm describing the $H^s$ regularities of the solutions.
We establish  a family of conserved energies for the one dimensional Gross-Pitaevskii equation,  such that the energy norms of the solutions are conserved globally in time. 
This family of energies is also conserved by the complex modified Korteweg-de Vries flow.
\end{abstract}


\vspace{.1cm}


\tableofcontents

\section{Introduction}
We consider the one dimensional Gross-Pitaevskii equation
\begin{equation}\label{GP}
i\d_t q+\d_{xx}q=2q(|q|^2-1), 
\end{equation}
where $(t,x)\in\R^2$ denote the time and space variables  and $q=q(t,x):\R\times\R\mapsto\C$ denotes the unknown complex-valued wave function.

It was used by E.P. Gross \cite{Gross} and L.P. Pitaevskii \cite{Pitaevskii} to describe the oscillations of a Bose gas at zero temperature.
In nonlinear optics, the equation \eqref{GP} models the propagation of a monochromatic wave in a defocusing medium and in particular the dark/black solitons with $|q|=1$ at infinity arise as solutions of \eqref{GP}. 
See the review paper \cite{KLD} for more physical interpretations. 

The Gross-Pitaevskii equation \eqref{GP} can be viewed as the defocusing cubic nonlinear Schr\"odinger equation (NLS), but with a nonstandard boundary condition at infinity: $|q|\rightarrow 1$ as $|x|\rightarrow \infty$. 
This nonzero boundary condition brings a substantial difference between \eqref{GP} and (NLS) (for which we assume zero boundary condition at infinity): For example, the  former equation has soliton solutions (e.g. the black soliton solution $q(t,x)=\tanh(x)$) while the latter  equation possesses  scattering phenomenon. 
One will see below that the solution space for the Gross-Pitaevskii equation \eqref{GP} is much  more delicate and we will derive a family of conserved energies which describe all  the $H^s$, $s>\frac 12$ regularities of the solutions in a nonstandard way.

\smallbreak

The equation \eqref{GP} can be viewed as a Hamiltonian evolutionary equation associated to the Ginzburg-Landau energy
\begin{equation}\label{cEGL}
\cE_{GL}(q)=\frac{1}{2}\int_{\R} \bigl((|q|^2-1)^2+ |\d_x q|^2 \bigr)\dx,
\end{equation}
with respect to the symplectic form
$\omega(u,v)=\Im \int_{\R} u\bar v \dx.$ 
P.E. Zhidkov \cite{Zhidkov87} proved the local-in-time well-posedness of the Gross-Pitaevskii equation  \eqref{GP}   in the so-called Zhidkov's space $Z^k$, $k=1,2,\cdots$, which is the closure of the space 
$\{q\in C^k(\R)\cap L^\infty(\R)\,|\, \d_x q\in H^{k-1}(\R)\}$ for the norm
{\small\begin{equation}\label{Zhidkovk}
\|q\|_{Z^k}=\|q\|_{L^\infty}+\sum_{1\leq l\leq k} \|\d_x^l q\|_{L^2},
\end{equation}}
and in particular when $k=1$, under the  initial finite-energy assumption $\cE_{GL}(q_0)<\infty$, the finite-energy solution exists globally in time.
See also \cite{BGS, BGSS, Gallo, Gallo07, Zhidkov87,Zhidkov} for more results in the $n$-dimensional case, with $Z^k=Z^k(\R)$ above replaced by  $Z^k(\R^n)$, $k>\frac n2$.  
In dimension $n=2$ or $3$,  P. G\'erard  \cite{Gerard06} showed the global-in-time well-posedness of \eqref{GP} in the energy space $Y^1=\{ q\in H^1_{\loc}(\R^n): |q|^2-1\in L^2(\R^n), \, \nabla q\in L^2(\R^n) \}$, endowed with the metric distance
{\small\begin{align*}
&d_{Y^1}(p,q)=\|p-q\|_{Z^1+H^1}+\bigl\||p|^2-|q|^2\bigr\|_{L^2},
\\
&\hbox{ with }\|u\|_{A+B}=\inf\{\|u_1\|_{A}+\|u_2\|_{B}\,|\, u=u_1+u_2,
\, u_1\in A,\, u_2\in B\},
\end{align*}}
and more topological properties of this complete metric space $Y^1$ can be found    in   \cite{Gerard08}.
%(noticing that in higher dimensions $n=3,4$ $Y^1$ is nothing but $1+H^1$).
Particular attention has been paid to show the existence (or non-existence) of the travelling wave solutions in  \cite{BS, Chiron, Gravejat, Maris} and there is also  rich literature contributed to their stability or instability issues: See   \cite{BGSS08, BGSS15, Lin} and the references therein.
Most authors in the study of the stability issues adopt the following metric distance  in the energy space: 
{\small\begin{align*}
&d_E(p,q)=\|p-q\|_{L^2(\{x\in\R^n:|x|\leq 1\})}+\||p|^2-|q|^2\|_{L^2(\R^n)}+\|\nabla p-\nabla q\|_{L^2(\R^n)}.
\end{align*}}
In higher dimensional case $n\geq 4$,
%$n=3$ a class of function, $n=2$ not completely correct
\cite{GNT} (see \cite{GNT09} for the results when $n=2,3$) established the scattering theory for the Gross-Pitaevskii equation with the initial data of form $1+\varphi$ and $\varphi\in H^s$, $s\geq \frac n2-1$ sufficiently small.

In this paper we take the general \emph{energy spaces}  as follows 
\begin{equation}\label{Xs}
X^s= \bigl\{ q\in H^s_{\loc}(\R)\,:\, |q|^2-1\in H^{s-1}(\R), \quad \d_x q\in H^{s-1}(\R) 
\bigr\}/ {\mathbb{S}^1},
\quad s\geq 0,
\end{equation}
where $\mathbb{S}^1$ denotes the unit circle, i.e. we identify functions which differ by a multiplcative constant  of modulus $1$. 
Recall that for $s\in\R$, the Sobolev space $H^s(\R)$ consists of tempered distributions $f$ with finite $H^s(\R)$-norm which is  defined as follows:
 \begin{equation*} 
 \|f\|_{H^s(\R)}
=\Bigl(\int_{\R} (1+\xi^2)^s |\hat{f}(\xi)|^2\dxi\Bigr)^{\frac12},
\end{equation*} 
where $\widehat{f}(\xi)$ denotes the Fourier transform of $f(x)$.
We endow the set of functions $X^s$ with the following metric $d^s(\cdot,\cdot)$
\footnote{If $p,q\in X^s$, then we indeed have   $\bigl\||p|^2-|q|^2\bigr\|_{H^{s-1}(\R)} \leq cd^s(p,q)$ by \eqref{Es,ds} below.
}:
 \begin{equation}\label{ds}\begin{split}
&d^s(p,q)=
\Bigl(\int _{ \R}
 \inf_{|\lambda|=1} \bigl\|\sech( \cdot-y) (\lambda p-q)\bigr\|_{H^{s}(\R)}^2 \dy
% +\bigl\||p|^2-|q|^2\bigr\|_{H^{s-1}(\R)} ^2
 \Bigr)^{\frac 12}, 
\end{split}\end{equation}
 where  $\sech(x)=\frac{2}{e^{x}+e^{-x}}$
\footnote{We can take any other strictly positive smooth function which decays fast at infinity instead of $\sech(x)$.} and we will prove in Section \ref{sec:metric} the following theorem:
\begin{thm}\label{thm:metric}
Let $X^s$, $d^s(\cdot,\cdot)$, $s\geq 0$ be defined in \eqref{Xs} and \eqref{ds}. Then 
the space $(X^s, d^s(\cdot,\cdot))$ is a complete metric space, with 
  the following topological properties:
  \begin{itemize}
\item The subset  $\{q\,|\, q-1\in C^\infty_0(\R)\}$
is dense in $X^s$ and hence $(X^s, d^s(\cdot,\cdot))$ is separable. 

\item Any ball  $B_r^s(q)=\{p\in X^s\, |\, d^s(p,q)< r\}$, $r\in\R^+$, $q\in X^s$, in $X^s$ is contractible.
\item Any set $\{ q\in H^s_\loc(\R):  \Vert \d_x q \Vert_{H^{s-1}} + \Vert |q|^2-1 \Vert_{H^{s-1}} < C\}$ is contained in some ball $B^s_r(1)$ with $r$ depending on $C$.
\item Any closed  ball $\overline{B_r^s(q)}$  in $X^s$, $s>0$ is weakly sequentially compact.
 % {\color{red} I am still uneasy if this is true or proven for $s=0$ - maybe it is, but I want to have a remainder that we have to sure with this point. This was one reason for me to drop the second term in the distance.} {\color{green}I drop the second term.}
\item There is an analytic structure on $X^s$ (see Theorem \ref{thm:analytic} for details).
\end{itemize}  
\end{thm}  
\smallbreak

In the following we will define the  solution of the Gross-Pitaevskii equation \eqref{GP}.  The initial data  $q_0\in X^s$ has a representative $\tilde q_0$. 
A solution $q(t,\cdot)\in X^s$ with $t\in I$ the time interval will be the projections  of some function in  $t$: $ \tilde q(t,\cdot) \in X^s$. We define the notion of a solution in terms of the representative. 
\begin{defi}[Solutions]\label{def}
  We call $q\in \mathcal{C}(I; X^s)$, $s\geq0$  to be a   solution of the Gross-Pitaevskii equation \eqref{GP} with the initial data $q|_{t=0}=q_0\in X^s$ on the  open time interval  $I\ni 0$, if there is $\tilde q : I \to H^s_{loc}$ which  
 satisfies    that
\begin{equation}\label{WS}
 I \ni t \to  \tilde  q(t)-\tilde q(0) \in L^2, 
\end{equation}
is weakly continuous   and  
\begin{equation}\label{Strichartz}   
\Vert \tilde q(\cdot) - \tilde q_{0,\varepsilon} \Vert_{L^4([a,b]  \times \R)} \leq C,  
 \end{equation}
for some regularized initial data $\tilde q_{0,\varepsilon}$ of $\tilde q(0)$ and all $0\in [a,b] \subset I$,
such that the equation \eqref{GP} holds in the distributional sense on $I\times\R$
and $\tilde q(t)$ projects to $q(t)$.
  
\end{defi}

We have the following  well-posedness results.
\begin{thm}\label{thm:gwp}
 Let $s\geq 0$.  The Gross-Pitaevskii equation \eqref{GP} is
 loally-in-time well-posed in the metric space $(X^s, d^s)$ in the
 following sense: For any initial data $q_0\in X^s$, there exists a
 positive time $\bar t\in (0,\infty)$ and a unique local-in-time
 solution $q\in \cC((-\bar t, \bar t); X^s)$ of \eqref{GP} and for any
 $t\in (0,\bar t)$, the Gross-Pitaevskii flow map $X^s\ni q_0\mapsto
 q\in \cC([-t, t]; X^s)$ is 
 %Lipschitz 
 continuous.  Let $s>\frac12 $,
 then the above holds for all $\bar t\in\R^+$ and hence the
 Gross-Pitaevskii equation \eqref{GP} is globally-in-time well-posed
 in the metric space $(X^s, d^s)$.
\end{thm}
\begin{rmk}
Compared with  the distance function $d^s$ introduced for the \emph{nonlinear} energy space $X^s$  here, the   Zhidkov's norm $\|\cdot\|_{Z^k}$ or the metric $d_{Y^1}$  is more rigid   and the subset  $\{v\,|\, v-1\in \cS(\R)\}$ is not dense in $Z^k$ or $Y^1$.
The known global well-posedness result in $Z^1$  does not cover the above global well-posedness result in $X^1$. 
\end{rmk}

\medbreak

The equation \eqref{GP} is completely integrable   by means of the  inverse scattering method. 
According to the seminal  paper by Zakharov-Shabat \cite{ZS73},  the equation \eqref{GP}
can be viewed as the compatibility condition for the two systems
\begin{equation}\label{LaxPair}\begin{split}
&  u_x = \left( \begin{matrix}
 -i\lambda &  q \\  \bar q & i \lambda 
 \end{matrix} \right) u ,
 \\
 & u_t=i\left(\begin{matrix}
 -2\lambda^2- (|q|^2-1) & -2i\lambda q+ \d_x q
 \\
 - 2i\lambda \ov q- \d_x \ov q & 2\lambda^2+ (|q|^2-1)
 \end{matrix}\right) u,
 \end{split}\end{equation}
 where  $u:\R\times\R\mapsto \C^2$ is the unknown vector and $\lambda\in\C$ can be viewed as  parameter. More precisely, if we fix $\lambda \in C$ then the \eqref{GP} is the compatibility condition. On the other hand, if \eqref{GP} holds then the compatibility condition is satisfied for all complex numbers $\lambda$. 
The first system in \eqref{LaxPair} can be written in the form of a spectral problem $Lu=\lambda u$   of the  so-called Lax operator 
\begin{equation}\label{LaxOp}
L=\left(\begin{matrix}
i\d_x&-iq \\ i\ov q&-i\d_x
\end{matrix}\right),
\end{equation} 
and correspondingly the second system of \eqref{LaxPair} reads  as a differential operator as follows   (by eliminating $\lambda$ using the relation $\lambda u=Lu$)
\begin{equation*}
P=i\left(\begin{matrix}
2\partial_x^2-(|q|^2-1)&-q\partial_x-\partial_x q
 \\
  \bar q\partial_x+\partial_x\bar q & -2\partial_x^2+(|q|^2-1)
\end{matrix}\right).
\end{equation*}

A formal calculation shows that  $q(t,x)$ solves the equation \eqref{GP} if and only if there holds the operator evolution equation $L_t=[P; L]:=PL-LP$, i.e. the two operators $(L,P)$ form the so-called Lax-pair, which {\textit{formally}} implies the invariance of the spectra of $L$ by time evolution.
Indeed, let the skewadjoint operator $P$ generate a unitary family of evolution operators $U(t',t)$, then 
\[ L(t)    =  U^*(t',t)  L(t') U(t',t)  \]
  and $L(t)$ and $L(t')$ are similar. 
%L u=\lambda u
%\Rightarrow L_t u+L u_t=\lambda_t u+\lambda u_t,
%\hbox{ i.e. }L_t u+L P u=\lambda_t u+P L u
%\Rightarrow \lambda_t=0.
%$$   
The inverse scattering transform relates the evolution of the Gross-Pitaevskii flow to the study of the spectral and scattering property of the Lax operator $L$.
In the classical framework where $q-1$ is Schwartz function,
% the book by Faddeev and Takhtajan \cite{FT} states that
%\begin{itemize}
%\item
 the self-adjoint operator $L$ has essential spectrum $(-\infty, -1]\cup [1,\infty)$ and  at most countably many \emph{simple real} eigenvalues $\{\lambda_m\}$  on $(-1,1)$.
%\item The Cauchy problem of the Gross-Pitaevskii equation and the initial data $q_0$ with $q_0-1\in \cS$ has a unique solution $q$ such that $q(t,\cdot)-1\in \cS(\R)$ for all $t\in\R$.
%\end{itemize}
See   \cite{AKNS, CJ, DZ, DPVV, FT, GZ, ZS73} for more discussions between the potential $q$ and the spectral information of $L$.  

It is interesting to notice that the \emph{complex} defocusing  modified Korteweg-de Vries equation
\begin{equation}\label{mKdV}
\psi_t+\psi_{xxx}-6|\psi|^2\psi_x=0, \quad \psi:\R\times\R\to \C,
\end{equation}
possesses also a Lax-pair structure and
shares   the same Lax operator \eqref{LaxOp}: $L_{\hbox{\tiny mKdV}}=\left(\begin{matrix}
i\d_x&-i\psi \\ i\ov \psi&-i\d_x
\end{matrix}\right)$  as the Gross-Pitaevskii equation, although the corresponding matrix/operator $P$ reads differently as
{\small\begin{equation}\label{PmKdV}\begin{split} 
P_{\hbox{\tiny mKdV}}=\left(\begin{matrix}
-4i\lambda^3-2i\lambda|\psi|^2+(\ov\psi\psi_x-\psi\ov\psi_x)
 & 4\lambda^2\psi+2i\lambda\psi_x-\psi_{xx}+2|\psi|^2\psi
 \\
 &
 \\
4\lambda^2\ov\psi-2i\lambda\ov\psi_x-\ov\psi_{xx}+2|\psi|^2\ov\psi
 &  4i\lambda^3+2i\lambda|\psi|^2-(\ov\psi\psi_x-\psi\ov\psi_x)
 \end{matrix}\right).
 \end{split}\end{equation}}

% of the Lax-pair reformulation and the inverse scattering transform of the cubic nonlinear Schr\"odinger equations.
%the Gross-Pitaevskii equation \eqref{GP} in the classical framework $q-1\in \cS(\R)$.
%Recently \cite{DPVV} studied the inverse scattering transform in the framework $q-1\in L^1(\R)$ with sufficiently decay at infinity.

%By virtue of this Lax-pair reformulation \eqref{LaxPair}, the authors in \cite{CJ, GZ}  established  stability results for the soliton solutions of the Gross-Pitaevskii equation \eqref{GP}.

In this paper we focus on  the first system in \eqref{LaxPair}, i.e. the spectral problem $Lu=\lambda u$ for the Lax operator $L$. It is not hard to see that $L$ is selfadjoint. We study in particular  the \emph{time-independent} transmission coefficient   $T^{-1}(\lambda)$ associated to it. 
For the cubic nonlinear Schr\"odinger equation case, Koch-Tataru \cite{KT} (see also \cite{KVZ}) made use of the corresponding invariant transmission coefficient to establish a family of   conserved energies which are equivalent to the $H^s$, $s>-\frac 12$-norms of the solutions and hence all the $H^s$-regularities are preserved a priori for regular initial data. 
We will adopt the idea in \cite{KT} to formulate the conserved energies for the Gross-Pitaevskii equation \eqref{GP} and the defocusing modified Korteweg-de Vries equation \eqref{mKdV}.

The first obstacles on the way are  the   mass $\cM$ and momentum $\cP$:
\begin{equation}\label{MP}
\cM=\int_{\R} (|q|^2-1)\dx,
\quad \cP=\Im\int_{\R} q\d_x \bar q\dx,
\end{equation} 
which are  only well-defined  under more integrability  assumptions on $|q|^2-1, \d_x q$,  rather than the mere $L^2$-type boundedness assumptions for $q\in X^s$.
In the classical setting where $(q-1)$ is a Schwartz function,  we have the following expansion for the logarithm of the transmission coefficient (see \cite{FT}): 
There exist   countably many \emph{real} numbers $\{\cH^{n}\}_{n\geq 0}$   such that for any $k\geq 1$,  
  \begin{equation}\label{lnT:FT}
 \begin{split}
&\ln T^{-1}(\lambda)
=i\sum_{l=0}^{k-1}  \cH^{l}(2z)^{-l-1}+(\ln T^{-1}(\lambda))^{\geq k+1},
  \quad \Im\lambda>0,
\\
&
 \hbox{ with } |(\ln T^{-1}( \lambda))^{\geq k+1}|=O(|\lambda|^{-k-1})
\hbox{ as }|\lambda|\rightarrow\infty,
 \end{split}
 \end{equation}
 where $(\lambda,z)$ stays on the upper sheet of a Riemann surface
  $\{(\lambda,z)\in\C^2|\lambda^2-z^2=1, \,\Im z>0\}$ (see Subsection \ref{subss:Riemann} for more details).
  We also have the corresponding expansion for $\ln T^{-1}(\lambda)$ as $|\lambda|\rightarrow\infty$ for $\Im\lambda<0$,  by use of the symmetry 
$\ln T^{-1}({\lambda})=\ov{\ln T^{-1}(\ov{\lambda})}$.
  The first three coefficients   $\cH^{0}, \cH^{1}, \cH^{2}$ in \eqref{lnT:FT} are   the conserved mass, momentum  and energy (see \eqref{cEGL}) for the Gross-Pitaevskii equation \eqref{GP} (and hence also for the mKdV \eqref{mKdV}) respectively:
  $$
\cH^{0}=\cM, \quad \cH^{1}=\cP, \quad \cH^{2}=2\cE_{GL}=\int_{\R}\bigl((|q|^2-1)^2+ |\d_x q|^2 \bigr)\dx,
  $$
 and the fourth conserved Hamiltonian $\cH^{3}$ reads (see also Remark \ref{rmk:a4})
 $$
 \cH^{3}=\Im \int_{\R}\bigl( \d_x q\,\d_{xx}\bar q+3(|q|^2-1)q\d_x\bar q\bigr)\dx-\cP.
 $$
 The momentum $\cP$ is not defined on $X^s$ for any $s\ge 0$ and 
hence in our $L^2$-framework $q\in X^s$ we have to consider the \textit{renormalised} transmission coefficient $\Tc^{-1}(\lambda)$ which will be $T^{-1}(\lambda)$ modulo the mass and momentum (see Theorem \ref{thm:Tc} below for more details).

In constrast to the Nonlinear Schr\"odinger equation  we cannot scale solutions of the Gross-Pitaevskii equation   because of the boundary condition at infinity.
Hence there is no scaling invariance property for the Gross-Pitaevskii equation and   it does not suffice to consider small data.
In order to handle the  large energy case we  introduce the frequence-rescaled Sobolev norm  $H^s_\tau(\R)$, $\tau\geq 2$ which is equivalent to $H^s(\R)$-norm as follows
\footnote{%%%
$\tau^{-s}\|f\|_{H^s_\tau}$ is the semiclassical Sobolev norm $\bigl( \int_{\R}(1+(\hbar\xi)^2)^s|\hat f(\xi)|^2\dxi\bigr)^{1/2}$, $\hbar=\tau^{-1}$.
We always take $\tau\geq 2$   in the frequency-rescaled Sobolev norms, in order to avoid the possible zeros on $(-1,1)$ of the transmission coefficient in the formulation of the conserved energies (see \eqref{cElarge} below).}
 \begin{equation}\label{Sobolev}
 \|f\|_{H^s_\tau(\R)}^2
=\int_{\R} (\tau^2+\xi^2)^s |\hat{f}(\xi)|^2\dxi.
\end{equation} 
For any $q\in X^s$,   we introduce the following notation
\begin{equation}\label{bq}
\bq:=(|q|^2-1, \d_x q), \hbox{ with }\|\bq\|_{H^s_\tau(\R)}^2=\||q|^2-1\|_{H^s_\tau(\R)}^2+\|\d_x q\|_{H^s_\tau(\R)}^2,
\end{equation}
and we  define the associated \emph{energy} $E^{s}_{\tau}(q)$  as  
 \begin{equation}\label{Es}
\begin{split}  
 &E^{s}_{\tau}(q) 
 := \|\bq\|_{H^{s-1}_\tau(\R)},
\end{split}
\end{equation} 
which describes the $H^s$-regularity of $q$ and in particular when $\tau=2$ we denote simply
\begin{equation}\label{Es2}
E^{s}(q):=E^{s}_{2}(q).
%=\bigl( \bigl\| |q|^2-1\bigr\|_{H^{s-1}_2(\R)}^2+\bigl\| \d_x q\bigr\|_{H^{s-1}_2(\R)}^2 \bigr)^{\frac 12}.
\end{equation}
We also introduce the Banach space $l^2_\tau DU^2=DU^2+\tau^{\frac12}L^2\supset H^{s-1}$, $s>\frac12$ (which can be viewed as a replacement of $H^{-\frac12}$ and see Subsection \ref{subs:norm} below for more details) and the norm
\begin{equation*}\begin{split}
\|\bq\|_{l^2_\tau DU^2}= \Bigl( \bigl\| |q|^2-1\bigr\|_{l^2_\tau DU^2}^2+\bigl\| \d_x q\bigr\|_{l^2_\tau DU^2}^2\Bigr)^{\frac12} .
\end{split}\end{equation*}
It is straightforward to check that the $H^s_\tau(\R)$-norm and the $l^2_\tau DU^2(\R)$-norm have the following scaling invariance property:
  \begin{equation}\label{ftau}\begin{split}
 & \|f\|_{H^s_\tau(\R)}=\tau^{s+\frac12}\|f_\tau\|_{H^s(\R)},
\quad \|f\|_{l^2_\tau DU^2(\R)}=\|f_\tau\|_{l^2_1 DU^2(\R)},
 \quad f_\tau=\frac1\tau f(\frac{\cdot}{\tau}).
 \end{split} \end{equation}

We establish a family of conserved energy functionals $(\cE^{s}_{\tau})_{\tau\geq 2}$ as follows:
 \begin{thm}\label{thm}  
Let $s>\frac12$.
There exist  a constant $C\geq2$ (depending only on $s$) and  a family of 
%locally Lipschitz continuous
analytic energy functionals $(\cE^{s}_{\tau})_{\tau\geq 2}: X^s\mapsto [0,\infty)$, such that
\begin{itemize} 
\item 
 $\cE^{s}_{\tau}(q)$  is equivalent to $(E^{s}_{\tau}(q))^2$ in the following sense:
\begin{equation}\label{Equivalence}\begin{split}
&|\cE^{s}_{\tau}(q)-(E^{s}_{\tau}(q))^2|
\leq  \frac{C}{\tau} \|\bq\|_{l^2_\tau DU^2}  (E^{s}_{\tau}(q))^2,
\\
&\hbox{ if } q\in X^s\hbox{ such that }\frac{1}{\tau}\|\bq\|_{l^2_\tau DU^2}<\frac{1}{2C},
\end{split}\end{equation}   

\item 
$\cE^{s}_{\tau}(\cdot)$, $\tau\geq 2$ is conserved by    the one-dimensional Gross-Pitaevskii flow \eqref{GP}.
\end{itemize} 
Correspondingly,  for any initial data $q_0\in X^s$, there exists $\tau_0\geq C$  depending only on $E^{s}(q_0)$ such that   the unique solution $q\in \cC(\R; X^s)$ (given in Theorem \ref{thm:gwp})  of the Gross-Pitaevskii equation \eqref{GP}   satisfies  the following energy conservation law:
 \begin{equation}\label{ConsEnergy}\begin{split}
&E^{s}_{\tau_0}( q(t)) \leq  2E^{s}_{\tau_0}( q_0),
\quad  \frac{1}{\tau_0}\|\bq(t)\|_{l^2_{\tau_0} DU^2}<\frac1{2C},  \quad \forall t\in\R.
\end{split}\end{equation}
 \end{thm}

 \begin{rmk}\label{rmk:E}
\begin{enumerate}[1)]

\item One can find  the precise definition and the trace formula of the  energies $\cE^{s}_{\tau}(q)$ in Theorem \ref{thm:Es}. 

For example, 
%when the non negative superharmonic function on the upper half plane 
%$G(z)=\frac12\sum_{\pm}\Re(4z^2\ln\Tc^{-1}(\pm\sqrt{z^2+1}))$, $\Im z>0$  
%(here we choosing  the square root branch $\Im\sqrt{z^2+1}\geq 0$)  
%expands at $i\infty$ as 
%(recalling the expansion \eqref{lnT:FT} and noticing 
%$\ln\Tc^{-1}(\lambda)=\ln T^{-1}(\lambda)-i\cH^{-1}(2z)^{-1}-i\cH^{0}(2z(\lambda+z))^{-1}$ in Theorem \ref{thm:Tc})
%$$
%G(i\frac\tau2)=\sum_{l=0}^{[s-1]}(-1)^l\cH^{2l+1}\tau^{-2l-1}+o(\tau^{-2s+1}),
%\quad \tau\rightarrow\infty
%$$ 
%  has the finite trace measure  on the real line, 
we have the following trace formula for the conserved energies $\cE^{s}_{\tau}(q)$ when $s=n\geq1$ is an integer (recalling  $\cH^{l}$ in \eqref{lnT:FT})
 \begin{align*}
&\cE^{n}_{\tau}(q)=\sum_{l=0}^{n-1}\tau^{2(n-1-l)}\begin{pmatrix}n-1\\l \end{pmatrix}\cH^{2l+2},
\\
&\cH^{2l+2}=\frac{1}{\pi} \int_{\R}\xi^{2l+2}\frac12\sum_{\pm}\ln|\Tc^{-1}|(\pm\sqrt{\xi^2/4+1})d\xi- \frac{1}{2l+3}\sum_{m}\Im (2z_m)^{2l+3},
%\\
%&\cE_{2,\tau}(q)=4\cE_{1,\tau}(q)
%+\frac{1}{\pi}\int_{\R}(\xi^2+\tau^2)d\mu(\xi) - \sum_{m}\Im (2z_m)^3\bigl(\frac13(2z_m)^2+\tau^2),
\end{align*}
where $\Tc^{-1}$ is the \emph{renormalised} transmission coefficient defined for any $q\in X^s$ in Theorem \ref{thm:Tc} and $z_m=i\sqrt{1-\lambda_m^2}\in i(0,1]$ with 
$\{\lambda_m\}_m\subset (-1,1)$ being  the possible countably many zeros of the holomorphic function $\Tc^{-1}(\lambda)$ and hence the possible eigenvalues of the Lax operator $L$.

In particular if $q-1\in\cS(\R)$, then 
%we can write directly the trace formula, 
by changing of variables $\xi\rightarrow\lambda$ with $\lambda^2=\xi^2+1$  and noticing the symmetry in Subsection \ref{subsubs:T,R}: $\ln|T^{-1}|(\lambda+i0)=\ln|T^{-1}|(\lambda-i0)$ for $\lambda\in \Ic=(-\infty,1]\cup[1,\infty)$,
{\small\begin{align*} 
&\cH^{2l+2}=\frac{2^{2l+3}}{\pi} 
 \int_{\Ic}
|\lambda|\sqrt{\lambda^2-1}^{{2l+1}} \ln|T^{-1}|(\lambda) d\lambda
- \frac{1}{2l+3}\sum_{m}\Im (2z_m)^{2l+3}, 
\end{align*}}
for $l\geq 0$.
This can  be compared  with $2^{2l+3}c_{2l+3,\varrho}$ on   Pages 76 in \cite{FT}. 
   
\item 
 For any ball $B_r^s(q_0)=\{p_0\in X^s\,|\, d^s(q_0, p_0)<r\}$, $r>0$, in $X^s$ such that    (see Lemma \ref{metricspace} below) for any $p_0\in B^s_r(q_0)$,
$$
E^{s}(p_0)\leq E^{s}(q_0)+c(1+E^s(q_0))^{\frac12}d^s(q_0,p_0)+c(d^s(q_0, p_0))^2\leq C(E^s(q_0),r), 
$$
there exists $\tau_0$  (depending only on $E^{s}(q_0), r$) such that all the   solutions $p\in \cC(\R; X^s)$  of the Gross-Pitaevskii equation \eqref{GP} with the corresponding initial data $p_0\in B_r^s(q_0)$ satisfy  the  energy conservation law \eqref{ConsEnergy}.

\item The idea of the proof of Theorem \ref{thm} is similar as in \cite{KT}, however due to the nonzero background, the proof requires substantial new ideas and concepts and the characterised quantities in the energy space are the nonlinear function of $q$: $|q|^2-1$ and its derivative $q'$ rather than the solution $q$ itself.

\item In the proof showing the asymptotic approximation of the Gross-Pitaevskii equation by the Korteweg-de Vries equations in long-wave regime, \cite{BGSS} made use of the uniform bounds of $E^{k}(q), k=1,2,3,4$ which were derived from  a linear (and not obvious at all) combination of the first nine energy conservation laws $\cH^{0},\cdots, \cH^{8}$.
Theorem \ref{thm} here is a first existence result of infinitely many conserved quantities which control $E^k(q), k=1,2,\cdots$   of the solutions of the Gross-Pitaevskii equation (and mKdV with the same boundary condition at $\infty$).
\end{enumerate}
\end{rmk}

 We also have   the following results for the modified KdV equation  \eqref{mKdV}. We recall that we define wellposedness in terms of the existence of a representative. 
 \begin{thm}\label{thm:mKdV}
 The complex modified KdV equation  \eqref{mKdV} is globally-in-time well-posed in the metric space $(X^s, d^s)$, $s > \frac34$ in the following  sense (as in Theorem \ref{thm:gwp}):  
For any initial data $\psi_0\in X^s$,   there exists a unique   solution $\psi\in \cC(\R; X^s)$ (by which we mean  that the flow map on $1+\mathcal{S}$ extends continuously to $X^s$) and the flow map $X^s \ni \psi_0\mapsto   \psi\in \cC(\R; X^s)$ is 
%Lipschitz 
continuous.    
The energy functionals $(\cE^s_\tau(\cdot))_{s>\frac12, \tau\geq2}$ established in Theorem \ref{thm}   are also conserved by the modified KdV flow \eqref{mKdV}. 

For real data the flow map extends to a continuous map from $X^s$ to $\cC(\R; X^s)$ for $s \ge 0$.   

 \end{thm}

 The following sections are organised as follows:
 \begin{itemize}
 \item In Section \ref{sec:lwp} we state and prove  Theorem \ref{thm:lwp} (\textit{resp.} Theorem \ref{thm:lwpmKdV}), which states the local-in-time well-posedness of the Gross-Pitaevskii equation \eqref{GP} (\textit{resp.} the modified KdV equation \eqref{mKdV}) in the energy space $(X^s, d^s)$, $s\ge 0 $ (\textit{resp.} $s>\frac34$ in the complex case and $s\geq 0$ in the real case): For any initial data $q_0\in X^s$, there exists a unique solution $ q\in \cC([-t_0, t_0]; X^s)$ of \eqref{GP} (\textit{resp.} \eqref{mKdV}), such that the flow map is %Lipschitz 
 continuous and the existence time $t_0$ depends only on $E^{s}(q_0)$.
 
 \item In Section \ref{sec:T} we state  Theorem \ref{thm:Tc}, where we introduce the renormalised transmission coefficient $\Tc^{-1}(\lambda)$ and show the conservation of $\Tc^{-1}(\lambda; q(t))$ by the Gross-Pitaevskii flow on the existence time interval $I$ for any solution $q\in \cC(I;X^s)$, $s>\frac12$.
 
 \item Section \ref{sec:Tc} is devoted to the proof of Theorem \ref{thm:Tc}.
 
 \item In Section \ref{sec:est} we state and prove Theorem \ref{thm:Es}, where we establish a family of energy functionals $(\cE^{s}_{\tau}: X^s\mapsto[0,\infty))_{s>\frac12,\tau\geq 2}$ in terms of $\ln\Tc^{-1}$, which satisfies the equivalence relation \eqref{Equivalence}.
 
 \item Section \ref{sec:metric} is devoted to the proof of Theorem \ref{thm:metric}.
  \item In the Appendix we calculate the quadratic term in the expansion of $\ln\Tc^{-1}(\lambda)$ on the imaginary axis.
 \end{itemize}
 
 At the end of this introduction, we prove our main Theorems \ref{thm:gwp} and   \ref{thm} concerning the Gross-Pitaevskii equation \eqref{GP} by use of the results from Theorems \ref{thm:lwp}, \ref{thm:Tc} and \ref{thm:Es}.
 Since the modified KdV equation \eqref{mKdV} shares the same Lax operator as the Gross-Pitaevskii equation, Theorem \ref{thm:mKdV} follows from Theorems \ref{thm:lwpmKdV}, \ref{thm:Tc} and \ref{thm:Es} exactly in the same way.
 
  We first  state the relations between  $E^{s}_{\tau}=\|\bq\|_{H^{s-1}_\tau}$, $E^{s}=E^{s}_{2}$ and $\frac1\tau\|\bq\|_{l^2_\tau DU^2}$.
 \begin{lem}\label{lem:ctau} 
 There exists a family of constants $(C_s)_{s>\frac12}$ with $C_s\geq 1$ 
and $C_s=C_1$, $s\geq 1$ 
such that   whenever $\tau\geq 2$, for all $s>\frac12$,
{\small\begin{equation}\label{ctau,Etau}\begin{split}
&\frac1\tau\|\bq\|_{l^2_\tau DU^2}\leq C_s\tau^{-\frac12-s}E^{s}_{\tau},
\quad E^{s}_{\tau}\leq C_s\tau^{\max\{0,s-1\}}E^{s}, 
\\
&\hbox{ and hence }\frac1\tau\|\bq\|_{l^2_\tau DU^2}\leq C_s\tau^{-\frac12-\min\{s,1\}}E^{s}.
\end{split}\end{equation}}  
 \end{lem}
 \begin{proof}We derive from the scaling property \eqref{ftau} and the embedding $H^{s-1}(\R)\hookrightarrow l^2_1 DU^2(\R)$, $s>\frac12$ that
{\small\begin{align*}
&\frac1\tau\|\bq\|_{l^2_\tau DU^2}=\frac{1}{\tau}\Bigl(\bigl\|(|q|^2-1)_\tau\bigr\|_{l^2_1 DU^2}^2
+\bigl\|(\d_x q)_\tau\bigr\|_{l^2_1 DU^2}^2\Bigr)^{\frac12},
\hbox{ with }f_\tau=\frac1\tau f(\frac\cdot\tau)
\\
&\qquad\qquad\leq C_s \frac{1}{\tau}\Bigl(\bigl\|(|q|^2-1)_\tau \bigr\|_{H^{s-1}}^2
+\bigl\|(\d_x q)_\tau\bigr\|_{H^{s-1}}^2\Bigr)^{\frac12}
=C_s\tau^{-\frac12-s}E^{s}_{\tau}(q).
\end{align*}}
By virtue of the fact that $\tau\mapsto E^{s}_{\tau}$ is decreasing if $s\in (\frac12,1]$ and $E^{s}_{\tau}\leq (\tau/2)^{ s-1}E^{s}$ if $s\geq 1$, we have $E^{s}_{\tau}\leq C_s\tau^{\max\{0,s-1\}}E^{s}$.
 \end{proof}

We are going to prove the global-in-time wellposedness result (Theorem \ref{thm:gwp}) and the energy conservation law \eqref{ConsEnergy} (Theorem \ref{thm}) simultaneously,  for the initial data $q_0\in X^s$, $s>\frac12$ by use of  the following facts from Theorems \ref{thm:lwp}, \ref{thm:Tc} and \ref{thm:Es}:
 \begin{itemize}
 \item There exists a unique solution $q\in \cC([-t_0, t_0]; X^s)$ of the Gross-Pitaevskii equation with $t_0>0$ depending only on $E^{s}(q_0)$ (by Theorem \ref{thm:lwp});
 \item The renormalised transmission coefficient $\Tc^{-1}(\lambda;q(t))$ is conserved by the Gross-Pitaevskii flow on $[-t_0, t_0]$ (by Theorem \ref{thm:Tc});
 \item The energy functional $\cE^{s}_{\tau_0}$, which is constructed in terms of $\ln\Tc^{-1}$,  is also conserved by the Gross-Pitaevskii flow, and furthermore,  the equivalence relation \eqref{Equivalence} holds (by Theorem \ref{thm:Es}).  
 \end{itemize}

For the initial data $q_0\in X^s$, we take  $\tau_0$ depending only on $E^{s}(q_0)$ such that  (with the   constant $C$ given in  \eqref{Equivalence})
{\small\begin{equation}\label{GeneralTau0}
C_s^2\tau_0^{-\frac12-\min\{s,1\}}
\bigl(2E^{s}( q_0)\bigr)<\frac{1}{2C}
\hbox{ and hence by Lemma \ref{lem:ctau}, }\frac1{\tau_0}\|\bq_0\|_{l^2_{\tau_0} DU^2}<\frac{1}{2C}.
\end{equation}}
The equivalence relation \eqref{Equivalence} implies initially  
 $E^{s}_{\tau_0}( q_0)\leq \sqrt{2\cE^{s}_{\tau_0}( q_0)}\leq 2E^{s}_{\tau_0}( q_0)$.

 By  the equivalence relation \eqref{Equivalence} and the conservation of the energy $\cE^{s}_{\tau_0}(q(t))$, the solution $ q\in \cC([-t_0,t_0];X^s)$  satisfies the conservation law \eqref{ConsEnergy} on the existence time interval $t\in [-t_0, t_0]$ as follows (noticing also \eqref{ctau,Etau}, \eqref{GeneralTau0}):
{\small \begin{align*}
& E^{s}_{\tau_0}(  q(t))\leq \sqrt{2\cE^{s}_{\tau_0}( q(t))}
 =\sqrt{2\cE^{s}_{\tau_0}(  q_0)}\leq 2E^{s}_{\tau_0}( q_0),
 \\
 &\frac1{\tau_0}\|\bq(t)\|_{l^2_{\tau_0} DU^2}\leq C_s\tau_0^{-\frac12-s}E^{s}_{\tau_0}( q(t))
 \leq C_s\tau_0^{-\frac12-s}(2E^{s}_{\tau_0}(q_0))
 \\
&\qquad\qquad\qquad\qquad\qquad
 \leq C_s^2\tau_0^{-\frac12-\min\{s,1\}}(2E^{s}(q_0))<\frac{1}{2C}.
 \end{align*}}  
 By a continuity argument, the solution $q$ exists globally in time  and satisfies the energy conservation law \eqref{ConsEnergy}: $E^{s}_{\tau_0}(q(t))\leq 2E^{s}_{\tau_0}(q_0)$, $\forall t\in\R$.
 Indeed, if not and suppose $I\neq\R$ is the maximal existence time interval for the solution $q$, then by the above argument we have $E^{s}_{\tau_0}(q(t))\leq 2E^{s}_{\tau_0}(q_0)$ for all $t\in I$.
 By Theorem \ref{thm:lwp} we can extend the solution to a strictly larger time interval than $I$, which is a contradiction of the maximality of $I$.

\setcounter{equation}{0}%%%%%%%%
\section{Local well-posedness}\label{sec:lwp}%%%%%%%%
We prove the locally-in-time  well-posedness  for the Gross-Pitaevskii equation (see Theorem \ref{thm:lwp}) and for the modified Korteweg-de Vries equation (see Theorem \ref{thm:lwpmKdV}) respectively  in this section.
\begin{thm}\label{thm:lwp}The Gross-Pitaevskii equation \eqref{GP} is locally-in-time well-posed in the metric space $(X^s, d^s)$, $s\geq0$ in the following sense (as in Theorem \ref{thm:gwp}):
\begin{itemize}
\item 
For any initial data $q_0\in X^s$, there exists  $t_0>0$ depending only on $E^{s}(q_0)= \|\bq_0\|_{H^{s-1}_2}$, $\bq_0=\bigl(   |q_0|^2-1, q_0'  \bigr)$, and a unique   solution $q\in \cC((-t_0, t_0); X^s)$ (defined in Definition \ref{def}) of the Gross-Pitaevskii equation;  

\item 
For the neighbourhood $B_r^s(q_0)=\{p_0\in X^s\,|\, d^s(q_0, p_0)<r\}$, $r>0$, of the initial data $q_0\in X^s$, there exists $t_1>0$ depending only on $E^{s}(q_0), r$ such that the flow map $B_r^s(q_0)\ni p_0\mapsto p\in \cC((-t_1, t_1); X^s)$  is  
%{\color{red} Lipschitz ? }  
continuous. 
\end{itemize} 
\end{thm}

We begin with a lemma. Let $\varepsilon\in(0,1]$, with $\varepsilon=1$ being a legitimate and most natural choice in what follows.  
   We regularize a function $f$ by taking its convolution with the mollifier $\rho_\varepsilon=\varepsilon^{-1}\rho(\varepsilon^{-1}x)$, $0\leq\rho\in C^\infty_0(\R)$, $\int_{\R}\rho=1$ as $f_{\varepsilon}:=f\ast \rho_\varepsilon$. 
\begin{lem} \label{lem:pq}
  Let $q\in X^s$, $s\geq 0$  and let $\tilde q$ be a representative.  Then 
   \begin{enumerate}
   \item $\tilde q_{\varepsilon} \in X^\sigma $ for all $\sigma \ge 0 $ and $\tilde q_{\varepsilon} \to \tilde q$ in $X^s$ as $\varepsilon \to 0 $.
The map $X^s\ni \tilde q \to \tilde q_{\varepsilon} \in X^\sigma $ is   Lipschitz continuous.  
   \item Let $ q\in X^s$, and $\phi$ a Schwartz function such that for one (and hence for all) representative $\tilde q$ there holds $\int_{\R} \tilde q \phi \dx \ne 0$. Then there is a neighborhood of $q$ such that this remains true. We fix the representatives with $\int_{\R} \tilde q\phi \dx \in (0,\infty)$, then the map   
     \[  X^s \ni q \to \tilde q- \tilde q_{\varepsilon} \in H^s  \]
     is continuous.
   \item  The map
     \[ H^s\ni b \to \tilde q+b \in X^s \]
     is   Lipschitz continuous. 
   \end{enumerate}
\end{lem}

\begin{proof} 
We derive from
\[
\begin{split} 
(f-f_{\varepsilon})(x)& \, = \int_{\R} (f(x)- f(x-y) ) \rho_\varepsilon(y) \dy  
  \\&  =    \int_{\R}\int^y_0 f'(x-a)\,\da\,\rho_\varepsilon(y)\dy
   = \int_{\R} f'(x-a)   \int_{A(a)}  \rho_{\varepsilon}(y) \dy                     \da
    \end{split} 
\]
where
\[ A(a) = \left\{ \begin{array}{rl}  (a,\infty) & \text{ if } a >0 \\
  (-\infty,a) & \text{ if } a <0 \end{array} \right. \] 
that
\begin{equation} \label{difference} 
 \Vert f -f_{\varepsilon} \Vert_{L^2} \lesssim    \Vert f' \Vert_{H^{-1}} \end{equation} 
with an absolute implicit constant. 

We choose $\eta \in C^\infty_0$. Then 
    \[ \int  \eta  |f_{\varepsilon}|^2 dx = \int \eta   (|f|^2 -1) dx+ \int \eta dx
    - \int \eta |f-f_\varepsilon|^2 dx -2 \Re\int  \eta \bar f_\varepsilon  (f-f_\varepsilon) dx, \]
    and hence
    \[ \int \eta |f_\varepsilon|^2 dx \le 2 \Vert |f|^2-1 \Vert_{H^{-1}} \Vert \eta \Vert_{H^1} 
    +2\|\eta\|_{L^1} + 3 \Vert f-f_\varepsilon \Vert_{L^2}^2 \Vert \eta \Vert_{sup}. \]
Choosing $\eta $ appropriately we see that there exists $C>0$ so that for all $x \in \R$
    \begin{equation}\label{uniforml2}   \Vert f_\varepsilon\Vert_{L^2([x,x+1])} \le C(1+ \Vert |f|^2-1 \Vert_{H^{-1}}^{\frac12}  + \Vert f' \Vert_{H^{-1}}). \end{equation} 
    We may choose $\tilde \rho = \rho * \rho $ and obtain with a small abuse of notation
    \begin{equation}\label{pointwisebound} \Vert f_\varepsilon \Vert_{L^\infty} \le c \varepsilon^{-1/2}
      C(1+ \Vert |f|^2-1 \Vert_{H^{-1}}^{\frac12}  + \Vert f' \Vert_{H^{-1}}). \end{equation}

    Using the embedding $L^1\hookrightarrow H^{-1} $ we estimate using a partition of unity 
    \[
\begin{split} 
    \Vert |f_{\varepsilon}|^2-1 \Vert_{H^{-1}}
   & \,  \le  \Vert |f|^2-1 \Vert_{H^{-1}} + c \sum_{k\in \Z}
    \Vert |f_{\varepsilon}|^2 - |f|^2 \Vert_{L^1((k-1,k+1))}
\\ & \le  \Vert |f|^2-1 \Vert_{H^{-1}} + c\,  \sup_{k\in \Z} \left( \Vert f_\varepsilon \Vert_{L^2((k-1,k+1))} + \Vert f\Vert_{L^2(({k-1,k+1}))} \right) 
    \Vert f_{\varepsilon} - f\Vert_{L^2}
%\\
% &\hbox{\color{blue}I think that in the above it should be $\|\cdot\|_{\ell^2_k}$-norm instead of $\|\cdot\|_{\ell^\infty_k}$-norm (i.e. $\sup_{k\in\Z}$ here).}
    \\ & \le  c \Big( \Vert |f|^2-1 \Vert_{H^{-1}} +  (1+\Vert f' \Vert_{H^{-1}})\Vert f'\Vert_{H^{-1}} \Big)
    \end{split} 
    \]
 Notice that for any $\sigma\in\R$,
\begin{equation}\label{q0epsilon'}
\|f_{\varepsilon}'\|_{H^{\sigma}}\leq C(\varepsilon,\sigma)\|f'\|_{H^{-1}}, 
\end{equation} 
and hence 
\begin{equation}\label{q0epsilon,-1}\begin{split}
\||f_{\varepsilon}|^2-1\|_{H^{\sigma+1}} + \Vert f_{\varepsilon}' \Vert_{H^{\sigma-1}} 
&\lesssim \||f_{\varepsilon}|^2-1\|_{H^{-1}}+\|\bar f_{\varepsilon}f_{\varepsilon}'\|_{H^{\sigma}}+\|f'\|_{H^{-1}}
\\
&\leq C(\varepsilon, \sigma, E^0(f)) E^0(f).
\end{split}\end{equation}

Therefore $\tilde q_\varepsilon\in X^\sigma$ for all $\sigma\geq 0$ and the convergence $\tilde q_\varepsilon\rightarrow \tilde q$ in $X^s$, i.e. $d^s(\tilde q_\varepsilon, \tilde q)\rightarrow 0$, as $\varepsilon \to 0 $ follows immediately from these considerations. Indeed,
\[
\begin{split} 
(d^s(\tilde q_\varepsilon, \tilde q))^2
=  \int_{\R} \inf_{|\lambda|=1} \Vert \sech(x-y) ( \tilde q - \lambda\tilde  q_{\varepsilon})  \Vert^2_{H^s} \dy \, &   \le  \int_{\R}  \Vert \sech(x-y) (\tilde q- \tilde q_\varepsilon) \Vert^2_{H^s} \dy
  \\ & \le c_1 \int_{\R} \Vert \sech(x-y)\tilde  q' \Vert_{H^{s-1}}^2 dy
  \\ & \le c_2  \Vert \tilde q' \Vert_{H^{s-1}}^2 <\infty.
\end{split}
\]
  Given $\delta>0$ we can restrict the above $y$ integration to a compact interval with an error at most $\delta$. The convergence on the compact $y$ interval is immediate.

We turn to the proof of 
Lipschitz  continuity of the map
\[ \ast\rho_\varepsilon: X^s \ni \tilde q \mapsto \tilde q_\varepsilon\in  X^\sigma.  \] 
 Indeed,  recalling the metric distance function $d^s$ in \eqref{ds}, we first calculate $d^s(\tilde q_\varepsilon, \tilde p_\varepsilon)$.
  We have  the following commutator formulae: 
 \begin{align*}
 &[\sech(\cdot-y), \ast\rho_\varepsilon] (\lambda \tilde q-\tilde p) 
 = \int_{\R}\int^m_0\sech'(\cdot-y-a)(\lambda \tilde q-\tilde p)(\cdot-m)\da\, \rho_\varepsilon(m) \dm,
\end{align*}
where $[\sech(\cdot-y), \ast\rho_\varepsilon]f=\sech(\cdot-y)\bigl(f\ast\rho_\varepsilon\bigr)-\bigl(\sech(\cdot-y)f\bigr)\ast\rho_\varepsilon$. 
%Thus  
%{\small \begin{align*}
% &[\sech(\cdot-y), \ast\rho_\varepsilon] (\lambda u-v) 
% = \int_{\R}\int^m_0\sech'(\cdot-y-a)(\lambda u-v)(\cdot-m)da\, \rho_\varepsilon(m) dm
%\\
%&   |u_\varepsilon|^2-1 - (|v_\varepsilon|^2-1)
%= 2 \Re (u_\varepsilon - v_\varepsilon )\overline{(u_\varepsilon+v_\varepsilon)} 
%\end{align*}} 
%\[ \Vert \sech(x) (|u_\varepsilon|^2-1 - (|v_\varepsilon|^2-1))\Vert_{H^{s-1}}
%\le c ( \Vert  u_\varepsilon \Vert_{H^{\max\{ 1,s-1\}}} +
%\Vert v_\varepsilon \Vert_{H^{\max\{ 1,s-1\}}}) \Vert \sech(x-y) ( u_\varepsilon - \lambda v_\varepsilon) \Vert_{H^{s-1}} . \]
%We minimze with respect to $\lambda$ and integrate with respect to $y$. 
Hence we derive the  Lipschitz continuity of the map $\ast\rho_\epsilon: X^s\mapsto X^s$ as follows: 
{\small\begin{align*}
&\bigl(d^s(\tilde q\ast\rho_\epsilon, \tilde p\ast\rho_\epsilon)\bigr)^2
=\int_{\R}\inf_{|\lambda|=1} \| \sech(\cdot-y)\bigl((\lambda \tilde q-\tilde p)\ast\rho_\epsilon\bigr)\|_{H^s}^2\dy
%+\bigl\||u\ast\rho_\epsilon|^2-|v\ast\rho_\epsilon|^2\bigr\|_{H^s}^2
\\
&\qquad\lesssim 
\int_{\R}\inf_{|\lambda|=1} 
\Bigl( \| \sech(\cdot-y) (\lambda \tilde q-\tilde p) \|_{H^s}^2
 +\| \sech'(\cdot-y) (\lambda \tilde q-\tilde p) \|_{H^s}^2\Bigr)\dy
%+ \||u|^2-|v|^2\|_{H^s}^2
%\\
%&\qquad\qquad
%+C(\|(u', v')\|_{H^{s-1}})\|u'-v'\|_{H^{s-1}}^2
\leq  
%C(\|(u', v')\|_{H^{s-1}})
C(d^s(\tilde q,\tilde p))^2.
\end{align*}}

We turn to the proof that
\[ X^s \ni   q \to \tilde q-\tilde q_{\varepsilon} \in H^s  \]
is continuous. First, it is not hard to see that, if $\phi$ is a Schwartz function and
\[ \int \tilde q \phi dx \ne 0, \]
then there is a neighborhood   so that this remains true. 
Now let $\varepsilon_0 >0$. There exists a smaller neighborhood and a compact interval $I$ so that
\[ \|\tilde q-\tilde q_\varepsilon\|_{H^s(\R\backslash I)}
\lesssim \Vert \tilde q' \Vert_{H^{s-1}(\R\backslash I ) } < \varepsilon_0/2 \]
for all representative of functions in this smaller neighborhood. Clearly there exists $\delta>0$ so that
\[   \Vert  p-\tilde q \Vert_{H^s(I)} < \varepsilon_0/2,
\quad \forall   p \in B_\delta^s(\tilde q).\] 

Now let $\tilde q$ be a representative and consider
\[ H^s \ni b \to \widetilde{q+b} \in X^s. \]
It suffices to prove that there exists $C$ so that
\[  d^s(q+b,q) \le C \Vert b \Vert_{H^s}, \]
which follows immediately from the definition of $d^s$.  
\end{proof}

\begin{proof}[Proof of Theorem \ref{thm:lwp} ]

If $q$  solves the Gross-Pitaevskii equation \eqref{GP} with the initial data $q_0\in X^0$ and $\tilde q$, $\tilde q_0$ are the corresponding representatives of $q(t), q_0$, then $b=\tilde q-\tilde q_{0,\varepsilon}$ satisfies the following nonlinear Schr\"odinger-type equation  
 \begin{equation}\label{eq:p}
 \begin{split}
   i\d_t b+\d_{xx}b
 = g(b),  \quad b|_{t=0}=b_0=\tilde q_0-\tilde q_{0,\varepsilon}\in H^s,
 \end{split}
 \end{equation}
 where
 \begin{align*}
 g(b)=&2|b|^2b+4\tilde q_{0,\varepsilon}|b|^2+2\overline {\tilde q_{0,\varepsilon}}b^2
 +(4|\tilde q_{0,\varepsilon}|^2-2)b+2(\tilde q_{0,\varepsilon})^2\bar b
  +2\tilde q_{0,\varepsilon}(|\tilde q_{0,\varepsilon}|^2-1)-\tilde q_{0,\varepsilon}''.
 \end{align*}
Vice versa: If $b$ satisfies this equation then $\tilde q$ satisfies \eqref{GP}.  

 We claim that there exist a positive time  $t^0$ and a positive constant $C^0$ depending only on $E^{0}(q_0), \varepsilon$ and a unique solution  of \eqref{eq:p}: $b\in \cC([-t^0, t^0]; L^2)$ such that
 $$ 
 \|b\|_{t^0}:=\|b(t)\|_{L^\infty([-t^0, t^0];L^2_x)}
 +\|b\|_{L^8([-t^0, t^0]; L^4(\R_x))}
 +\|b\|_{L^6([-t^0,t^0]\times\R)}\leq C^0 E^{0}(q_0).$$
 Indeed, recall the Strichartz estimates for the  Schr\"odinger semigroup $S(t)=e^{it\d_{xx}}$:
 \begin{align*}
 \|S(t)b_0\|_{T}\lesssim \|b_0\|_{L^2},
 \quad \Bigl\| \int^t_0 S(t-t') g(t') dt'\Bigr\|_{T}\lesssim \|g\|_{L^1([-T, T];L^2_x)}.
 \end{align*}
 Since we derive from the estimates \eqref{q0epsilon,-1} that
 \begin{align*}
\|g(b)\|_{L^1([-T, T];L^2_x)}
&\lesssim T^{\frac12}\|b\|_T^3+T^{\frac34}\|\tilde q_{0,\varepsilon}\|_{L^\infty_{x}}\|b\|_T^2
\\
&\quad+T(\|\tilde q_{0,\varepsilon}\|_{L^\infty_{x}}^2+1)\bigl(\|b\|_T+\||\tilde q_{0,\varepsilon}|^2-1\|_{L^2_x}+\|\tilde q_{0,\varepsilon}''\|_{L^2_x}\bigr)
\\
&\leq C(\varepsilon, E^0(q_0))\Bigl( T^{\frac12}\|b\|_T^3+T^{\frac34}\|b\|_T^2 
+T(\|b\|_T+1)\Bigr),
 \end{align*}
there exist a small enough positive time $t^0$ and a positive constant $C^0$ (depending only on $\varepsilon, E^0(q_0)$) such that the map
 {\small
 \begin{align*}
 b\mapsto S(t)b_0+\int^t_0S(t-t')g(b(t'))dt', 
 \end{align*}}
 is a contraction map in the complete metric space
  $\{b\in \cC([-t^0,t^0]; L^2)\,|\, \|b\|_{t^0}\leq C^0 E^s(q_0)\}$, and hence its  fixed point is the unique solution of \eqref{eq:p}.
  It is easy to see that the flow map $b_0\mapsto b(t)$ is locally Lipschitz in $L^2$.
  Correspondingly there exists   a  solution (in Definition \ref{def}) with $\tilde q=\tilde q_{0,\varepsilon}+b\in \cC((-t^0,t^0);X^s+L^2)=\cC((-t^0,t^0);X^0)$ of the Gross-Pitaevskii equation \eqref{GP} with the initial data $q_0\in X^0$,  such that 
 $\|\tilde q(t)-\tilde q_0\|_{L^2}\leq \|\tilde q_{0,\varepsilon}-\tilde q_0\|_{L^2}+\|b(t)\|_{L^2}\leq (C+C^0) E^{0}(q_0)$, $\forall t\in  [-t^0, t^0]$
 and $\tilde q(t)-\tilde q_{0,\varepsilon}=b(t)\in L^4([-t^0,t^0]\times \R)$. 
 
  Consider   two solutions (in Definition \ref{def})  $q_1, q_2\in \cC((-t^0, t^0); X^0)$ of the Gross-Pitaevskii equation \eqref{GP} with the initial data $ q_0\in X^0$. We may choose for both solutions the same representative $\tilde q_0$. 
Then on any compact time interval $I\ni 0$ in $(-t^0,t^0)$ their difference $b_{12}=\tilde q_1-\tilde q_2\in L^\infty(I; L^2)\cap L^4(I\times\R)$  with zero initial data satisfies in the distribution sense the following equation similar as \eqref{eq:p}:
 \begin{align*}
  i\d_t b_{12}+\d_{xx}b_{12}
  &=2|b_{12}|^2b_{12}+4\tilde q_{2}|b_{12}|^2+2\overline{\tilde q_2}(b_{12})^2
 +(4|\tilde q_2|^2-2)b_{12}+2(\tilde q_{2})^2\bar b_{12},
 \end{align*}
 which has a unique solution $0$ in $L^\infty([a, b]; L^2)$, by virtue of the energy inequality  
\footnote{%%%%%%%%%
We can follow the standard regularizing procedure to derive the energy inequality: We regularize the $b_{12}$-equation by convolution with $\rho_\delta$, take the $L^2(\R)$-inner product between the regularized equation and  $b_{12, \delta}=b_{12}\ast\rho_\delta$, take the imaginary part and finally we use Gronwall's inequality and let $\delta\rightarrow0$.
}.%%%%%%%%$\|b_{12}\|_{L^2}\leq  C(a,b)\|b_{12}|_{t=0}\|_{L^2}$. 
Hence $\tilde q_1=\tilde q_2=\tilde q_{0, \varepsilon}+b$ with $b$ satisfying \eqref{eq:p}.

If $s\in (0,1)$, we decompose $g=g(b)$ into 
$$
g=g_2(b)+g_1(b)+g_0,
\quad g_1(b)=(4|\tilde q_{0,\varepsilon}|^2-2)b+2(\tilde q_{0,\varepsilon})^2\bar b,
\quad g_0=2\tilde q_{0,\varepsilon}(|\tilde q_{0,\varepsilon}|^2-1)-\tilde q_{0,\varepsilon}''.
$$
Recall the definition of the Besov-norm for $s\in (0,1)$:
\begin{align*}
\|f\|_{\dot B^s_{\alpha,r}}=\Bigl\| \frac{\|f(x-y)-f(x)\|_{L^\alpha_x}}{|y|^s}\Bigr\|_{L^r(\R;\frac{dy}{|y|})},
\quad \|f\|_{B^s_{\alpha,r}}=\|f\|_{L^\alpha}+\|f\|_{\dot B^s_{\alpha,r}},
\end{align*}
and in particular   $\dot B^{s}_{2,2}=\dot H^s$. We apply the previous construction to the finite differences, and integrate the estimates for fixed $y$ according to the Besov norm above. It follows from these construction that the time of existence is the same for all $s\in [0,1)$.

The case $s \ge 1 $ follows similarly.

Therefore the Gross-Pitaevskii flow map $X^s\ni \tilde q_0 \mapsto \tilde q_0\ast\rho_\varepsilon+b\in X^s$ is  continuous on the existence time interval 
$[-t^0, t^0]$.
Indeed,  by the Lipschitz continuity of the flow \eqref{eq:p}, for any two solutions $\tilde q_1(t)=\tilde q_{1,\varepsilon}+b_1(t)$ and $\tilde q_2(t)=\tilde q_{2,\varepsilon}+b_2(t)$,
  \begin{align*}
 d^s(\tilde q_1(t), \tilde q_2(t))
 &\leq  d^s(\tilde q_{1,\varepsilon}, \tilde  q_{2,\varepsilon})+C\|b_1(0)-b_2(0)\|_{H^s},
% \leq Cd^s(\tilde q_1, \tilde q_2)+C\|(\tilde q_1-\tilde q_{1,\varepsilon})-(\tilde q_2-\tilde q_{2,\varepsilon})\|,
 \end{align*}
 and the continuity of the GP flow follows from Lemma \ref{lem:pq}.

\end{proof}

 We complete this section by a discussion of the flow defined by modified Korteweg-de Vries equation \eqref{mKdV}: $q_t+q_{xxx}-6|q|^2q_x=0$.

\begin{thm}\label{thm:lwpmKdV}The complex modified KdV equation  \eqref{mKdV} is locally-in-time well-posed in the metric space $(X^s, d^s)$, $s > \frac34$ in the following sense (as in Theorem \ref{thm:gwp}):
\begin{itemize}
\item 
For any initial data $q_0\in X^s$, there exists  $t_0>0$ depending only on $E^{s}(q_0)= \|\bq_0\|_{H^{s-1}_2}$, $\bq_0=\bigl(   |q_0|^2-1, q_0'  \bigr)$, and a unique   solution $q\in \cC((-t_0, t_0); X^s)$, by which we mean  that the flow map on $1+\mathcal{S}$ extends continuously to $X^s$. 

\item 
  For the neighbourhood $B_r^s(q_0)=\{p_0\in X^s\,|\, d^s(q_0, p_0)<r\}$, $r>0$, of the initial data $q_0\in X^s$, there exists $t_1>0$ depending only on $E^{s}(q_0), r$ such that the flow map $B_r^s(q_0)\ni p_0\mapsto p\in \cC((-t_1, t_1); X^s)$  is  continuous. 
\end{itemize}
For real data the flow map extends to a 
%Lipschitz  
continuous map from $X^s_{\R}$  to   $\cC((-t_1, t_1); X^s)$ for $s \ge 0$. Here $X^s_{\R}$ denotes the subspace of real valued functions.    
\end{thm} 

\begin{proof}
  We proceed in the same fashion as for the Gross-Pitaevskii equation. Now $b=\tilde q-\tilde q_{0,\varepsilon}$ satisfies
  \begin{equation}\label{mKdVm}   b_t + b_{xxx} = g(b) \end{equation} 
  where
  \[
\begin{split} 
  g(b) = & \, 
  6 |b|^2 b_x + 12 \Re (b\overline{\tilde q_{0,\varepsilon}}) b_x
  + 6 b_x + 
  6 (|\tilde q_{0,\varepsilon}|^2- 1) b_x +6 |b|^2  \tilde q'_{0,\varepsilon} 
\\ & + 12 \Re ( b\overline{\tilde q_{0,\varepsilon}}_x) \tilde q_{0,\varepsilon} 
  + 6|\tilde q_{0,\varepsilon}|^2 \partial_x \tilde q_{0,\varepsilon} 
  - \tilde q^{(3)}_{0,\varepsilon}.
  \end{split} 
  \]
Changing coordinates $(t,x) \to (t,y)$ with $ y = x+6t$ we remove the term $6b_x$.  Notice that
\[ 6|\tilde q_{0,\varepsilon}|^2 \partial_x \tilde q_{0,\varepsilon} - \tilde q^{(3)}_{0,\varepsilon} \in L^2  \]
and $\tilde  q_{0,\varepsilon}$ is together with all its derivatives uniformly bounded. The most critical terms are $6 |b|^2 b_x$, $12\Re (b\overline{\tilde q_{0,\varepsilon}}) b_x$
and $6 (|\tilde q_{0,\varepsilon}|^2-1)  b_x$.

We claim that \eqref{mKdVm} is locally wellposed in $H^s$, $s >\frac34$, and that the solution is  continuous with values in $H^s$. 
Indeed, this follows from a contraction argument as for the Korteweg-de Vries equation
\begin{equation}\label{KdV}
u_t+u_{xxx}-6uu_x=0
\end{equation}
 by Kenig, Ponce and Vega
\cite{MR1211741, MR1086966}. More precisely their arguments allow to deal with
$|b|^2b_x$ and $\Re (b\overline{\tilde q_{0,\varepsilon}}) b_x$. Since $|\tilde q_{0,\varepsilon}|^2-1 \in H^N$ for all $N$ the term
\[ (|\tilde q_{0,\varepsilon}|^2-1 ) b_x \]
is covered by the same estimates as the previous terms.

For real initial data we use a different argument.
 Let $s \ge 0$. 
 %We denote the subspace of real function in $X^s$ by $X^s_{\R}$. 
Then again $|q_{0,\varepsilon}|^2 -1 \in H^N$ for all $N$. Since it is also real we must have one of the following alternatives for fixed $N$:
 \begin{enumerate}
 \item  $  q_{0,\varepsilon} -1 \in H^N $
 \item $q_{0,\varepsilon} - \tanh(x) \in H^N $
 \item $q_{0,\varepsilon}+1 \in H^N $
 \item $q_{0,\varepsilon}+\tanh(x) \in H^N $.
   \end{enumerate} 
 Replacing $q$ by $-q$ if necessary it remains to consider two situations: 
 \begin{enumerate}[(i)]
 \item $q+1 \in H^s $
 \item $q+\tanh(x) \in H^s$.
 \end{enumerate}
 It is easy to see that $q \in X^s$ if one these situations holds.
 We recall the definition of the Miura map
 \[ M(q) = q_x+q^2. \] 
 Then the following lemma holds.
 \begin{lem}\label{lem:ABC}
 A)  The map
   \[ H^s \ni w \to \big(M( w-1) -1\big) \in H^{s-1} \]
   is a diffeomorphism of $H^s$ to its range 
\[   \{ u\in H^{s-1}  : -\partial_{xx}+u \text{ has no eigenvalue } \le -1 \}. \]    

   B)  The map
   \[ H^s\times (0,\infty)  \ni (w, \lambda) \to  \bigl( M(w-\lambda \tanh(\lambda x) )-\lambda^2\bigr)  \in H^{s-1}  \]
   is a diffeomorphism to its range 
   \[ \{ u\in H^{s-1} : -\partial_{xx}+u \text{ has  a negative  eigenvalue} \}. \] 
   Moreover $-\lambda^2$ is the lowest eigenvalue.

%{\color{red} I think that the sign $-$ is important. The statement is not true with $+$. There is a problem with taking only $\lambda=1$: The range would then be a Hilbert manifold, not an open set in a Hilbert space}   

C) In both cases A) and B), let $q=w-1$ resp. $q=w- \tanh(x)$, then $q:\R\times\R\mapsto \R$ satisfies the real modified KdV \eqref{mKdV} iff
\[ u=M(q)-1= q_x+q^2-1 \]
satisfies the KdV equation \eqref{KdV}:
\begin{equation}\label{KdV6} u_t - 6 u_x + u_{xxx} -6 u u_x = 0. \end{equation}
\end{lem}  
 Since the KdV equation \eqref{KdV} is wellposed in $H^{-1}$ \cite{KV} (in the sense that the flow map extends continuously), if $w=q+1\in H^s$ then it follows from A) and C) that $u=M(q)-1\in H^{s-1}\subset H^{-1}$ and \eqref{mKdV} is well-posed in $H^s$, $s\geq0$. 
 Similarly Theorem \ref{thm:lwpmKdV} follows in the case (ii): $w=q+\tanh(x)\in H^s$.
 
 It remains to prove Lemma \ref{lem:ABC}.
 Part B) of the Lemma has been proven by Buckmaster and the first author
 \cite{MR3400442}.  If $q$ satisfies \eqref{mKdV} then $u= q_x + q^2$
 satisfies \eqref{KdV}. Now suppose that $u$
 satisfies  KdV   \eqref{KdV6} (the term $6u_x$ is
 inessential, and can be removed by a Galilean transform). Since the
 preimage is unique it has to be a solution to mKdV, at least if the
 initial data is sufficiently smooth. This can be achieved by an approximation argument.

 It remains to prove A). It is easy to see (compare \cite{MR3400442}) that
 $w \in H^s$ implies $  M(w-1) -1 \in H^{s-1} $ for $ s \ge 0$. Moreover this map is clearly analytic.   The derivative
 at $w_0$ is
 \[   \dot w \to \dot w_x + 2 (w_0-1)  \dot w  \]
 which has the (right) inverse 
 \[ (Tf)(x) = -\int_{x}^\infty e^{2\int_x^y w_0 d\tau -2(y-x)} f(y) dy. \]
It is easy to see that $T$ maps $H^{s-1}$ to $H^{s}$  for all $ s \ge 0$ and $w_0 \in H^{s}$.
 Moreover the linearization is injective. 
 Indeed, suppose that $\dot w \in L^2$ satisfies
 \[   \dot w_x + 2 (w_0-1)  \dot w = 0 . \]
 Then $\dot w$ is absolutely continuous and decays to $0$ as $x \to \infty$.
 The variation of constants formula and a limit argument show that  $\dot w$ vanishes.

 To verify injectivity of the nonlinear map we assume that $w_0$ and $w_1$ are mapped to the same function. Then, with $\dot w= w_1-w_0$  
 \[ \dot w_x + (w_0+w_1)\dot w - 2 \dot w = 0    \]
 and hence $\dot w = 0 $ by the same argument as for the injectivity of the linearization.

 The argument for surjectivity is  based on Kappeler {\textit{et al}} \cite{MR2189502}. 
 Let $u \in H^{-1}$ be a function so that the spectrum of $-\partial^2 + u $
 is contained in $(-1,\infty)$. According to \cite{MR2189502} there exists
 a bounded positive function $\phi$ which satisfies 
\begin{equation}\label{phi} - \phi'' + u \phi + \phi = 0.  \end{equation}
 Let $ v = \frac{d}{dx} \ln \phi$. A straightforward calculation shows that
 \[ v' + v^2  = u +1, \hbox{ i.e. }M(v)-1=u.\]
 Let $\tilde v = v \ast \rho$ where $\rho\in C^\infty_0$ is supported in $[-1,1]$
 with integral $1$. It suffices to find $\tilde v$ so that
 \[ \lim_{x\to -\infty} \tilde v(x) = -1, \quad \lim_{x\to \infty} \tilde v(x) = -1. \]
Now we use \cite{MR3400442} to see that $\tilde v$ has a limit in $\{ \pm 1\} $ as $x\to \pm \infty$, possibly different on both sides. If 
 \[ \lim_{x\to -\infty} \tilde v(x) = 1 , \lim_{x\to \infty}\tilde  v(x) = -1, \]
 then $\phi \in L^2$ and it  were an eigenfunction of the eigenvalue $-1$, which  contradicts our assumption.  
Thus, if 
\[ \lim_{x\to \infty}\tilde  v(x) =-1, \]
then $\lim_{x\to -\infty}\tilde  v(x) = -1$ and we found the preimage in case A). 
Hence suppose that
\[ \lim_{x\to \infty}\tilde  v(x) = 1. \]  
Then 
 \[  \phi_1 = \phi(x) \int_x^\infty \phi(y)^{-2} dy  \]
 is a nonnegative solution of \eqref{phi} and is bounded for positive $x$. Hence
$   v_1= \frac{d}{dx} \ln   \phi_1 $
 satisfies (with $\tilde v_1=v_1\ast\rho$ as above) 
 \[ \lim_{x\to \infty}  \tilde v_1(x) = -1 \]
 and, by the previous considerations of our assumption,  
 \[ \lim_{x\to -\infty} \tilde v_1(x) = -1. \]
 With this we have found the preimage $v_1$ in case A).  
 %For selfcontainedness we prove surjectivitity in case B). 
 %Suppose that $\lambda=-1$ is the ground state energy with eigenfunction $\phi$. Then with $v$ defined as above,
% \[ \lim_{x\to -\infty} \tilde v(x) = 1, \quad \lim_{x\to \infty} \tilde v(x) =-1. \]  
%This is the  searched   preimage. 
   \end{proof}

\setcounter{equation}{0}
\section{The transmission coefficient} \label{sec:T}%%%%%%%%%%%%%%
%%%%%%%%%%%%%%%%%%%%%%%%%%%%%%%%%%%%%%%%%%%%%%%%%

We  introduce    the renormalised transmission coefficient $\Tc^{-1}(\lambda)$ and state its properties in Theorem \ref{thm:Tc} in this section.

We will first  recall  the definition of the transmission coefficient $T^{-1}$ associated to the Lax operator \eqref{LaxOp}, i.e. the Lax equation:
\begin{equation}\label{Lax} 
 u_x 
 = \left( \begin{matrix} -i\lambda & q \\ \bar q & i \lambda \end{matrix} \right) u, 
 \end{equation} 
on the Riemann surface $\cR$ (see Subsection \ref{subss:Riemann} below for the definition),
  in the classical functional setting where $q-1$ is Schwartz function in Subsection \ref{subs:Jost}.
  
With the notations introduced in Subsection \ref{subs:notation}, we will give an \emph{asymptotic expansion} of the transmission coefficient $T^{-1}$ in Subsection \ref{subsubs:T}, which will play a key role in the analysis of  $T^{-1}$.

Finally in Subsection \ref{subss:Tc-1} we discuss the renormalisation of the transmission coefficient and give  Theorem \ref{thm:Tc} stating the well-definedness and the asymptotic expansion of the renormalised transmission coefficient $\Tc^{-1}$ in our finite energy setting $q\in X^s$, $s>\frac12$, whose proof will be postponed in Section \ref{sec:Tc}.

\subsection{A Riemann surface}\label{subss:Riemann}
We define a Riemann surface  by
\[\{(\lambda,z)\in \C^2\,|\,  \lambda^2 = 1+ z^2\}. \]  
If  infinity is added,  its genus is $0$ and it is indeed the Riemann sphere with respect to the complex variable $\zeta:=\lambda+z$.

We  typically choose the upper sheet $\mathcal{R}$ of this Riemann surface:  
\begin{equation}\label{cR}\begin{split}
&\mathcal{R}=\bigl\{(\lambda,z)\,|\, \lambda\in \mathcal{V},
\quad 
z=z(\lambda)=\sqrt{\lambda^2-1}\in \mathcal{U}\bigr \},
\\
&\hbox{ where }\mathcal{V}:=\C\setminus \Ic,
\, \Ic:= (-\infty, -1]\cup [1,+\infty), 
\,\mathcal{U}:=\{z\in \C\,| \, \Im z>0\},
\end{split}\end{equation} 
and we can take simply  $\lambda\in \mathcal{V}$ as the coordinate on 
$\mathcal{R}$.
We notice the following symmetry of $\cR$
\begin{equation}\label{Rsymmetry}\begin{split}
&(\lambda,z)\in \mathcal{R}\Leftrightarrow (\ov \lambda, -\ov z)\in \mathcal{R}.
\end{split}\end{equation}
In particular, the points $(\pm \sqrt{1-\tau^2/4}, i\tau/2)$, $\tau\in (0,2]$ and the purely imaginary points $(\pm i\sigma, i\tau/2)$, $\tau\in [2,\infty)$ stay on $\cR$:
\begin{equation}\label{sigma-tau}\begin{split} 
(\pm i\sigma, i\tau/2)\in \cR  \hbox{ whenever }   \tau\geq 2\hbox{ and }
 \sigma= \sqrt{\tau^2/4-1}\in\R.
\end{split}\end{equation}

We can  define a conformal mapping from $(\lambda,z)\in \cR$ to $\zeta\in\cU$ the upper half-plane  by
$\zeta=\zeta(\lambda)=\lambda+z$.
The mapping takes the cuts $\lambda\in\Ic$  to the real axis $\zeta\in\R$ and the neighbourhood of $\infty$ for $\Im \lambda<0$ to a neighbourhood of $\zeta=0$.
The inverse mapping is given by the so-called Zukowsky mapping
$\zeta\mapsto\lambda=\lambda(\zeta)=\frac 12(\zeta+\frac{1}{\zeta})$ and hence $1=(\lambda-z)\zeta$, $z=z(\zeta)=\frac 12(\zeta-\frac 1\zeta)$, $\frac{1}{\bar \zeta}=\bar\lambda-\bar z\in \cU$.

 \subsection{Jost solutions and the transmission coefficient}\label{subs:Jost}In this subsection we assume the classical functional setting (as in \cite{FT, ZS73})
\begin{equation}\label{SchwCond}
q=1+q^0,\quad q^0\in \mathcal{S}(\R) \hbox{ Schwartz function},
\end{equation} 
and we are going to introduce the Jost solution of the Lax equation \eqref{Lax} as well as the associated  transmission coefficient.
\subsubsection{Real line case $(\lambda,z)=(\hat{\xi},\xi ) \in\R^2$}
\label{subs:cR}

Let $0\neq z=\xi \in\R$ and $\lambda=\hat{\xi}\in \R$  such that $\hat{\xi}=(1+\xi^2)^{\frac12}>1$.
 Then under the assumption \eqref{SchwCond} on the potential $q$, $\pm i \xi$ are the two eigenvalues of the matrix in \eqref{Lax} at infinity:
 $ \left( \begin{matrix} -i\hat{\xi} & 1 \\ 1 & i \hat{\xi} \end{matrix} \right).$
 
Let   the Jost solution $u_l$  solve the Lax equation   \eqref{Lax} (viewing $\lambda=\hat{\xi}$ as parameter) satisfying the following boundary conditions at 
$- \infty$ 
\begin{align*}
&u_l( x )
=  e^{-i\xi x}
\left(\begin{matrix}
1 \\ i(\hat{\xi}-\xi) 
\end{matrix}\right)
+o(1)
\hbox{ as }x\rightarrow -\infty.
\end{align*}
Then there exist two complex numbers $T^{-1}, R\in\C$ such that $u_l$ takes the following asymptotic at $+ \infty$:
\begin{align*} 
&u_l( x )
=e^{-i\xi x} T ^{-1}  \left(\begin{matrix}
1\\ 
i(\hat{\xi}-\xi)  
\end{matrix}\right)
   + e^{i\xi x}\,R\, T^{-1} 
    \left(\begin{matrix}
1\\ i(\hat{\xi}+\xi)  
\end{matrix}\right)
+o(1)
\hbox{ as }x\rightarrow +\infty.
\end{align*}  
These two complex numbers $T^{-1},  R$ are called the transmission coefficient
 and the  right reflection coefficient respectively 
 \footnote{In this paper we  call $T^{-1}$ the transmission coefficient while its reciprocal   $T$ is the physical relevant  transmission coefficient. 
 We can define similarly the left reflection coefficient by considering the asymptotic at $-\infty$ of  the Jost solution with the boundary condition $ e^{i\xi x}
 \begin{pmatrix}
1 \\ i(\hat{\xi}+\xi) 
\end{pmatrix}+o(1)$ as $x\rightarrow+\infty$. }. 

Observe that if $u=(u^1, u^2)^T$ is the solution of \eqref{Lax}, then the quantity $|u^1|^2-|u^2|^2$   is  constant.
 We compare the asymptotic behaviours of the Jost solution $u_l$ at $\pm\infty$ respectively to acquire
\begin{equation}\label{Tleq1}
|T|^2=1-(\hat{\xi}+\xi)^2|R|^2\leq 1,
\quad\hbox{if } (\lambda,z)=(\hat{\xi},\xi)\in\R^2.
\end{equation}

%Noticing the symmetry $(u^1, u^2)\mapsto (\ov u^2, \ov u^1):=\ov u^T$ in the systems \eqref{Lax} when $\lambda=\hat{\xi}\in\R$, we also derive from the conservation of the Wronskian $\det(\ov{(u_l)}^T, u_r)$ that
%$$
%l\ov T=-(\hat{\xi}+\xi)^2\ov r T.
%$$

\subsubsection{Upper Riemann sheet case $(\lambda,z)\in\mathcal{R}$}\label{subsubs:T,R}
The Jost solution $u_l(x;\lambda)$ defined above on the ``real axis'' 
$(\lambda,z)=(\hat{\xi},\xi)\in\R^2$ can be analytically continued to the upper Riemann sheet   $(\lambda, z)\in\mathcal{R}$,  taking the following asymptotics
\begin{equation}\label{Asymul}\begin{split}
&u_l(\lambda,x,t)
= e^{-iz x }
\left(\begin{matrix}
1 \\ i(\lambda-z)
\end{matrix}\right)
+o(1)e^{(\Im z) x}
\hbox{ as }x\rightarrow -\infty,
\\ 
&u_l(\lambda,x,t)
= e^{-iz x} T^{-1} (\lambda)
\left(\begin{matrix}
  1
\\ 
i(\lambda-z)
\end{matrix}\right)
+o(1) e^{(\Im z) x}
\hbox{ as }x\rightarrow +\infty.
\end{split}\end{equation}
Under the potential assumption \eqref{SchwCond},   $ T^{-1}(\lambda) $
 is a holomorphic function on $\mathcal{R}$
  and  $\lim_{|\lambda|\rightarrow \infty} T^{-1}(\lambda)=1$. 

  The possible zeros of $T^{-1}(\lambda)$ for $\lambda \in \C \backslash (\R \backslash (-1,1) )$  are   located on the interval
$(-1,1)\subset\R$.  Indeed, if $T^{-1}(\lambda)=0$, then
$\lambda\in\cV=\C\backslash \Ic$ by \eqref{Tleq1}.  Thus
$z=\sqrt{\lambda^2-1}\in\cU$ has strictly positive imaginary part such
that $\lambda, u_l$ (with the asymptotics \eqref{Asymul} and with $T^{-1}(\lambda)=0$) are the  eigenvalues and the
corresponding eigenfunctions of the self-adjoint Lax operator $L$ given in
\eqref{LaxOp}, and thus $\lambda\in(-1,1)\subset\R$.  Let $\lambda \in
(-1,1)$ be an eigenvalue, then $\pm iz=\mp \sqrt{1-\lambda^2}$ are negative and positive real numbers. By checking the characteristic exponents of the
ODE \eqref{Lax}: $Lu=\lambda u$ near infinity, the geometric multiplicity of $\lambda$ is $1$.
Since the Lax operator $L$ is self-adjoint, the algebraic multiplicity is also $1$, and
all eigenvalues in $(-1,1)$ are simple. As a consequence, the
root $\lambda$ of $T^{-1}$ has multiplicity $1$. 
%{\color{red} Should we include a proof?} 

We denote these at most countably many zeros on $(-1,1)$ by $\{\lambda_m\}_m$ and  
\begin{equation}\label{zk}
z_m=i\sqrt{1-(\lambda_m)^2}\in i(0, 1],
\quad m\in\N.
\end{equation}

We have the following symmetry for $T^{-1}$:
\begin{equation}\label{T:symm}
\ov T^{-1}(\lambda)=T^{-1}( \ov \lambda),
\quad (\lambda, z),\,( \ov \lambda, -\ov z)\in\mathcal{R}.
\end{equation}
 Indeed,  the symmetry of the Lax equation \eqref{Lax} implies that $\ov{u_l}^T
:=(\ov {u_l^2}, \ov {u_l^1})$ with the asymptotics
\begin{align*}
&\ov {u_l}^T
=-i(\ov \lambda-\ov z)
 e^{ i\ov z x}
\left(\begin{matrix}
1 \\  i(\ov\lambda+\ov z)
\end{matrix}\right)
+o(1)e^{\Im z x}
\hbox{ as }x\rightarrow -\infty,
\\ 
&\ov {u_l}^T
=-i(\ov\lambda-\ov z)e^{i\ov zx} \ov T^{-1}(\lambda) 
\left(\begin{matrix}
1 \\  i(\ov\lambda+\ov z)
\end{matrix}\right)+o(1)e^{\Im z x}
\hbox{ as }x\rightarrow +\infty,
\end{align*}
  satisfies the Lax equation  \eqref{Lax}  with $(\lambda, z)\in\mathcal{R}$ replaced by 
$(\ov \lambda, -\ov z)\in\mathcal{R}$, which itself possesses a Jost solution with the following asymptotics:
\begin{align*}
&u_l( \ov\lambda)
= e^{i\ov z x}
\left(\begin{matrix}
1 \\ i(\ov\lambda+\ov z) 
\end{matrix}\right)
+o(1)e^{\Im z x}
\hbox{ as }x\rightarrow -\infty,
\\ 
&u_l(\ov \lambda)
=e^{i\ov z x} T ^{-1}(\ov \lambda) \left(\begin{matrix}
1\\ 
i(\ov\lambda+\ov z)  
\end{matrix}\right)+o(1)e^{\Im z x}
\hbox{ as }x\rightarrow +\infty.
\end{align*}
By uniqueness we deduce \eqref{T:symm}.
Therefore
\begin{itemize}
\item   For $\hat{\xi} \in \Ic$, the limits $\lim_{\lambda\rightarrow \hat{\xi}+i0}|T^{-1}(\lambda)|$
and $\lim_{\lambda\rightarrow\hat{\xi}-i0}|T^{-1}(\lambda)|$  are the same.
Hence
the subharmonic function $\ln |T^{-1}(\lambda)|$ on $\cV$ is continuous on $\Ic$ and generally $\ln |T^{-1}(\lambda)|=\ln |T^{-1}(\ov\lambda)|$ for $(\lambda,z)\in\cR$;

\item For $(\lambda_{\pm},z)=(\pm i\sigma, i\tau/2)\in\cR$ with $\tau\geq 2$ and $\sigma=\sqrt{\tau^2/4-1}\in\R$, 
{\small\begin{equation}\label{RealSymm}
\begin{split}
\frac 12 \sum_{\pm }\Re &\ln T^{-1}(\lambda_{\pm})
%=\frac 12  \Re \bigl(\ln T^{-1}(\lambda_{+}) 
 %+ \ln T^{-1}(\lambda_-)\bigr) 
%\\
 =\frac 12  \Re \bigl(\ln T^{-1}(\lambda_{+}) 
+ \ln T^{-1}(\overline{\lambda_+})\bigr)
=  \Re \ln T^{-1}(\lambda_{+}).
\end{split}\end{equation}}
\end{itemize}

Let us take the time variable into account.  
We multiply $u_l$ by $e^{-iz(2\lambda)t}$   such that the time evolutionary equation in \eqref{LaxPair} ensures
$$
\d_t(T^{-1})=0, 
\quad 
\d_t R=4iz\lambda R.
$$
That is, the transmission coefficient $T^{-1}(\lambda)$ is conserved by the Gross-Pitaevskii flow and we will make use of it to define the conserved energies for the Gross-Pitaevskii equation \eqref{GP}.

Similarly, $e^{-iz(4\lambda^2+2)t}u_l$ satisfies \eqref{LaxPair} with $q=\psi$ and with the matrix in $\eqref{LaxPair}_2$ replaced by $P_{\hbox{\tiny mKdV}}$ in \eqref{PmKdV}. The same transmission coefficient $T^{-1}(\lambda)$ as for the Gross-Pitaevskii equation is conserved by the modified KdV flow \eqref{mKdV}.

\subsection{Notations}\label{subs:notation}
Let $q\in X^s$, $s>\frac12$.
Let $(\lambda,z)\in\cR$ and $\zeta=\lambda+z\in\cU$ the upper half-plane be as in Subsection \ref{subss:Riemann}.
Then
\begin{equation}\label{BoundZeta}
|q|^2-\zeta^2\neq 0,  
\hbox{ and  } \frac1{\bigl| |q|^2-\zeta^{2}\bigr|}\leq \frac1{(\Im\zeta)^{2}},
\quad \frac{|\zeta|}{ \bigl| |q|^2-\zeta^{2}\bigr|}\leq \frac1{\Im\zeta}.
\end{equation}
{\small
Indeed, if $|\Re\zeta|\geq\Im\zeta$, then 
$$
\frac1{\bigl| |q|^2-\zeta^{2}\bigr|}
\leq \frac1{ |\Im \zeta^2| }  \leq \frac1{2(\Im\zeta)^{2}}
\hbox{ and }
\frac{|\zeta|}{\bigl| |q|^2-\zeta^{2}\bigr| }
\leq \frac{\sqrt2|\Re\zeta|}{  2|\Re\zeta|\Im\zeta}\leq  \frac1{\Im\zeta},
$$
while if $|\Re\zeta|\leq\Im\zeta$, then
$$
\frac1{\bigl| |q|^2-\zeta^{2}\bigr|}
= \Bigl( \bigl(|q|^2+(\Im\zeta)^2-(\Re\zeta)^2\bigr)^2+(2\Re\zeta\Im\zeta)^2\Bigr)^{-\frac 12} \leq \frac1{|\zeta|^2}.
$$  }

We introduce  the following  functions which will play an essential role in the  analysis of the transmission coefficient: 
 \begin{equation}\label{q1234}\tag{$\ast$}
 \begin{split}
 & q_1= \frac{i\zeta(|q|^2-1)-\bar q q'}{|q|^2-\zeta^2},
 \quad q_2=\frac{i\zeta q'+ (|q|^2-1)q   }{|q|^2-\zeta^2},
 \quad q_3=\frac{-i\zeta \bar q'+ (|q|^2-1)\bar q   }{|q|^2-\zeta^2},
 \\
&  q_4=\frac{2i\zeta (|q|^2-1)+q\bar q'-\bar q q'   }{|q|^2-\zeta^2},
 \quad \varphi(x)=2izx+\int^x_0 q_4(x_1) \dx_1.
 \end{split}
 \end{equation}

As in \cite{KT}, let the symbols $ \<X>, \<Y>$ correspond to the ordered integrals  with respect to the functions  $e^{-\varphi(x)}q_3(x)$ and $e^{\varphi(y)}q_2(y)$ respectively in the following way
\begin{align*}
&  \<XY> 
: 
=\int_{x_1<y_1} 
e^{\varphi(y_1)-\varphi(x_1)}q_3(x_1) q_2(y_1)
   \dx_1\dy_1,  
\\
&  
\<XY>^j 
: =\int_{x_1<y_1<\cdots<x_j<y_j}
\prod_{n=1}^j e^{\varphi(y_n)-\varphi(x_n)}q_3(x_n) q_2(y_n)
   \dx \dy,
\\
&  
\<XXYY> 
:=\int_{t_1<t_2<t_3<t_4}
  e^{ \varphi(t_{4})+\varphi(t_3)-\varphi(t_2)-\varphi(t_1)}
  q_3(t_1)q_3(t_2) q_2(t_3)q_2(t_4) dt,  
\end{align*}  
{\small\begin{align*}
&\<XXXYYY>=\int_{t_1<\cdots<t_6}
  e^{ \varphi(t_{6})+\varphi(t_5)+\varphi(t_4)-\varphi(t_3)-\varphi(t_2)-\varphi(t_1)}
  q_3(t_1)q_3(t_2)q_3(t_3) q_2(t_4)q_2(t_5)q_2(t_6) dt,  
  \\
&\<XXYXYY>=\int_{t_1<\cdots<t_6}
  e^{ \varphi(t_{6})+\varphi(t_5)-\varphi(t_4)+\varphi(t_3)-\varphi(t_2)-\varphi(t_1)}
  q_3(t_1)q_3(t_2)q_2(t_3) q_3(t_4)q_2(t_5)q_2(t_6) dt,  
\end{align*}}
and so on.
In particular, a symbol
%{\color{red} It seems not to be a graph in the sense of the definition of a graph - it could be rewritten to be graph so. I am uncertain with we should keep the notation graph.  I am not against it.} 
of form $\<XBBY>_{2j}$, where $\<B>\<B>$ under the arc $\<XY>$ consists of $(j-1)$    non-interacting  symbols $\<XY>$, is said to be \textit{connected} %{\color{red} connected or primitive instead of regular?} 
of degree $2j$. We will simply omit the subscript $2j$ in $\<XBBY>_{2j}$  when the degree is clear.
For example, $\<XY>, \<XXYY>$ are connected symbols of degree $2,4$ respectively, and $\<XXXYYY>, \<XXYXYY>$ are connected symbols of degree $6$, while $\<XYXY>\,\,\,$ is not a connected symbol.

We  introduce the operator $S$ as follows:
\begin{equation}\label{SS}
(Sf)(t)=\int_{x<y<t}e^{\varphi(y)-\varphi(x)} q_2(y)(q_3 f)(x)\dx\dy,
\end{equation}
such that we can express $\<XY>^j=\lim_{t\rightarrow\infty}(S^j 1)(t)$.

  %%%%%%%%%%%
\subsection{Asymptotic expansion of the transmission coefficient $T^{-1}$}\label{subsubs:T}
Let $q-1\in\cS(\R)$.
Recall the definition of  the transmission coefficient  $T^{-1}(\lambda)=\lim_{x\rightarrow\infty} e^{izx}u^1(x)$ in Subsection \ref{subs:Jost}, where 
 the Jost solution $u$ satisfies the Lax equation \eqref{Lax} and  the asymptotics  \eqref{Asymul}.
 We are going to solve the initial value problem \eqref{Lax}-$\eqref{Asymul}_1$  iteratively,    to derive the asymptotic expansion  of the transmission coefficient $T^{-1}$.
 
More precisely, we are going to follow the following procedure: We will first make change of variables $u\mapsto w$ (see \eqref{u-w}  below) to renormalise the   problem \eqref{Lax}-$\eqref{Asymul}_1$ into the following ODE, where $q_1, q_2, q_3, q_4$ are given in \eqref{q1234}:
\begin{equation}\tag{ODE}\label{ODE} 
\begin{split} 
w_x = &     \left( \begin{matrix} 0 &0 \\ 0 & 2iz \end{matrix} \right) w 
+\left( \begin{matrix} 0 
& q_2  
 \\ q_3
  & q_4  \end{matrix} \right)w,
  \quad \lim_{x\rightarrow-\infty}w(x) =\begin{pmatrix} 1\\0\end{pmatrix},
\end{split} 
 \end{equation}  
such that 
\begin{align*}
e^{ \int_{-\infty}^{\infty} q_1\dm }T^{-1}(\lambda)
=\lim_{x\rightarrow\infty}w^1(x)
\hbox{ the asymptotic of the first component of }w.
\end{align*}
Then we formally  solve \eqref{ODE}  iteratively as follows
\begin{equation}\label{eq:integral} 
\begin{split}
&w=\sum_{n=0}^\infty w_{n},
\quad w_{0}=\begin{pmatrix} 1\\0\end{pmatrix},
\\
&w_{n}(x)=\int^x_{-\infty}  \left( \begin{matrix}
0&  q_2(x_1)
\\
e^{2iz(x-x_1)+\int^x_{x_1} q_4 \dm} q_3(x_1)
 &  0
\end{matrix}\right) \,  w_{n-1}(x_1)\dx_1,
\end{split}
\end{equation} 
to derive the following formal asymptotic expansion for $T^{-1}$: 
{\small\begin{align*}
e^{ \int_{-\infty}^{\infty} q_1\dm }T^{-1}(\lambda)
=\lim_{x\rightarrow\infty}w^1(x)
=\sum_{n=0}^\infty\lim_{x\rightarrow\infty}w_n^1(x)
=1+\sum_{j=1}^\infty\<XY>^j, 
\end{align*} }
where we noticed $\lim_{x\rightarrow\infty}w_{2j-1}^1(x)=0$
and $\lim_{x\rightarrow\infty}w_{2j}^1(x)=\<XY>^j$, $j\geq 1$, with $\<XY>^j$ given in Subsection \ref{subs:notation}. 
 \begin{prop} \label{Prop:Expan,T}
 Let $q-1\in\cS$. Let $(\lambda,z)\in\cR$ with $\zeta=\lambda+z\in\cU$ as in Subsection \ref{subss:Riemann}.
 Recall the notations in Subsection \ref{subs:notation}.
 Then the transmission coefficient $T^{-1}$ defined in  Subsection \ref{subs:Jost} expands asymptotically as follows:
\begin{equation}\label{Expansion:T}
e^{\int_{-\infty}^{\infty} q_1 \dx}T^{-1}(\lambda) 
=1+\sum_{j=1}^\infty T_{2j}(\lambda),
\quad T_{2j}=\<XY>^j,
\end{equation}
and  its logarithm expands asymptotically as 
\begin{equation}\label{Expansion:lnT-1} 
 \int_{-\infty}^{\infty} q_1 \dx+\ln T^{-1}(\lambda)
% =\<XY>-2\<XXYY>+(12\<XXXYYY>+4\<XXYXYY>)+\cdots
 =T_2+\sum_{j=2}^\infty \tilde T_{2j},
\end{equation}
where $\tilde T_{2j}$ is linear combination of \emph{connected} symbols $\<XBBY>_{2j}$  of degree $2j$.
\end{prop}

\begin{proof}
 It is convenient to rewrite the Lax equation \eqref{Lax} for the Jost solution $u_l$ (we will omit the subscript $l$) in several steps.
 
A straightforward calculation shows that  
  \[ \frac1{|q|^2-\zeta^2} 
\Bigl[  \left(\begin{matrix} -i\zeta & q \\
                          \bar q & i\zeta \end{matrix} \right) 
 \left( \begin{matrix} -i\lambda  & q \\ \bar q & i\lambda \end{matrix} \right) 
 -\left( \begin{matrix} |q|^2-1  & 0 \\ 0 & |q|^2-1 \end{matrix} \right) \Bigr]
\left(\begin{matrix} -i\zeta & q \\ \bar q &i\zeta   \end{matrix} \right) 
 =
  \left( \begin{matrix} -iz & 0 \\ 0 & iz \end{matrix} \right).
  \] 
  We define the renormalised Jost solution of $u$ as
\[v= \left( \begin{matrix} -i\zeta & q \\ \bar q & i\zeta \end{matrix} \right)u,
\hbox{ such that }u=\frac1{|q|^2-\zeta^2} 
\left( \begin{matrix}-i\zeta & q \\ \bar q & i\zeta \end{matrix} \right)v. \] 
Hence $v$ solves (with $q_1, q_2, q_3, q_4$ defined in \eqref{q1234})
\[ 
\begin{split} 
v_x = & \frac1{|q|^2-\zeta^2}\Big[ 
\left( \begin{matrix}-i\zeta & q \\ \bar q & i\zeta \end{matrix} \right)
 \left( \begin{matrix} -i\lambda & q \\ \bar q  & i\lambda \end{matrix} \right) \left( \begin{matrix} -i\zeta   & q \\ \bar q & i\zeta \end{matrix} \right)     
 +\left( \begin{matrix} 0 & q_x \\ \bar q_x & 0 \end{matrix} \right) 
\left( \begin{matrix} -i\zeta & q \\ \bar q &i\zeta  \end{matrix} \right) \Big] v 
 \\   
= &   \left( \begin{matrix} -iz &0 \\ 0 & iz \end{matrix} \right) v
+ 
\left( \begin{matrix} -q_1    
&q_2   
 \\ q_3
  &q_4-q_1  \end{matrix} \right)v.
\end{split} 
 \]

We want to remove the upper left entries of the two matrices:
Let  
\begin{equation}\label{u-w} 
w=-\frac{1}{2iz}e^{izx+\int_{-\infty}^x q_1 \dm} v
=-\frac{1}{2iz}e^{izx+\int_{-\infty}^x q_1 \dm}
\left( \begin{matrix} -i\zeta & q \\ \bar q & i\zeta \end{matrix} \right)u,
\end{equation}
then it satisfies the renormalised \eqref{ODE} above. 
In other words, the renormalized Jost solution
$w $
satisfies    the following integral equation 
\begin{equation*}\label{eq:w,integral}
\begin{split} 
 w (x)
&=  \left( \begin{matrix} 1 \\ 0 \end{matrix} \right)  
  +
\int_{-\infty}^x  
\left( \begin{array}{cc}
0
&  q_2(x_1)
\\
e^{2iz(x-x_1)+\int^x_{x_1} q_4 \dm} q_3(x_1)
 & 0
\end{array}\right) \,  w(x_1)\,dx_1,
\end{split}
\end{equation*}   
with the following asymptotics as 
$x\rightarrow\pm\infty$ (recalling $u$'s asymptotics \eqref{Asymul}):
\begin{align*}
&   w (\lambda,x,t)= \left(\begin{matrix}
1\\0
\end{matrix}\right)+o(1)
\hbox{ as }x\rightarrow -\infty,
\\
&  w (\lambda,x,t)= 
\left(\begin{matrix}
e^{   \int_{-\infty}^{\infty} q_1\dm } T^{-1}(\lambda)\\0
\end{matrix}\right)+o(1)
\hbox{ as }x\rightarrow +\infty.
\end{align*}  
Hence we use the iterative procedure in \eqref{eq:integral} to derive the formal asymptotic expansion \eqref{Expansion:T} of $e^{ \int_{-\infty}^{\infty} q_1\dm }T^{-1}(\lambda)
=\lim_{x\rightarrow\infty}w^1(x)=\sum_{n=0}^\infty\lim_{x\rightarrow\infty}w_n^1(x)$. 

Finally, it follows from Theorem 3.3 \footnote{Indeed, we can endow the ring of the formal power series of the unknown symbols $ \<X>, \<Y>$ with a shuffle product and a coproduct, and we typically take a  commutative subalgebra $H$ where the symbols $ \<X>, \<Y>$ appear in non interacting  pairs, such that $g$ is a group-like element and $\ln g$ is a primitive element in $H$.
See also  \cite{KT, Lothaire, Reutenauer}    for more shuffle  algebra theory.} in \cite{KT} that whenever we have the \emph{formal} expansion as in \eqref{Expansion:T}: $g=1+\sum_{j=1}^\infty T_{2j}$, $T_{2j}=\<XY>^j$, we will have the formal expansion of its logarithm in \eqref{Expansion:lnT-1}: $\ln g=T_2+\sum_{j=2}^\infty \tilde T_{2j}$.

\end{proof}

  %%%%%%%
  \subsection{The renormalised transmission coefficient $\Tc^{-1}$}\label{subss:Tc-1}
  Let $q\in X^s$, $s>\frac12$ and let $(\lambda, z)\in\cR$, $\zeta=\lambda+z\in\cU$. 
  
  Recall $q_1= \frac{i\zeta(|q|^2-1)-\bar q q'}{|q|^2-\zeta^2}$  defined in \eqref{q1234}, then its integral $\int^\infty_{-\infty} q_1\dx$ (which appears in the  expansion  of $T^{-1}$) may not be well-defined for $q\in X^s$. 
More precisely, by view of the following fact coming from $\lambda^2-z^2=1$ and $\zeta=\lambda+z$:
{\small\begin{equation}\label{RelationZeta}\begin{split}
&\frac1{1-\zeta^2}=\frac{-1}{2z\zeta},
\quad\frac{\zeta}{1-\zeta^2}=\frac{-1}{2z},
%\quad \frac{1}{2z\zeta}+\frac{1}{4z^2\zeta^2}=\frac{1}{4z^2},
%\\
%&
\quad \frac1{|q|^2-\zeta^2}
=\frac1{1-\zeta^2}-\frac{|q|^2-1}{(1-\zeta^2)(|q|^2-\zeta^2)},
\end{split}\end{equation}}
we  rewrite the integral of $q_1$ when $q-1\in \cS(\R)$ as follows:
{\small\begin{equation}\label{Integral:q1}\begin{split}
\int_{-\infty}^{\infty} q_1 \dx 
&=-\frac{i}{2z}\cM 
-\frac{i}{2z\zeta}\cP+\frac{i}{2z}\int_{\R}\frac{(|q|^2-1)^2}{|q|^2-\zeta^2}\dx 
%-\frac{1}{4z^2\zeta^2}\int_{\R}i\Im (q'\bar q)\dx
-\frac1{2z\zeta}\int_{\R} \frac{\bar q q'(|q|^2-1)}{(|q|^2-\zeta^2)} \dx,  
\end{split}\end{equation}  }
where $\cM=\int_{\R} (|q|^2-1)\dx$, $\cP=\Im\int_{\R} q\bar q'\dx$ are the mass and momentum (in \eqref{MP}), which can be well-defined only under further integrability assumptions on $|q|^2-1, q'$.

We hence introduce the \textit{renormalised} transmission coefficient $\Tc^{-1}$ which is the transmission coefficient $T^{-1}$ module the mass and momentum, such that it is well-defined for $q\in X^s$. 
More precisely, we have
\begin{thm}\label{thm:Tc} 
Let $q\in X^s$, $s>\frac 12$.
Then there exists a renormalised transmission coefficient $\Tc^{-1}(\lambda)$ which is holomorphic on the Riemann surface $\cR$ (defined in \eqref{cR}), such that 
\begin{itemize}
%\item  If $ E_{s_0}(q)\leq c_0$, $-\ln \Tc(\lambda)$ is well-defined on 
%$\{(\lambda,z)\in\cR\,|\,|z|\geq 1\}$; 
%and  expands as in the righthand side of \eqref{expan:lnT};
\item   If $q=1+q^0$, $q^0\in\cS(\R)$, then  
\begin{equation}\label{lnTc}\begin{split}
&\Tc^{-1}(\lambda)=e^{ -i\cM(2z)^{-1}-i \cP(2z\zeta)^{-1}
}T^{-1}(\lambda),
\\
&\hbox{ i.e. }
-\ln \Tc(\lambda)=-\ln T(\lambda)-i \cM (2z)^{-1} -i \cP(2z\zeta)^{-1},
\quad \zeta=\lambda+z,
\end{split}\end{equation} 
where $T^{-1}$ is the transmission coefficient defined in Subsection \ref{subs:Jost}
and $\cM=\int_{\R} (|q|^2-1)\dx$ and $\cP=\Im\int_{\R} q\bar q'\dx$ 
are the  conserved mass and momentum associated to the Gross-Pitaevskii equation, such that 
\begin{equation}\label{Tc1}
\qquad |\Tc^{-1}(\lambda)|\geq 1\hbox{ if }\lambda\in\Ic=(-\infty,-1]\cup[1,\infty),  
\quad \Tc^{-1}\rightarrow 1 \hbox{ as }|\lambda|\rightarrow \infty. 
\end{equation}
\item 
For any fixed $(\lambda, z)\in\cR$, the renormalised transmission coefficient $\Tc^{-1}(\lambda; q)$ is extended uniquely to an analytic function in $q\in X^s$ (with respect to the analytic structure in Theorem \ref{thm:analytic})   and $\Tc^{-1}(\lambda; q(t))$ is conserved by the Gross-Pitaevskii flow on the existence time interval of the solution $q(t)$ (defined in \eqref{def}).
\item $\Tc^{-1}(\lambda)$ has the following asymptotic expansion
\begin{equation}\label{Expansion:Tc} 
\Tc^{-1}(\lambda)
 %=e^{\Phi}\Bigl(1+\<XY>+\<XY>\<XY>+\<XY>\<XY>\<XY>+\cdots\Bigr)
 =e^{\Phi(\lambda)}\Bigl(1 +\sum_{j=1}^\infty T_{2j}(\lambda)\Bigr),
 \quad T_{2j}=\<XY>^j,
\end{equation}
and its logarithm expands asymptotically  as 
\begin{equation}\label{Expansion:lnT} 
 \ln \Tc^{-1}(\lambda)
 %=\Phi+\<XY>-2\<XXYY>+(12\<XXXYYY>+4\<XXYXYY>)+\cdots
 =\Phi(\lambda)+T_2(\lambda)+\sum_{j=2}^\infty \tilde T_{2j}(\lambda),
\end{equation}
where
\begin{equation}\label{Phi}
\Phi(\lambda):=-\frac{i}{2z}\int_{\R}\frac{(|q|^2-1)^2}{|q|^2-\zeta^2}\dx 
+\frac1{2z\zeta}\int_{\R} \frac{\bar q q'(|q|^2-1)}{(|q|^2-\zeta^2)} \dx,
\end{equation}
 and $\tilde T_{2j}$ is linear combination of \emph{connected} symbols $\<XBBY>_{2j}$  of degree $2j$. 
Here the symbols are defined in Subsection \ref{subs:notation}.

\item  $\Tc^{-1}$ satisfies the following properties:
{\small\begin{equation}\label{Tc:property}\begin{split} 
&\Re {\Tc}^{-1}(\lambda)=\Re \Tc^{-1}(\bar\lambda)
 \hbox{ if } (\lambda,z)=(i\sigma,\pm i\frac\tau2)\in \cR,\,\tau\geq2,
 \,\sigma=\sqrt{\frac{\tau^2}4-1},
\\
& \Tc^{-1}\hbox{ has at most countably many simple zeros }\{\lambda_m\}\subset (-1,1). 
\end{split}\end{equation}  } 
\end{itemize}

We can define a superharmonic function $G(z)$ on the upper half plane $\cU$ as follows
 \begin{equation}\label{G}
G(z): = \frac 12\sum_{ \pm}
\Re\Bigl(4z^2\ln \Tc^{-1}\bigl(\pm\sqrt{z^2+1}\bigr)\Bigr),
\quad \Im z>0,
\quad \Im\sqrt{z^2+1}\geq0,
\end{equation}
such that $G\geq 0$ on the  upper half plane $\cU$ and $-\Delta G\geq 0$ is a nonnegative measure on the upper half plane $\cU$ as follows
\begin{equation}\label{DeltaS}
\begin{split}
&\nu_G(z)=-\Delta_{z} G(z) 
=-\pi\sum_{m}(2z)^2\delta_{z=z_m}\geq 0,
 \,\,z_m=i\sqrt{1-\lambda_m^2}\in i(0,1],
%\, \lambda_m\in (-1,1).
\end{split}
\end{equation} 
where $\{\lambda_m\}$ are the simple zeros of $\Tc^{-1}(\lambda)$ in \eqref{Tc:property}.
\end{thm} 
 The proof of this theorem will be demonstrated  in next section where the functional spaces $l^p_\tau U^2$, $l^p_\tau V^2$, $l^p_\tau DU^2$ will come into play.

%\subsection{Conserved mass, momentum and energy $\cM, \cP, \cE$}\label{subs:FT}
 
 \setcounter{equation}{0}%%%%%%%%%%%%%%%%%%%%%%%%%%%%%%%%%%
  \section{Proof of Theorem \ref{thm:Tc}} \label{sec:Tc}
  
 In this section we will prove Theorem \ref{thm:Tc} concerning the well-definedness and the property of the renormalised transmission coefficient $\Tc^{-1}(\lambda)$ in the energy framework $q\in X^s$, $s>\frac12$ in the following steps: \begin{itemize}
 \item In Subsection \ref{subs:norm}  we introduce the function spaces $l^p_\tau U^2, l^p_\tau V^2, l^p_\tau DU^2$. 
 
 \item We  derive some preliminary estimates   in Subsection \ref{subs:wn}.
 Then we   solve the renormalised Lax equation \eqref{RenormalLax} rigorously when $q\in X^s$ and define the renormalised transmission coefficient $\Tc^{-1}(\lambda)$ in Subsection \ref{subs:Holo}. 
 
 \item We  study the nonnegative superharmonic function $G(z)$ and conclude the proof of Theorem \ref{thm:Tc} in Subsection \ref{subsubs:G}.
 \end{itemize}
 
 %%%%%%%%%%%%
\subsection{Function spaces}\label{subs:norm} 
In this subsection we will  briefly recall the function spaces $U^2, V^2, DU^2$  and    the inhomogeneous  norms $\|\cdot\|_{l^p_\tau V}$, $V=U^2, V^2, DU^2$.
See \cite{HHK, KT} for more details of the $U^2, V^2$ theory.

%%%%%%
\subsubsection{Spaces $U^2, V^2, DU^2$}\label{subs:UDU}
We denote the bounded functions on $\R$ by $B(\R)$.
We use the spaces $U^2, V^2\subset B(\R)$ and $DU^2$ to substitute the Sobolev spaces $\dot H^{\frac 12}\not\hookrightarrow L^\infty$ and $\dot H^{-\frac 12}$.
The space  $V^2$ is defined as follows
\begin{align*}
V^2=\Bigl\{ v\,\Big|\,   \|v\|_{V^2}=\mathop{\sup}\limits_{-\infty<t_1<\cdots<t_N=\infty}
\Bigl( \sum_{j=1}^{N-1} |v(t_{j+1})-v(t_j)|^2 \Bigr)^{\frac 12}<\infty \Bigr\},
\end{align*} 
where we always set $v(\infty)=0$. In particular, the constant function $1\in V^2$ with norm $1$.
For any finite sequence $\{\phi_j\}_{j=1}^{N-1}$  with $\sum_{j=1}^{N-1}|\phi_j|^2=1$, the step function $\phi=\sum_{j=1}^{N-1}\phi_j 1_{[t_j, t_{j+1})}$ with $-\infty<t_1<\cdots<t_N=\infty$ is called  $U^2$ atom.
We  define the space $U^2$ by 
\begin{align*}
U^2 =\Bigl\{ u=\sum_{k=1}^\infty c_k \psi_k\,\Big|\,
(c_k)_k \in \ell^1(\N)
\hbox{ and } \psi_k \hbox{ is }U^2\hbox{ atom} \Bigr\},
\end{align*}
endowed with the $U^2$-norm:
\begin{align*}
\|u\|_{U^2 }=\inf\Bigl\{\sum_{k=1}^\infty |c_k|\,\Big|\, u=\sum_{k=1}^\infty c_k \psi_k,
\quad c_k\in\C,\quad \psi_k \hbox{ is }U^2\hbox{ atom}\Bigr\}.
\end{align*}

We define the space $DU^2$ via the distributional derivatives as
%\footnote{We can define the space $DV^2$ similarly.}
\begin{align*}
DU^2=\{u'\,|\, u\in U^2\}, 
\end{align*}
with the norm $\|u'\|_{DU^2}=\|u\|_{U^2}$.
Then $DU^2$ function is a distribution function with the following finite norm: 
\begin{equation*}\begin{split}
&\|f\|_{DU^2}= \sup\Bigl\{\int_{\R} f\varphi dt
 \,\Big|\, \|\varphi\|_{V^2}\leq 1, \varphi\in C^\infty_0(\R)\Bigr\}.
\end{split}\end{equation*}

We have the following pleasant   estimates which will be used frequently:  
\begin{equation}\label{product}
\begin{split}
&\|f\|_{L^\infty}\leq \|f\|_{V^2}\leq \|f\|_{L^\infty}+2\|f'\|_{DU^2},
\quad \|f\|_{V^2}\leq 2\|f\|_{U^2},
\\
&\|fg\|_{V^2}\leq \|f\|_{L^\infty}\|g\|_{V^2}+\|f\|_{V^2}\|g\|_{L^\infty},
\quad \|fg\|_{DU^2}\leq 2\|f\|_{V^2}\|g\|_{DU^2},
\\
&\|f(u)\|_{V^2}\leq C(f',\|u\|_{L^\infty})\|u\|_{V^2}. 
\end{split}
\end{equation}

\subsubsection{Space $l^p_\tau DU^2$} \label{subsubs:est}
We take  the localised version of $U^2, V^2, DU^2$-norms 
\begin{equation*}\label{Unorm}
\|u\|_{l^p_\tau U}=\bigl\| \|\chi_{\tau,k} u\|_{U} \bigr\|_{\ell^p_k(\Z)},
\quad U=V^2, U^2\hbox{ or }DU^2,
\end{equation*}
where $\tau\geq 2$ is the frequency scale and $\chi$ is a smooth function compactly supported on $[-\frac23, \frac23]$ with value $1$ on the interval $[-\frac13,\frac13]$, such that  $\chi_{\tau,k}$ form  a partition of unity:
\begin{equation}\label{unity}
1=\sum_{k\in\Z}\chi_{\tau,k},
\quad \chi_{\tau,k}=\chi\bigl(\tau(\cdot-\frac{k}{\tau})\bigr)=\chi(\tau\cdot-k).
\end{equation}
For any fixed real positive number $a>0$, $\|u\|_{l^p_\tau U}$ is equivalent to $\bigl\| \|\tilde\chi_{\tau,k} u\|_{U} \bigr\|_{\ell^p_k(\Z)}$ with $\tilde\chi=\chi(\cdot/a)$.

\begin{prop}[\cite{KT}]\label{prop:est}
We  have the following properties of the  norm $\|\cdot\|_{l^p_\tau DU^2}$:
\begin{itemize}
\item
The following inequality describes the effect of the phase shift  
\begin{equation}\label{PhaseShift}
\|e^{i\xi x}u\|_{l^2_\tau DU^2}\leq C\sqrt{(\tau+|\xi|)/\tau}\|u\|_{l^2_\tau DU^2}.
\end{equation}

\item 
The following inequality describes the effect of taking the derivative
\begin{equation}\label{fg:free}
\|f\|_{l^p_\tau U^2}\lesssim \tau\|f\|_{l^p_\tau DU^2}+\|f'\|_{l^p_\tau DU^2}.
\end{equation} 

\item The $l^p_\tau DU^2$, $p\geq 2$-norm and $\dot H^{\frac 12-\frac 1p}$-norm is related by 
$$\|u\|_{l^p_\tau DU^2}\leq C\tau^{\frac 1p-1}\|u\|_{\dot H^{\frac 12-\frac 1p}}.$$

 \end{itemize}
\end{prop}

\subsection{The renormalised transmission coefficient}\label{subs:Tr} 
Let $q\in X^s$, $s>\frac12$ and $(\lambda, z)\in\cR$ with $\zeta=\lambda+z\in\cU$.
Recall the renormalised Lax equation \eqref{ODE} in the proof of Proposition \ref{Prop:Expan,T}:
\begin{equation}\label{RenormalLax}\tag{ODE}
\begin{split} 
w_x = &     \left( \begin{matrix} 0 &0 \\ 0 & 2iz \end{matrix} \right) w 
+\left( \begin{matrix} 0 
& q_2  
 \\ q_3
  & q_4  \end{matrix} \right)w,
  \quad \lim_{x\rightarrow-\infty}w(x) =\begin{pmatrix} 1\\0\end{pmatrix},
\end{split} 
 \end{equation}  
 and its formally equivalent  iterative version \eqref{eq:integral}:
\begin{equation}\label{SolveJost}
\begin{split}
w=\sum_{n=0}^\infty w_n,
\,  w_{0}=\begin{pmatrix} 1\\0\end{pmatrix},
  \,\,
 w_{n}(t)=   \begin{pmatrix}\int^t_{-\infty}  q_2(x) w_{n-1}^2(x)\dx
\\
\int^t_{-\infty} e^{ \varphi(t)-\varphi(x)} q_3(x) w_{n-1}^1(x)\dx
\end{pmatrix}, 
\end{split}
\end{equation} 
$n\geq 1$, such that  if $q-1\in\cS$ then  
$$e^{ \int_{-\infty}^{\infty} q_1\dm }T^{-1}(\lambda)
=\lim_{t\rightarrow\infty}w^1(t)=\sum_{n=0}^\infty\lim_{t\rightarrow\infty}w_n^1(t),
$$  where   $q_1, q_2, q_3, q_4, \varphi$ are given in \eqref{q1234}. 
In particular,  noticing $w^1_{2j+1}=w^2_{2j}=0$, we can rewrite \eqref{SolveJost} as
{\small
\begin{equation}\label{SolveJost1}\begin{split}
 w^1(t)=\sum_{j=0}^\infty w_{2j}^1(t)
=\sum_{j=0}^\infty (S^j 1)(t),
\quad w^2(t)=\sum_{j=0}^\infty w_{2j+1}^2(t)
=\sum_{j=0}^\infty  (S_1 S^j 1)(t),
%=\sum_{j=0}^\infty \int_{x_1<y_1<\cdots<x_{j+1}<t} e^{\varphi(t)-\varphi(x_{j+1})}q_3(x_{j+1})
%\\
%&\qquad\qquad\qquad\qquad\qquad\times\prod_{l=1}^{j}e^{\varphi(y_{l})-\varphi(x_l)}q_2(y_{l})q_3(x_l) dx_1\cdots dx_{j+1},
\end{split}\end{equation} 
} 
where we recall the definition of the operator $S$ in \eqref{SS}:
$$(Sf)(t)=\int_{x<y<t}e^{\varphi(y)-\varphi(x)} q_2(y)(q_3 f)(x)\dx\dy,$$
and we define the operator $S_1$:
$$
(S_1 f)(t)=\int_{-\infty}^t e^{\varphi(t)-\varphi(x)} (q_3 f)(x)\dx.
$$
%Recalling the definition of the operator $S$ in \eqref{SS}: $(Sf)(t)=\int_{x<y<t}e^{\varphi(y)-\varphi(x)}q_2(y)(q_3f)(x)\dx\dy$ and noticing $w_{2j+1}^1=w_{2j}^2=0$, we can rewrite \eqref{SolveJost} as
%{\small\begin{align*}
%&w^1(y)=\sum_{j=0}^\infty w_{2j}^1(y)=\sum_{j=0}^\infty (S^j1)(y),
%\\
%&w^2(y)=\sum_{j=0}^{\infty}w_{2j+1}^2(y)=\sum_{j=0}^\infty \int^y_{-\infty} e^{\varphi(y)-\varphi(x)}q_3(x) (S^j 1)(x)\dx.
%\end{align*}} 
We are going to   solve \eqref{RenormalLax} rigorously and define the renormalised transmission coefficient $\Tc^{-1}$ in Subsection \ref{subs:Holo}.
Before that we give some preliminary estimates for $q, q_2, q_3, q_4$ and $w_{n}$.

  \subsubsection{Preliminary estimates}\label{subs:wn}
We first claim the following $L^\infty$-estimate:
\begin{equation}\label{qLinfty}\begin{split}
&\|q\|_{L^\infty}\lesssim 1+\tau^{\frac12} \||q|^2-1\|_{l^\infty_{\tau} DU^2} ^{\frac12}
+\|q'\|_{l^\infty_{\tau} DU^2},
\quad\forall\tau>0,
\\
&\hbox{ and in particular, }\|q\|_{L^\infty}\lesssim 1+\|\bq\|_{l^\infty_2DU^2}
\lesssim 1+E^{s}(q).
\end{split}\end{equation}
Indeed, we notice that by the partition of unit \eqref{unity}, for any $\tau>0$, at each point $x\in \R$, there exists $k\in\Z$ such that the  function $\chi_{\tau,k}$ or $\bar\chi_{\tau,k}$ with $\bar\chi=\chi(\cdot/2)$ taking value $1$ at point $x$, and thus
$\|q\|_{L^\infty}\leq \sup_k \|\chi_{\tau,k}q\|_{L^\infty}+\sup_k\|\bar\chi_{\tau,k}q\|_{L^\infty}$. By fundamental theorem of calculus: 
$(\chi_{\tau,k}q)(x)=\int^x_{\frac{k-1}{\tau}}  \chi_{\tau,k}'(y)q(y)\dy +\int^x_{\frac{k-1}{\tau}}  \chi_{\tau,k}q' \dy$, 
we  derive \eqref{qLinfty} from H\"older's inequality and \eqref{product} as follows:
{\small\begin{align*}
  |(\chi_{\tau,k}q)(x)|
&  \leq  \Bigl(\int^x_{\frac{k-1}{\tau}} | \chi_{\tau,k}'| |q|^2\dy\Bigr)^{\frac12}
\Bigl(\int^x_{\frac{k-1}{\tau}} | \chi_{\tau,k}'| \dy\Bigr)^{\frac12}
 +\bigl|\int^x_{\frac{k-1}{\tau}}  \chi_{\tau,k}q' \dy\bigr|
 \\
&\lesssim  \Bigl( \int^x_{\frac{k-1}{\tau}} | \chi_{\tau,k}'|  \,(|q|^2-1)\dy
+\int^x_{\frac{k-1}{\tau}} | \chi_{\tau,k}'| \dy\Bigr)^{\frac12}
 +\|q'\|_{l^\infty_{\tau} DU^2}
 \\
&\lesssim  (\tau\||q|^2-1\|_{l^\infty_{\tau} DU^2}+1)^{\frac12}+\|q'\|_{l^\infty_{\tau} DU^2}.
\end{align*}}

By use of the fact \eqref{BoundZeta} and \eqref{qLinfty} above, we derive immediately from the estimates in \eqref{product} the following   estimates for $q$ and $q_2, q_3, q_4$ defined in \eqref{q1234}:
\begin{lem}\label{Lem:cC0}
Let $(\lambda,z)\in\cR$, $\zeta=\lambda+z\in\cU$, 
with $\tau=2\Im z>0$, $\omega=\Im\zeta>0$.
Let  $q\in X^s$, $s>\frac12$, $\bq=(|q|^2-1, q')\in H^{s-1}\hookrightarrow l^2_1 DU^2$ and $ q_2, q_3, q_4$ be defined in \eqref{q1234}.
Then there exists a constant $\cC_0$, depending only on (for any fixed $\tau_1>0$)
\begin{equation}\label{C0,tau}
\omega^{-1},
 \,  \tau_1\||q|^2-1\|_{l^\infty_{\tau_1} DU^2},
\,  \|q'\|_{l^\infty_{\tau_1} DU^2},
\,\|q'\|_{l^\infty_\tau DU^2}
\end{equation}
such that
\begin{equation*}\begin{split} 
&\|q\|_{l^\infty_\tau V^2}\leq \cC_0,
\quad \|q_\kappa\|_{l^p_\tau DU^2}  
\leq \cC_0 \|\bq\|_{l^p_\tau DU^2},
\quad \kappa=2,3,4.
\end{split}\end{equation*}  
\end{lem}

%%%%%%%%%%
We are going to study the functions $w_{2j}^1(t), w_{2j+1}^2(t)$ in \eqref{SolveJost1}.  
For notational simplicity we first introduce the function 
{\small$$
 \tilde\varphi=\varphi-2i\Re z x=-\tau x+\int^x_0 q_4,\quad \tau=2\Im z>0.
$$
If  $\bq\in l^\infty_\tau DU^2$, then by Lemma \ref{Lem:cC0} (with a possibly larger $\cC_0$): 
%{\color{red} Do we really use the series? It should suffice to use the Lipschitz constant in the range of the argument of the exponential}
\begin{equation}\label{q4}\begin{split}
&\Bigl\|\chi_{\tau,k }e^{\int^y_{\frac k\tau} q_4 \dx}\Bigr\|_{ V^2}
\lesssim 
%\sum_{\ell=0}^\infty \frac{1}{\ell!} 
%\Bigl\|  \tilde\chi_{\tau,k}\int^y_{\frac k\tau}q_4\Bigr\|_{  V^2}^\ell
%\leq 
%\sum_{\ell=0}^\infty \frac{1}{\ell!} \bigl(  C \|q_4\|_{l^\infty_\tau DU^2}\bigr)^\ell
%\leq 
e^{\cC_0\|\bq\|_{l^\infty_\tau DU^2}},  
\quad
 \Bigl| e^{\int^{\frac k\tau}_{\frac{k'}{\tau}} q_4}\Bigr|
\lesssim 
%\sum_{\ell=0}^\infty \frac{1}{\ell!} \Bigl\|  \int^{\frac k\tau}_{\frac {k'}\tau} q_4\Bigr\|_{V^2}^\ell
%\leq \sum_{\ell=0}^\infty \frac{1}{\ell!} \bigl(  C(k-k') \|q_4\|_{l^\infty_\tau DU^2}\bigr)^\ell
%\leq 
e^{\cC_0\|\bq\|_{l^\infty_\tau DU^2}(k-k')}.
\end{split}\end{equation}
%where we recalled  the partition of unity \eqref{unity} and took $\tilde \chi=\chi(\frac1{12}\cdot)\in C^\infty_0(\R^d)$. }

Therefore we have the following properties for the functions $w_{2j}(t), w_{2j+1}(t)$:
\begin{lem}\label{Lem:wn} 
Assume the same hypothesis as in Lemma \ref{Lem:cC0} and  (with a possibly larger $\cC_0$ which depends only on \eqref{C0,tau}) 
\begin{equation}\label{small:cC0}
\|\bq\|_{l^\infty_\tau DU^2}
\leq \frac{1}{2\cC_0}.
\end{equation}

%{\color{red} This needs a variant for $s\in [0,\frac12]$.
%\begin{equation}\label{small:cC0}
%\|\Vert \langle D^2 +\tau^2\rangle^{-1/2}   \bq   \Vert_{L^2} 
%\leq   \frac{\tau^{\frac12} }{2\cC_0}.
%\end{equation}
%I am guessing the dependence on $\tau$.} 

Then the functions $w_{2j}, w_{2j+1}$ given in \eqref{SolveJost1} are well-defined, depending analytically on $\bq\in l^2_\tau DU^2$ and satisfying  the following estimates (with an universal constant $C$):
\begin{equation*}\label{est:wn}\begin{split}
&\|w_{2j}\|_{U^2}
\leq  \bigl(C\frac{|\Re z|+\tau}{\tau}\bigr)^{j} 
\bigl( \|q_2\|_{l^2_\tau DU^2}\|q_3\|_{l^2_\tau DU^2}\bigr)^{j} ,
%e^{\cC_0\|\bq\|_{l^\infty_\tau DU^2}},
\\
& \|w_{2j+1}\|_{L^\infty}
\leq \bigl(C\frac{|\Re z|+\tau}{\tau}\bigr)^{j+3/2} 
\bigl( \|q_2\|_{l^2_\tau DU^2}\|q_3\|_{l^2_\tau DU^2}\bigr)^{j}
\|q_3\|_{l^\infty_\tau DU^2} .
%e^{\cC_0\|\bq\|_{l^\infty_\tau DU^2}}.
\end{split}\end{equation*}  
\end{lem}
\begin{proof}
As $w_{2j}^1(t)=(S^j 1)(t)$,   we consider 
{\small\begin{align*}
&\|S f\|_{U^2}
 =\Bigl\| \int_{x<y<t} \bigl( e^{2i\Re z y} q_2(y)\bigr) e^{\tilde\varphi(y)-\tilde\varphi(x)} \bigl( e^{-2i\Re z x} (q_3 f)(x)\bigr) \dx\dy\Bigr\|_{U^2_t}
\\
&=\Bigl\|  \bigl( e^{2i\Re z y} q_2(y)\bigr) \int^y_{-\infty} e^{\tilde\varphi(y)-\tilde\varphi(x)} \bigl( e^{-2i\Re z x} (q_3 f)(x)\bigr) \dx\Bigr\|_{DU^2_y}
\\
&\leq\sum_{k} \Bigl\|  \bigl( e^{2i\Re z y}\chi_{\tau,k}(y) q_2(y)\bigr) \int^y_{y-\frac3\tau} e^{\tilde\varphi(y)-\tilde\varphi(x)}
\bigl(e^{-2i\Re zx}\tilde\chi_{\tau,k}(x) (q_3 f)(x)\bigr)\dx\Bigr\|_{DU^2_y}
\\
&+\sum_{k} \Bigl\|  \bigl( e^{2i\Re z y}\chi_{\tau,k}(y) q_2(y)\bigr) \int^{y-\frac3\tau}_{-\infty} \sum_{k'\leq k-1}e^{\tilde\varphi(y)-\tilde\varphi(x)}
\bigl(e^{-2i\Re zx}\chi_{\tau,k'}(x) (q_3 f)(x)\bigr)\dx\Bigr\|_{DU^2_y}.
\end{align*}
In the above, the first part on the righthand side can be bounded by (recalling \eqref{product})
\begin{align*}
&\sum_{k}\Bigl\| e^{2i\Re z y}\chi_{\tau,k}(y) q_2(y)   e^{\tilde\varphi(y)-\tilde\varphi(\frac k\tau)}\Bigr\|_{DU^2_y}
\Bigl\|  \int^y_{y-\frac3\tau} e^{\tilde\varphi(\frac k\tau)-\tilde\varphi(x)}
\bigl(e^{-2i\Re zx}\tilde\chi_{\tau,k}(x) (q_3 f)(x)\bigr)\dx\Bigr\|_{V^2_y} 
\\
&\lesssim \sum_k \bigl\|e^{2i\Re z\cdot}\chi_{\tau,k}q_2\bigr\|_{DU^2}
 \bigl\|\tilde\chi_{\tau,k}e^{\tilde\varphi(y)-\tilde\varphi(\frac k\tau)}\bigr\|_{V^2}
 \bigl\|\tilde\chi_{\tau,k}e^{\tilde\varphi(\frac k\tau )-\tilde\varphi(x)}\bigr\|_{V^2}
  \bigl\| e^{-2i\Re z\cdot}\tilde\chi_{\tau,k}(q_3 f)\bigr\|_{DU^2}
 \\
 &\lesssim \bigl(\frac{|\Re z|+\tau}{\tau}\bigr) \|\chi_{\tau,k}q_2\|_{\ell^2_k DU^2}
 \|\tilde\chi_{\tau,k}(q_3 f)\|_{\ell^2_k DU^2}
 \Bigl\|\tilde\chi_{\tau,k} e^{\pm\int^x_{\frac k\tau}q_4}\Bigr\|_{\ell^\infty_k V^2}^2,
\end{align*}
where we have used \eqref{PhaseShift} and $\|\tilde\chi_{\tau,k}e^{-\tau(\cdot-\frac k\tau)}\|_{V^2}\lesssim 1$,
and the second part on the righthand side can be bounded by
\begin{align*}
&\sum_{k}\Bigl\| e^{2i\Re z y}\chi_{\tau,k}(y) q_2(y)   e^{\tilde\varphi(y)-\tilde\varphi(\frac k\tau)}\Bigr\|_{DU^2_y}
\\
&\qquad\times\sum_{k'\leq k-1} 
 e^{\tilde\varphi(\frac k\tau)-\tilde\varphi(\frac{k'}{\tau})}
\Bigl\| \int^{y-\frac3\tau}_{-\infty} e^{\tilde\varphi(\frac {k'}\tau)-\tilde\varphi(x)}
\bigl(e^{-2i\Re zx}(\chi_{\tau,k'}q_3 f)(x)\bigr)\dx\Bigr\|_{V^2_y} 
\\
&\lesssim\sum_k \sum_{ k'\leq k-1} 
\bigl\|e^{2i\Re z\cdot}\chi_{\tau,k}q_2\bigr\|_{DU^2}
  \bigl\| e^{-2i\Re z\cdot}\chi_{\tau,k'}q_3 f\bigr\|_{DU^2} 
  \Bigl\|\tilde\chi_{\tau,k} e^{\pm\int^x_{\frac k\tau}q_4}\Bigr\|_{V^2}^2
  e^{-(k-k')} \Bigl| e^{\int^{\frac k\tau}_{\frac{k'}{\tau}}q_4}\Bigr|.
\end{align*} 
Under the smallness condition \eqref{small:cC0}, by use of the inequality \eqref{q4}, we derive 
\begin{align*}
\|S\|_{V^2\mapsto U^2}
\lesssim \frac{|\Re z|+\tau}{\tau} \|q_2\|_{l^2_\tau DU^2}
 \|q_3\|_{l^2_\tau DU^2} e^{\cC_0\|\bq\|_{l^\infty_\tau DU^2}}
 \leq  \frac{2(|\Re z|+\tau)}{\tau} \|q_2\|_{l^2_\tau DU^2}
 \|q_3\|_{l^2_\tau DU^2},
\end{align*}
and hence  Lemma \ref{Lem:wn} holds for $w_{2j}=S^j1$.
}

Similarly we consider the operator $S_1$ (noticing $\bigl\|\chi_{\tau,k}e^{2i\Re z t}\bigr\|_{V^2}\lesssim 1+|\Re z|/\tau$):
{\small
\begin{align*}
&\|S_1f\|_{L^\infty}
\lesssim 
\sup_k \bigl\|\chi_{\tau,k}e^{2i\Re z t}\bigr\|_{V^2} 
\bigl\|\tilde\chi_{\tau,k} e^{-2i\Re zx}q_3f\bigr\|_{DU^2}
\Bigl\|\tilde\chi_{\tau,k} e^{\pm\int^x_{\frac k\tau}q_4}\Bigr\|_{V^2}^2
\\
&+\sup_k\sum_{k'\leq k-1}
 \bigl\|\chi_{\tau,k}e^{2i\Re z t}\bigr\|_{V^2} 
\bigl\|\chi_{\tau,k} e^{-2i\Re zx}q_3f\bigr\|_{DU^2}
\Bigl\|\tilde\chi_{\tau,k} e^{\pm\int^x_{\frac k\tau}q_4}\Bigr\|_{ V^2}^2
e^{-(k-k')} \Bigl| e^{\int^{\frac k\tau}_{\frac{k'}{\tau}}q_4}\Bigr|
\\
&\lesssim (\frac{|\Re z|+\tau}{\tau})^{\frac32} \|q_3\|_{l^\infty_\tau DU^2}\|f\|_{V^2}
e^{\cC_0\|\bq\|_{l^\infty_\tau DU^2}},
\end{align*}
and hence Lemma \ref{Lem:wn} holds for $w_{2j+1}=S_1S^j1$.} 
 \end{proof}

%%%%%%%%%%
\subsubsection{The renormalised transmission coefficient $\Tc^{-1}$ for  $q\in X^s$}\label{subs:Holo}  
In order to emphasize the dependence of the renormalised transmission coefficient $\Tc^{-1}$ on $q$, we will denote $\Tc^{-1}(\lambda)=\Tc^{-1}(\lambda;q)$ in this subsection. 
\begin{prop}\label{Prop:Tr}
Let $q\in X^s$, $s>\frac12$, then the renormalised Lax equation \eqref{RenormalLax} has a unique solution $w\in L^\infty$.
 
We define the renormalised transmission coefficient as
$$\Tc^{-1}(\lambda)=e^{\Phi}\lim_{x\rightarrow\infty}w^1(x),$$ 
where $\Phi= -\frac{i}{2z}\int_{\R}\frac{(|q|^2-1)^2}{|q|^2-\zeta^2}\dx 
+\frac1{2z\zeta}\int_{\R} \frac{\bar q q'(|q|^2-1)}{(|q|^2-\zeta^2)} \dx$ is given in \eqref{Phi}.
Then 
\begin{itemize}
\item $\Tc^{-1}(\lambda; q)$ is a well-defined holomorphic function in $(\lambda,z)\in\cR$ and depends analytically on $q\in X^s$ with respect to the analytic structure given in Theorem \ref{thm:analytic} below;

\item When   $q-1\in\cS$, the   relation \eqref{lnTc}:  $\Tc^{-1}(\lambda)=e^{ -i\cM(2z)^{-1}-i \cP(2z\zeta)^{-1}}T^{-1}(\lambda)$, 
the properties \eqref{Tc1}: $|\Tc^{-1}(\lambda)|\geq 1\hbox{ if }\lambda\in\Ic=(-\infty,-1]\cup[1,\infty)$,
 $\Tc^{-1}\rightarrow 1 \hbox{ as }|\lambda|\rightarrow \infty,$ and the asymptotic  expansions \eqref{Expansion:Tc}, \eqref{Expansion:lnT} all hold true;

\item Let $q(t,x)\in \cC(I; X^s)$ be a solution of the Gross-Pitaevskii equation  (in Definition \ref{def}), then $\Tc^{-1}(\lambda; q(t))$ is conserved  on the existence time interval $I$.

\item  
$\Re {\Tc}^{-1}(\lambda)=\Re \Tc^{-1}(\bar\lambda)
 \hbox{ if } (\lambda,z)=(i\sigma,\pm i\frac\tau2)\in \cR,\,\tau\geq2,
 \,\sigma=\sqrt{\frac{\tau^2}4-1}$.
\end{itemize} 
\end{prop}
\begin{proof}
\noindent\textbf{Step 1. Resolution of \eqref{RenormalLax}.}\\
If $q\in X^s$, $s>\frac12$ and there are  the points $(\lambda,z)\in\cR$   such that $\tau=2\Im z\geq 2$ and the following  smallness condition holds (with a possibly larger $\cC_0$ than the ones in Lemmas \ref{Lem:cC0} and \ref{Lem:wn})
\begin{equation}\label{small:cC0,z}
\cC_0\bigl(\frac{|\Re z|+\tau}{\tau}\bigr)^{\frac 12}  \|\bq\|_{l^2_\tau DU^2}
 \leq \frac12,
\end{equation}
 we have by Lemmas \ref{Lem:cC0} and \ref{Lem:wn}  that
$$
\|w_n\|_{L^\infty}\leq \frac{|\Re z|+\tau}{\tau} \bigl( \cC_0 \bigl(\frac{|\Re z|+\tau}{\tau}\bigr)^{\frac12} \|\bq\|_{l^2_\tau DU^2}\bigr)^n.
$$
Hence \eqref{RenormalLax} (or equivalently \eqref{SolveJost1}) has a unique solution $w=\sum_{n\geq 0}w_n\in L^\infty$, depending analytically on $\bq\in l^2_\tau DU^2$.

For general $q\in X^s$, $s>\frac12$ with $E^{s}(q)<\infty$,   for any fixed point $(\lambda,z)\in\cR$,   we can take two points $a_0, a_1\in\R$ such that the smallness condition \eqref{small:cC0,z} holds for $q|_{(-\infty,a_0]\cup[a_1,\infty)}$, by virtue of the embedding $H^{s-1}\hookrightarrow l^2_1DU^2$, $s>\frac12$.
We solve  \eqref{RenormalLax} as follows:
\begin{itemize}
\item The above analysis implies that we can solve  \eqref{RenormalLax}  until the point $a_0$: $w(a_0)=b_0$;

\item Recalling the change of variables $u\mapsto w$ in \eqref{u-w} which renormalises the Lax equation \eqref{Lax} to \eqref{ODE}, we do the  change of variables $w\mapsto \underline{u}$  to solve the Lax equation \eqref{Lax} with the following initial data at $a_0$
$$\underline{u}(a_0)=e^{iz(a_1-a_0) +\int^{a_1}_{a_0} q_1 \dx}
\frac{1}{|q(a_0)|^2-\zeta^2}
\left(\begin{matrix}-i\zeta&q(a_0)\\ \bar q(a_0)&i\zeta\end{matrix}\right)b_0$$  
  until the point $a_1$: $\underline{u}(a_1)$. This is possible since $q\in H^s([a_0, a_1])$, $s>\frac12$;
  
  \item We  finally  solve again  \eqref{RenormalLax} with the initial data $b_1=\left(\begin{matrix}-i\zeta&q(a_1)\\ \bar q(a_1)&i\zeta\end{matrix}\right)\underline{u}(a_1)$ on the semiline $[a_1,\infty)$.
More precisely, we take $\tilde w=\begin{pmatrix}0\\ e^{2iz(x-a_1)}b_1^2\end{pmatrix}\Big|_{[a_1,\infty)}$ such that $\dot w=w-\tilde w$ satisfies
  \begin{align*}
  \dot w_x=\left(\begin{matrix} 0 &q_2\\q_3&q_4+2iz \end{matrix}\right)\dot w
  +b_1^2 e^{2iz(x-a_1)} 
  \begin{pmatrix} q_2|_{[a_1,\infty)} \\ q_4|_{[a_1,\infty)}\end{pmatrix},
  \quad \dot w|_{a_1}=\begin{pmatrix} b_1^1\\0\end{pmatrix}.
  \end{align*}
    Then under the smallness condition \eqref{small:cC0,z} for $q|_{[a_1,\infty)}$, we follow exactly Lemma \ref{Lem:wn}
    \footnote{
  We can solve the ODE for $\dot w$ by the following iterative procedure:
  \begin{align*}
  \dot w=\sum_{n=0}^\infty \dot w_n,
  \quad
   \dot w_{n+1}=\begin{pmatrix}\int^t_{a_1} q_2 \dot w_n^2\dx
  \\
  \int^t_{a_1} e^{\varphi(t)-\varphi(x)}q_3 \dot w_n^1 \dx\end{pmatrix},
  \quad 
  \dot w_0=\begin{pmatrix} b_1^1+b_1^2\int^t_{a_1}e^{2iz(x-a_1)}q_2\dx
  \\
  b_1^2\int^t_{a_1} e^{\varphi(t)-\varphi(x)+2iz(x-a_1)} q_4\dx\end{pmatrix}.
  \end{align*}
Although $\dot w_0^2\neq 0$, there is an exponential decay in the integrand and the same estimates as in Lemma \ref{Lem:wn} imply the well-definedness of $\dot w_n$ and   $\dot w$. For example, we can  control straightforward
\begin{align*}
&\dot w_0=\begin{pmatrix} b_1^1+b_1^2\int_{a_1<x_1<t}e^{2iz(x_1-a_1)}q_2 dx_1
  \\
  b_1^2\int_{a_1<x_1<t} e^{2iz(t-a_1)+\int^t_{x_1}q_4} q_4 dx_1\end{pmatrix},
  \\
  &
  \dot w_1=\begin{pmatrix} b_1^2\int_{a_1<x_1<y_1<t} e^{2iz(y_1-a_1)+\int^{y_1}_{x_1}q_4} q_2(y_1)q_4(x_1) dx_1  dy_1
  \\
 b_1^1\int_{a_1<x_1<t}e^{\varphi(t)-\varphi(x_1)}q_3
 +b_1^2\int_{a_1<x_1<y_1<t}e^{\varphi(t)-\varphi(y_1)+2iz(x_1-a_1)}q_2(x_1) q_3(y_1)dx_1dy_1 \end{pmatrix}
\end{align*}
in terms of $\|\bq\|_{l^2_\tau DU^2}$.
}   to derive  the unique solution  $\dot w$.
 Hence the solution $w=\dot w+\tilde w\in L^\infty$ with $w^1\in V^2$ exists. 
  \end{itemize}
% The solution $w\in L^\infty$ with $w^1\in V^2$ depends analytically on 
% $\bigl( \bq|_{(-\infty,a_0]}, q|_{[a_0,a_1]}$, $\bq|_{[a_1,\infty)}\bigr)\in 
% l^2_\tau DU^2 \times H^s([a_0,a_1])\times l^2_\tau DU^2$. 
  
  \smallbreak%%%%%%%%%%
  \noindent\textbf{Step 2. Well-definedness of $\Tc^{-1}$.}\\
  When $q\in X^s$, $s>\frac12$, then $\Phi<\infty$ by Lemma \ref{Lem:cC0}.
  We can define  the renormalised transmission coefficient $\Tc^{-1}(\lambda)=e^{\Phi}\lim_{x\rightarrow\infty}w^1(x)$, which is holomorphic on the Riemann surface $(\lambda, z)\in\cR$.

   If $q-1\in \cS$, then by virtue of $e^{\int^\infty_{-\infty}q_1}T^{-1}(\lambda)=\lim_{x\rightarrow\infty}w^1(x)$ in Subsection \ref{subsubs:T} and the equality \eqref{Integral:q1}: $\int^\infty_{-\infty}q_1=-\frac{i}{2z}\cM-\frac{i}{2z\zeta}\cP-\Phi$, we have the relation \eqref{lnTc}:  $\Tc^{-1}(\lambda)=e^{ -i\cM(2z)^{-1}-i \cP(2z\zeta)^{-1}
}T^{-1}(\lambda)$ and hence the asymptotic expansions \eqref{Expansion:Tc} and \eqref{Expansion:lnT} follow from Proposition \ref{Prop:Expan,T}.
The properties \eqref{Tc1}: $|\Tc^{-1}(\lambda)|\geq 1\hbox{ if }\lambda\in\Ic=(-\infty,-1]\cup[1,\infty)$, $\Tc^{-1}\rightarrow 1 \hbox{ as }|\lambda|\rightarrow \infty,$ and the symmetry $\Re {\Tc}^{-1}(\lambda)=\Re \Tc^{-1}(\bar\lambda)$
  for $(\lambda,z)=(i\sigma,\pm i\frac\tau2)\in \cR$ follow from the results in  Subsection \ref{subs:Jost}.

Fix $(\lambda,z)\in\cR$. 
Provided with the analytic structure of $X^s$ in Theorem \ref{thm:analytic} below, for any  neighbourhood $B_r^s(q)$ of $q$, we can choose $a_0, a_1$ (depending on $E^s(q), r$) such that the smallness condition \eqref{small:cC0,z} holds for $p|_{(-\infty, a_0]\cup [a_1,\infty)}$ for all $p\in B_r^s(q)$. 
Therefore the  corresponding solution $w_p$ for \eqref{ODE}  depends analytically on 
$\bigl( \mathbf{p}|_{(-\infty,a_0]}, \mathbf{p}|_{[a_0,a_1]}$, $\mathbf{p}|_{[a_1,\infty)}\bigr)\in 
 l^2_\tau DU^2 \times H^s([a_0,a_1])\times l^2_\tau DU^2$ in $B^s_r(q)$ and hence $\Tc^{-1}(\lambda;\cdot)$ depends analytically on $q\in X^s$.

\smallbreak%%%%%%%%
\noindent\textbf{Step 3. Conservation of $\Tc^{-1}$ by the Gross-Pitaevskii flow.}\\  
If initially  $q_0-1\in\cS$, then the Gross-Pitaevskii equation \eqref{GP} has a unique global-in-time solution  $q\in \cC(\R; Z^1)$ (see \eqref{Zhidkovk} for the definition of Zhidkov's space $Z^1$) by Zhidkov's well-posedness result.
By Faddeev-Takhtajan \cite{FT}, $(q-1)(t,\cdot)\in \cS$ and  $\cM, \cP, T^{-1}(\lambda)$ are all conserved by the Gross-Pitaevskii flow, 
and hence $\Tc^{-1}(\lambda)=e^{ -i\cM(2z)^{-1}-i \cP(2z\zeta)^{-1}
}T^{-1}(\lambda)$ is also conserved.

Now let $q\in \cC(I; X^s)$, $s>\frac12$ be the  solution    of  the Gross-Pitaevskii equation \eqref{GP} with the initial data $q_0\in X^s$ on the time interval $I$.
Then by the density result in Theorem \ref{thm:metric}, we take $\{q_{0,n}\}\subset 1+\cS$ such that $d^s(q_{0,n}, q_0)\rightarrow0$ as $n\rightarrow\infty$.
By the
% locally Lipschitz
 continuity of the  Gross-Pitaevskii flow in Theorem \ref{thm:lwp},  for all $t\in I$, the corresponding solutions $q_n, q$ satisfy $d^s(q_n(t), q(t))\rightarrow0$ and hence $\Tc^{-1}(q_n(t))\rightarrow\Tc^{-1}(q(t))$ as $n\rightarrow\infty$ by the analyticity  of $\Tc^{-1}(\lambda;\cdot)$ above.
The conservation of $\Tc^{-1}(q_n)$ along the Gross-Pitaevskii flow  implies the conservation of $\Tc^{-1}(q(t))$ on the existence time interval $I$. 
\end{proof}

%%%%%%
\subsection{Superharmonic function $G$ on the upper half-plane}\label{subsubs:G}
If $q-1\in \cS$, then   $G(z)$ defined in \eqref{G} is  a well-defined nonnegative superharmonic function on the $z$-upper half-plane. Indeed, as 
$|\Tc^{-1}(\lambda)|\geq 1\hbox{ if }\lambda\in\Ic=(-\infty,-1]\cup[1,\infty)$,
the trace of $G(z)$ on the real line is non negative:
{\small\begin{align*} 
 \mu(\xi)
=\frac 12\sum_{\pm} (2\xi)^2 \ln \bigl|\Tc^{-1}\bigr|\bigl(\pm \sqrt{\xi^2+1}\bigr)\geq 0, \quad \xi\in\R.
 \end{align*} }
 On the other hand, since   the meromorphic function $\Tc$  has only  simple poles $\lambda_m\in (-1,1)$, we can take a small enough neighborhood $\cV_m$ of $\lambda_m$ such that 
$\Tc(\lambda)=  A_0(\lambda)+\frac{A_1(\lambda)}{\lambda-\lambda_m}$ on $\cV_m$, 
with $ A_1\neq 0, A_0$  holomorphic functions on $\cV$. 
For $\lambda_m\neq 0$, $\lambda\in\cV_m$ and correspondingly for 
$z\in \cU_m=\{z\in\cU: (\lambda,z)\in\cR, \, \lambda\in \cV_m\}$, we can write (noticing $\lambda^2-\lambda_m^2=z^2-z_m^2$, $z_m=i\sqrt{1-\lambda_m^2}\in i(0,1]$){\small\begin{align*}
 \ln \Tc (\lambda)+  \ln \Tc (-\lambda)
&= \ln \bigl(A_0(\lambda)\,(\lambda-\lambda_m)+A_1(\lambda)\bigr)  + \ln \frac{\lambda+\lambda_m}{z+z_m}
+ \ln  \frac{1}{z-z_m}+\ln \Tc(-\lambda).
%\\
%&\quad
%+\ln \bigl(A_0(-\lambda)-\frac{A_1(-\lambda)}{\lambda+\lambda_m}\bigr).
\end{align*}} 
For $\lambda_m=0$, $z_m=i$, we can still write $\ln \Tc (\lambda)+  \ln \Tc (-\lambda)$ in $\cV_0$ as
{\small$$ 
 \ln \bigl(A_0(\lambda)\lambda+A_1(\lambda)\bigr) 
+ \ln \frac{1}{z+z_m}
+ \ln  \frac{1}{z-z_m}
+\ln \bigl(A_0(-\lambda)\,\lambda-A_1(-\lambda)\bigr).$$}
Hence  $-\Delta_zG=\Delta_z \Re \frac 12\sum_{\lambda=\pm \sqrt{1+z^2}}
\bigl( 4z^2 \ln \Tc(\lambda) \bigr)$ is a nonnegative measure \eqref{DeltaS} on $\cU$.
As $\Tc^{-1}\rightarrow 1$ as $|\lambda|\rightarrow \infty$, we derive $G\geq 0$ on the whole upper half plane by maximum principle.

Let $q\in X^s$, then by the density argument as in the last part of last subsection,  we deduce that   $G$ defined in \eqref{G} still satisfies $G\geq 0$, $-\Delta G\geq 0$ on the upper half plane.
Since the meromorphic function $\Tc(q)$ has  countably many simple poles 
$\{\lambda_m\}\subset \C\backslash \Ic$, the fact $-\Delta G=-\pi\sum(2z)^2\delta_{z=z_m}\geq 0$ implies $z_m\in i(0,1]$ and hence $\lambda_m\in (-1,1)$.
This completes the proof of \eqref{Tc:property} for general $q\in X^s$.
Theorem \ref{thm:Tc} follows from Proposition \ref{Prop:Tr}.

%%%%%%%%%%%%%%%%%%%%%%%%%%%%%%%%%%%%%%%%%%%%%%%%%%%%%%
%%%%%%%%%%%%%%%%%%%%%%%%%%%%%%%%%%%%%%%%%%%%%%%%%%%%%%

 \setcounter{equation}{0}%%%%%%%%%%%%%%%%%%%%%%%%%%%%%%%%%%
  \section{The energies} \label{sec:est} 
  
 In  this section we will formulate the  energies $\cE^{s}_{\tau}(q)$, $\tau\geq 2$ for $q\in X^s$, $s>\frac12$ in terms of the renormalised transmission coefficient $\Tc^{-1}(\lambda)$ defined  in   Theorem \ref{thm:Tc}:
 \begin{thm}\label{thm:Es}
 Let $q\in X^s$, $s>\frac12$.
 Let $\Tc^{-1}(\lambda)$ be the  renormalised transmission coefficient which is a holomorphic function on the Riemann surface $\cR\ni(\lambda,z)$ and has   countably many simple zeros $\{\lambda_m\}\subset(-1,1)$ given in Theorem \ref{thm:Tc}.
 Let  $G(z)=\frac12\sum_{\pm}\Re\bigl(4z^2\ln\Tc^{-1}(\pm\sqrt{z^2+1})\bigr)$ be the nonnegative superharmonic function  on the upper half plane $\cU$,  with  $-\Delta_zG=-\pi\sum_{m}(2z)^2\delta_{z=z_m}\geq 0$,
$z_m=i\sqrt{1-\lambda_m^2}\in i(0,1]$,  given by Theorem \ref{thm:Tc}.

Then for $N=[s-1]$, $G(i\tau/2)$  has the following finite expansion as $\tau\rightarrow\infty$ (recalling the notations $\cH^l$ in the asymptotic expansion \eqref{lnT:FT}):
{\small\begin{equation}\label{ExpanG}
G(\frac{i\tau}2)=\sum_{l=0}^{N} (-1)^l \cH^{2l+2}\tau^{-2l-1}+\cH^{>2N+2}(\frac{i\tau}2), 
 \quad \cH^{>2N+2}=o(\tau^{-2s+1}),
\end{equation}}
 such that the trace of $G$ on the real line  $\mu$ exists,  the measure $(1+\xi^2)^N\mu$ is finite
 and the following trace formula for $\cH^{2l+2}$, $0\leq l\leq N$ holds:
  \begin{equation}\label{cE2j}
 \cH^{2l+2} 
 = \frac{1}{\pi}\int_{\R}   \xi^{2l+2}\frac12\sum_{\pm}\ln|\Tc^{-1}(\pm\sqrt{\xi^2/4+1})| d\xi-\frac{1 }{2l+3}\sum_{m}   \Im (2z_m)^{2l+3}  .
 \end{equation}

% Moreover, whenever $(q,\tau_0)\in X^s\times[C,\infty)$ satisfying \eqref{tau0}, 
We define a family of  energy functionals $(\cE^{s}_{\tau'})_{\tau'\geq 2}: X^s\mapsto [0,\infty)$    as follows:
% \begin{equation}\label{cEsmall}\begin{split}
% \cE^{s}_{\tau_0} 
% =-\frac{2}{\pi}\sin(\pi(s-1))\int_{\tau_0}^\infty (\tau^2-\tau_0^2)^{s-1}  
%G(i\tau/2) \dtau,
%\hbox{ if }s\in (\frac12,1),
% \end{split}\end{equation}
 \begin{align}
  \cE^{s}_{\tau'} (q)
 = &-\frac{2}{\pi}\sin(\pi(s-1))
 \int_{\tau'}^\infty (\tau^2-\tau'^2)^{s-1}  
 \Bigl (G(i\frac \tau2) 
 - \sum_{l=0}^N(-1)^l  \cH^{2l+2} \tau^{-2l-1} \Bigr)  
\dtau\notag
\\
&+\sum_{l=0}^N \tau'^{2(s-1-l)}
 \begin{pmatrix}  s-1 \\l\end{pmatrix} \cH^{2l+2},
 \hbox{ with }N\leq s-1<N+1,\label{cElarge} 
 \end{align}  
 such that $\cE^{s}_{\tau'}$ is analytic in $q\in X^s$ and we have the following trace formula for $\cE^{s}_{\tau'}$:
\begin{equation}\label{cEtrace}\begin{split}
\cE^{s}_{\tau'}
&=\frac{1}{\pi}\int_{\R} (\xi^2+\tau'^2)^{s-1} d\mu_{G(\frac\cdot2)}
 -\sum_{m} \Im\int^{2z_m}_0 w^2 (w^2+\tau'^2)^{s-1} dw.
\end{split}\end{equation}
Then there exists a universal constant $C\geq2$ (depending only on $s$) such that
whenever $(q,\tau_0)\in X^s\times [C,\infty)$ satisfying  
  \begin{equation}\label{tau0}\begin{split}
\frac{1}{\tau_0}\|\bq\|_{l^2_{\tau_0}DU^2}< \frac{1}{2C},
\hbox{ with } \|\bq\|_{l^2_{\tau_0}DU^2}^2=  \||q|^2-1\|_{l^2_{\tau_0}DU^2}^2+\|q'\|_{l^2_{\tau_0}DU^2}^2,
 \end{split}\end{equation} 
  $\cE^{s}_{\tau_0}(q)$ is   equivalent to the square of the energy norm $(E^{s}_{\tau_0}(q))^2$ in the sense of \eqref{Equivalence}: 
\begin{equation}\label{Equiv}
|\cE^{s}_{\tau_0} -(E^{s}_{\tau_0})^2|
\leq \frac{C}{\tau_0}\|\bq\|_{l^2_{\tau_0}DU^2} (E^{s}_{\tau_0})^2. 
 \end{equation}  
%  %We also have the following trace formula for $\cE^{s}_{\tau_0}(q)$:
%\begin{align*}
%\cE^{s}_{\tau_0}(q)
%&= \frac{2}{\pi}\int_{\R}   \xi^{2}(\xi^2+\tau_0^2)^{s-1} 
% \ln \bigl|\Tc^{-1}(\sign(\xi) \sqrt{\xi^2+1} )\bigr| \dxi
% \\
%&\quad+\sum_{m}(-8z_m^2)\Im\int^{\zeta_m}_0 (\frac{z^2}{4}+\tau_0^2)^{s-1} \frac{2\lambda}{\zeta}\dzeta,
%\quad z^2=\frac 14(\zeta-\frac{1}{\zeta})^2,\quad \lambda=\frac 12(\zeta+\frac 1\zeta).
%\end{align*}
 \end{thm}

\subsection{The framework}\label{subs:fw}
We are going to introduce the assumptions and the notations which will be used throughout this section.

\subsubsection{Assumption}\label{subss:small}
We restrict ourselves on the imaginary axis in this section:
\begin{equation}\label{ImaginAxis}\begin{split}
&(\lambda,z)=(i\sigma,i\tau/2)\in\cR,
\quad \tau\geq\tau_0\geq 2,\quad\sigma=\sqrt{\tau^2/4-1}\in [\tau/2-1, \tau/2),
\\
&\zeta=\lambda+z=i\omega,\quad  \omega=\sigma+\tau/2\in[\tau/2,\tau).
\end{split}\end{equation}
Here $\tau_0\geq2$ is a constant (to  be chosen sufficiently large later), such that  
  the following  assumption   holds:
\begin{equation}\label{smallq}
\bigl| |q|^2-1\bigr|\leq \frac{1}{64}{\tau_0}^2,
\quad q\in X^s,\,s>\frac12.
\end{equation}

\subsubsection{Functions $q_2, q_3, q_4$}\label{subss:nota}
We evaluate the functions $q_2, q_3, q_4$ defined in \eqref{q1234} on the imaginary axis to arrive at
\begin{equation*}\label{qIm}\begin{split}
 &q_2=-\frac{1}{\omega^{-2}|q|^2+1}(\frac{1}{\omega}q')
 + \frac1\omega\frac{q}{\omega^{-2}|q|^2+1} \bigl(\frac{1}{\omega}(|q|^2-1)\bigr),
 \\
& q_3= \frac{1}{\omega^{-2}|q|^2+1}(\frac{1}{\omega}\bar q')+\frac{1}{\omega}\frac{\bar q}{\omega^{-2}|q|^2+1}\bigl(\frac{1}{\omega}(|q|^2-1)\bigr),
\\
&q_4= \frac{ -2 }{\omega^{-2}|q|^2+1}
\bigl(\frac{1}{\omega}(|q|^2-1)\bigr)
+\Bigl( \frac{\frac1\omega q }{\omega^{-2}|q|^2+1}(\frac{1}{\omega}\bar q')
-\frac{\frac1\omega  \bar q  }{\omega^{-2}|q|^2+1}(\frac{1}{\omega} q')\Bigr).
\end{split}\end{equation*}
Notice that $q_2, q_3, q_4$ are all linear combinations of the components in
\begin{equation}\label{Q}
Q=\frac1\tau (|q|^2-1, q', \bar q'),
\end{equation}
up to  coefficients being some  polynomials of the following form  
\begin{equation}\label{P}
P=P\Bigl(\frac{1}{\omega^{-2}|q|^2+1},  \frac1\tau q, \,\frac1\tau \bar q, 
\,\frac{1}{\tau^2}(|q|^2-1)\Bigr),
\end{equation}
and for notational simplicity,  we denote $O$ to be the following set:
\begin{equation}\label{O}
O:=\Bigl\{ P\cdot\frac1\tau (|q|^2-1), \, P\cdot \frac1\tau q', \, P\cdot\frac1\tau \bar q'
\,\Big|\, P \hbox{ is any polynomial of form \eqref{P}}\Bigr\}.
\end{equation}

Under the  assumption \eqref{smallq}, 
 we are going to estimate $\|q_\kappa\|_{l^2_\tau DU^2}$   as follows (similar but more accurate  as the estimates in Lemma \ref{Lem:cC0}):
{\small\begin{equation}\label{cq}\begin{split}
&\|q_\kappa\|_{l^2_\tau DU^2}\lesssim C_\tau c_\tau,\quad \kappa=2,3,4,
\quad\hbox{ with }  C_\tau =1+\frac1\tau\|q'\|_{l^\infty_{\tau} DU^2},
\\
&  c_{\tau}(q)=\frac{1}{\tau}\|\bq\|_{l^2_\tau DU^2(\R)}
=\frac{1}{\tau}\bigl( \bigl\| |q|^2-1\bigr\|_{l^2_\tau DU^2}^2+\bigl\| \d_x q\bigr\|_{l^2_\tau DU^2}^2\bigr)^{\frac12}:= \|Q\|_{l^2_\tau DU^2}.
\end{split}\end{equation} }
%\begin{equation}\label{tildeC0}
%  C_\tau =1+\frac1\tau\|q'\|_{l^\infty_{\tau} DU^2}.
%\end{equation}

\subsubsection{Asymptotic expansion  of $\ln\Tc^{-1}$ on the imaginary axis}\label{subss:SJ}
Recall the   aymptotic expansions of $\Tc^{-1}(\lambda)$ and $\ln\Tc^{-1}$ in Theorem \ref{thm:Tc}:
{\small\begin{align*}
&\Tc^{-1}(\lambda) 
 =e^{\Phi}\Bigl(1 +\sum_{j=1}^\infty T_{2j}\Bigr),
 \quad T_{2j}=\<XY>^j,
 \\
 & \ln \Tc^{-1}(\lambda) 
 =\Phi+T_2+\sum_{j=2}^\infty \tilde T_{2j},
 \, \Phi(\lambda):=-\frac{i}{2z}\int_{\R}\frac{(|q|^2-1)^2}{|q|^2-\zeta^2}\dx 
+\frac1{2z\zeta}\int_{\R} \frac{\bar q q'(|q|^2-1)}{(|q|^2-\zeta^2)} \dx,
\end{align*} }
where $\tilde T_{2j}$ is linear combination of \emph{connected} symbols $\<XBBY>_{2j}$  of degree $2j$.   
Recall the symbols in Subsection \ref{subs:notation}:  
{\small\begin{equation}\label{S}
\<XY>^j=\int_{x_1<y_1<\cdots<x_j<y_j}
\prod_{n=1}^j e^{\varphi(y_n)-\varphi(x_n)}q_3(x_n) q_2(y_n)
   \dx \dy,
\,\, \varphi(x)=-\tau x+\int^x_0 q_4, 
\end{equation}  
\begin{equation}\label{tildeS}\begin{split}  
&\<XBBY>_{2j}=\int_{t_1<\cdots<t_{2j}} 
\prod_{n=1}^{2j-1}e^{\delta_n(\varphi(t_{n+1})-\varphi(t_n))}
q_{\kappa_1}(t_1)\cdots q_{\kappa_{2j}}(t_{2j}) dt, 
\end{split}\end{equation}}
for some $\delta_n\in \{1,\cdots,j\}$ with $\delta_1=\delta_{2j-1}=1$ and
$\kappa_n\in \{2,3\}$ with $\kappa_1=3$, $\kappa_{2j}=2$. 
For notational simplicity, we also introduce the following symbol
 \begin{equation}\label{J}\begin{split} 
&\<XBY>= \int_{x<y} (e^{\varphi(y)-\varphi(x)}-e^{-\tau(y-x)}) q_2(y) q_3(x)\dx\dy.
\end{split}\end{equation} 

 We will rewrite  $(\Phi+T_2)$  in the asymptotic expansion  of $\ln\Tc^{-1}(\lambda)$ above   in Appendix \ref{appendix} as:
 $$\Phi+T_2=\tilde T_2+\tilde T_3,$$
 where $\tilde T_2, \tilde T_3$ identify the \emph{quadratic} and \emph{cubic} terms (in terms of  elements in the set $O$) in the expansion of $\ln\Tc^{-1}$ respectively.
 More precisely   we will prove in Appendix \ref{appendix}     that
\begin{lem}\label{Lem:IA}
The asymptotic expansion \eqref{Expansion:lnT} for the logarithm of the transmission coefficient $\ln\Tc^{-1}$ reads on the imaginary axis \eqref{ImaginAxis} as follows
\begin{equation}\label{Expansion:lnT,T2}
 \ln\Tc^{-1}(i\sigma)
 =\tilde T_2(i\sigma)+\tilde T_3(i\sigma)+\sum_{j\geq 2}\tilde T_{2j}(i\sigma), 
\end{equation} 
where $\tilde T_2$ identifies the quadratic part in the expansion of $\ln\Tc^{-1}(i\sigma)$:
{\small\begin{equation}\label{T2} \begin{split}
  \tilde T_2(i\sigma)
 =& \frac{-1}{\tau^2}\int_{x<y}e^{-\tau(y-x)}
\Bigl( (|q|^2-1)(y)(|q|^2-1)(x)+ q'(y)\bar q'(x)\Bigr) \dx\dy
\\
&  -i\frac{\tau+2\omega}{\tau^3\omega^2}\int_{\R}\Im(q'\bar q)(|q|^2-1)\dx 
\\
&+\frac{1}{\tau^3\omega} 
\int_{x<y}e^{-\tau(y-x)}\Bigl( q'(y) 
 \bar q'(x)- \bar q'(y)q'(x)\Bigr)\dx\dy,
\end{split}\end{equation}  }
   $\tilde T_3(i\sigma)$ identifies the cubic part and reads as  linear combination  of finite integrals of the following type
{\small\begin{equation}\label{ReT3}\begin{split}
&\int_{\R}  \bigl( \frac{|q|^2-1}{\tau^2}\bigr)^2\, h_1\dx,
\quad\<XBY>,
\quad \int_{x<y\hbox{ or }x>y}  e^{-\tau|y-x|} \bigl( \frac{|q|^2-1}{\tau^2}h_1\bigr)(y)h_2(x)\dx\dy,
\\
& % \int_{x<y} e^{-\tau(y-x)}  h_1(y) \bigl( \frac1{\tau^2}(|q|^2-1)h_2\bigr)(x)\dx\dy,
 \int_{x<y}e^{-\tau(y-x)} h_1(y) \int^y_x \bigl(\frac{q'\hbox{ or }\bar q'}{\tau}\bigr)(m)\dm h_2(x)\dx\dy,
 \quad h_1, h_2 \in O,
\end{split}\end{equation}  }
and $\tilde T_{2j}$, $j\geq 2$ remains the same linear combination of integrals 
$\<XBBY>_{2j}$.
\end{lem} 
\begin{rmk}\label{rmk:a4}
If $q-1\in\cS$, then by integration by parts
$\int_{x<y}e^{-\tau(y-x)}f(y)g(x)\dx\dy
=\frac1\tau\int_{\R} fg-\frac1\tau\int_{x<y}e^{-\tau(y-x)} f(y)g'(x)\dx\dy$ 
 we can expand $\tilde T_2(i\sigma)$ as
 {\small
\begin{align*}
\tilde T_2(i\sigma)=-\frac1{\tau^3}\int_{\R}\bigl((|q|^2-1)^2+|q'|^2 \bigr)\dx
+\frac{1}{\tau^4}\int_{\R}\bigl(q'\bar q''-3i\Im(q'\bar q)(|q|^2-1)\bigr) \dx
+O(\frac1{\tau^{5}}),
\end{align*}}
while $\tilde T_3(i\sigma)=O(\frac1{\tau^5})$, $T_{2j}(i\sigma)=O(\frac{1}{\tau^{3j}})$, $\tilde T_{2j}(i\sigma)=O(\frac{1}{\tau^{4j-1}})$, 
as $\tau\rightarrow\infty$.
Recalling $\ln T^{-1}(\lambda)=i\cM(2z)^{-1}+i\cP(2z\zeta)^{-1}+\ln\Tc^{-1}(\lambda)$, we derive the finite expansion until the fourth-order for $\ln T^{-1}(i\sigma)=\frac1\tau\cM-i\frac{1}{\tau^2}\cP-\frac{1}{\tau^3}\cH^2+i\frac{1}{\tau^4}\cH^3+O(\frac{1}{\tau^5})$, as $\tau\rightarrow\infty$, which can be  compared with \eqref{lnT:FT}.
\end{rmk}

\subsection{Trace formula and the organisation of the section}
\label{subs:super}
 In this section we will recall the trace formula (from e.g.  \cite{ABR,KT})   for   the nonnegative superharmonic functions $G$ on the upper half-plane $\cU$ given in Theorem \ref{thm:Tc}.
As a consequence we derive the formulation of the energy norm $E^{s}_{\tau_0}(q)$ in terms of the quadratic term $\tilde T_2(i\sigma)$, as well as the equivalence relation \eqref{Equiv} between $E^{s}_{\tau_0}(q)$ and the energy  $\cE^{s}_{\tau_0}(q)$ (defined in \eqref{cElarge}).

\subsubsection{Trace formula}
Recall  the superharmonic function 
$$G(z)=\frac12\sum_{\pm}\Re\bigl(4z^2\ln\Tc^{-1}(\pm\sqrt{z^2+1})\bigr)$$
defined on the upper half-plane $\cU$ in Theorem \ref{thm:Tc}, with the nonnegative Radon measure $\mu$ as the trace of $G$ on the real line $\R$ and the nonnegative Radon measure $\nu=-\Delta_z G=-\pi\sum_m (2z)^2\delta_{z=z_m}$ on the upper half-plane $\cU$, $z_m\in i(0,1]$. We define
\[ \Xi_s(z)=\Im \int^z_0 w^2(w^2+\tau_0)^s dw  \]
for $z$ in the upper half plane.

\begin{lem}\label{prop:G} The followings hold true:
\begin{itemize}
\item Representation  of $G$ through  $\mu,\nu$.\\ 
The function $G$ can be represented in terms  of the Poisson kernel and the fundamental solution of the Laplace equation as follows:
\begin{equation}\label{SupHarm}
G(z)=\frac{1}{\pi}\int_{\R} \frac{\Im z}{|z-\xi|^2} d\mu(\xi)
+\frac{1}{2\pi}\int_{\cU} \ln\Bigl| \frac{z-\bar \zeta}{z-\zeta} \Bigr| d\nu(\zeta).
\end{equation}

\item Expansion of $G$ at $+i\infty$.\\ 
  If  two measures $(1+|\xi|^2)^N\mu$, $\Im z(1+|z|^2)^N\nu$ 
    are finite, then  we have the following precise expansion of $G$ at $+i\infty$:
\begin{equation}\label{ExpanG}
G(\frac{i\tau}2)=\sum_{l=0}^N(-1)^l \cH^{2l+2} \tau^{-2l-1}
+ \cH^{>2N+2},
\,\, \cH^{> 2N+2}=o(\tau^{-2N-1}),
\end{equation} 
where $\cH^{2l+2}$ is given in \eqref{cE2j} and 
{\small\begin{align*} 
&\cH^{>2N+2}=(-1)^{N+1}
\Bigl( \frac{1}{\pi} \int_{\R}\frac{\xi^{2N+2}}{\tau^2+\xi^2} d\mu_{G(\frac\cdot2)} 
- \sum_m \Im \int^{2z_m}_0 \frac{w^{2N+4}}{\tau^2+w^2} dw  \Bigr)\tau^{-2N-1}.
\end{align*}}
\item Trace formula of $G$.\\
%Let $G=H+\sum_{m}\Re\bigl (z^2\ln(z-z_m)^{-1}\bigr)$, where $H$ is harmonic on the upper half plane and $z_m\in i\R^+$.
% connecting $G|_{i[\tau_0,\infty)}$ and $\mu, \nu$ if $\nu$ is supported on the strip $\{z\,|\,0<\Im z<\tau_0\}$.\\
Let $N=[s-1]$. Then the following trace formula holds
{\small\begin{equation}\label{TracForm}
\begin{split}
&-\frac{2\sin(\pi (s-1))}{\pi}\int^{\infty}_{\tau_0} (\tau^2-\tau_0^2)^{s-1}
\Bigl( G(\frac{i\tau}2)-\sum_{l=0}^N (-1)^l \cH^{2l+2} \tau^{-2l-1}\Bigr) d\tau
\\
& 
+\sum_{l=0}^N \tau_0^{2(s-1-l)}\begin{pmatrix}  s-1 \\l\end{pmatrix} \cH^{2l+2}
 = \frac1\pi\int_{\R} (\xi^2+\tau_0^2)^{s-1}  d\mu_{G(\cdot/2)} -\sum_m \Xi_{s-1}(2z_m),
\end{split}
\end{equation}}
%where
%\small{
%\begin{align*}
%G_{2l}=\frac1\pi\int_{\R}\xi^{2l}G(\xi)d\xi 
%\end{align*}} $\Xi_s(\zeta):=\Im\int^\zeta_0 (w^2+\tau_0^2)^s dw$,
 whenever either side is finite. 
\end{itemize}
\end{lem}
We have the following description of the energy  norm by the trace formula:
\begin{lem} \label{col:T2} 
Let $q\in X^s$, $s>\frac12$, then the energy norm defined in \eqref{Es}:
{\small\begin{align*}
(E^{s}_{\tau_0}(q) )^2
 = \bigl\| \bq\bigr\|_{H^{s-1}_{\tau_0}}^2 
 =\int_{\R}(\xi^2+\tau_0^2)^{s-1}\bigl(|\widehat{|q|^2-1}|^2+|\widehat{q'}|^2\bigr)(\xi)\dxi
\end{align*}} 
can be identified  as the integral of the quadratic term $\tilde T_2(i\sigma)$ in \eqref{T2} as 
\begin{equation}\label{Estau0} \begin{split}
(E^{s}_{\tau_0}(q))^2=&-\frac{2}{\pi}\sin(\pi (s-1))\int_{\tau_0}^\infty (\tau^2-\tau_0^2)^{s-1} \cH_2^{>2N+2}(i\sigma)\dtau 
 \\
& +\sum_{l=0}^{N} \begin{pmatrix}  s-1 \\l\end{pmatrix}
 \tau_0^{2(s-1-l)} \cH_2^{2l+2}, \quad N=[s-1],
 \end{split}\end{equation}   
 where $\cH_2^{2l+2}   =  \int_{\R}\xi^{2l} \bigl( \bigl|\widehat{|q|^2-1}\bigr|^2
 +|\widehat{q'}|^2 \bigr)  (\xi)\dxi\lesssim \tau_0^{-2(s-1-l)}(E^{s}_{\tau_0}(q))^2$ and
 $$ 
\cH_2^{>2N+2}(i\sigma)
 = \Re(4z^2\tilde T_2)(i\sigma)- \sum_{l=0}^{N} (-1) ^l
\cH_2^{2l+2}\,\tau^{-2l-1}=o(\tau^{1-2s}),
\quad\tau\rightarrow\infty.
 $$  
\end{lem}
\begin{proof}  
The proof (with $\tau_0=1$)  in \cite{KT} works here  and we  give here the proof for readers' convenience.
We make use of the unitary Fourier transform and inverse Fourier transform
$$
\hat f(\xi)
=\frac{1}{\sqrt{2\pi}}\int_{\R} e^{-ix\xi} f(x)\dx,
\quad
f(x)=\frac{1}{\sqrt{2\pi}}\int_{\R} e^{ix\xi} \hat f(\xi) \,d\xi,
$$
to write for any function $f\in H^s$,
\begin{equation*}\label{QuaSym}\begin{split}
\Re \int_{x<y} e^{-\tau (y-x)} \bar f(x) f(y)\dx\dy
&=\frac{1}{2\pi}\Re\int_{\R^3}
\frac{1}{\tau-i\xi} e^{iy\eta-iy\xi}
\hat f(\eta)
\ov{\hat f(\xi)}  \dy\,d\xi\,d\eta 
\\
&= \Re\int_{\R}\frac{1}{\tau-i\xi}  \hat f(\xi) 
\ov{\hat f(\xi)} \,d\xi= \int_{\R}\frac{\tau}{\tau^2+\xi^2}  |\hat f(\xi) |^2
 \,d\xi, 
\end{split}\end{equation*}  
where the righthand side is the value at the point $i\tau$ of  the harmonic function   on the upper half plane with the trace 
$\pi|\hat f(\xi)|^2$ on the real axis.
Therefore noticing from the definition \eqref{T2}   that
{\small\begin{align*}
\Re(4z^2\tilde T_2)(i\sigma)
=\int_{x<y}e^{-\tau(y-x)}
\Bigl( (|q|^2-1)(y)(|q|^2-1)(x)+ q'(y)\bar q'(x)\Bigr) \dx\dy,
\end{align*}}
 let $f=|q|^2-1$ or $q'$ above, then \eqref{Estau0}  follows from   the trace formula \eqref{TracForm}.\end{proof}
%%%%%%%%%%%%%%%%%

 \subsubsection{Ideas and the organisation of the rest of this section}
Recall the notations on the imaginary axis \eqref{ImaginAxis} and the expansion \eqref{Expansion:lnT,T2}: 
$$\ln\Tc^{-1}(i\sigma)=\tilde T_2(i\sigma)+\tilde T_3(i\sigma)+\sum_{j\geq 2}\tilde T_{2j}(i\sigma).$$
Recall   the energy $\cE^{s}_{\tau_0}$ defined in \eqref{cElarge}   (noticing the symmetry $\Re\ln\Tc^{-1}(i\sigma)=\Re\ln\Tc^{-1}(-i\sigma)$ in \eqref{Tc:property} and hence $G(i\frac\tau2)=   \Re\bigl( 4z^2\ln\Tc^{-1}\bigr)(i\sigma)$):
 {\small\begin{align*}
  \cE^{s}_{\tau_0} 
 = &-\frac{2}{\pi}\sin(\pi(s-1))
 \int_{\tau_0}^\infty (\tau^2-\tau_0^2)^{s-1}  
 \Bigl (\Re\bigl( 4z^2\ln\Tc^{-1}\bigr)(i\sigma)
 - \sum_{l=0}^N(-1)^l  \cH^{2l+2} \tau^{-2l-1} \Bigr)  
\dtau\notag
\\
&+\sum_{l=0}^N \tau_0^{2(s-1-l)}
 \begin{pmatrix}  s-1 \\l\end{pmatrix} \cH^{2l+2}.
 \end{align*}   }
Recall the formulation of the energy norm  $(E^{s}_{\tau_0})^2$ in \eqref{Estau0}.

In order to show the equivalence \eqref{Equiv} between $\cE^s_{\tau_0}$  and  $(E^{s}_{\tau_0})^2$,
it suffices to estimate,  if $s\in(\frac12,\frac32)$ such that $[s-1]<1$, 
their difference $|\cE^{s}_{\tau_0}-(E^{s}_{\tau_0})^2|$ which concerns cubic or higher order terms   in the expansion of $\ln\Tc^{-1}(i\sigma)$:
{\small\begin{equation}\label{Difference:E}
|\cE^{s}_{\tau_0}-(E^{s}_{\tau_0})^2|
=\Bigl|\frac{2}{\pi}\sin(\pi(s-1))\int^\infty_{\tau_0}(\tau^2-\tau_0^2)^{s-1}  \Re\Bigl(4z^2\bigl(\tilde T_3 +\sum_{ j\geq 2 }\tilde T_{2j}\bigr) \Bigr)(i\sigma) \dtau \Bigr|
\end{equation}  }
by $Cc_\tau (E^{s}_{\tau_0})^2$ whenever $c_\tau<\frac{1}{2C}$.

If $s\geq\frac32$ is large enough
\footnote{
Recalling Remark \ref{rmk:a4}: $\cH_3^2=0$, we indeed have to do finite expansions only when $s\geq 2$. }, then  we also have to do finite expansions for   $\tilde T_3(i\sigma)$ and $\tilde T_{2j}(i\sigma)$, $2\leq j\leq s$ until $k$-th order, $k=[2s]$:
{\small\begin{align*}
&\tau^2\tilde T_3(i\sigma)=\sum_{l=3}^{k} \cH_{3}^{l}\tau^{-l+1}
+\cH_3^{>k}(i\sigma),
\,\,
 \tau^2\tilde T_{2j}(i\sigma)
=\sum_{l=2j}^k \cH_{2j}^{l}\tau^{-l+1}+\cH_{2j}^{>k}(i\sigma)\hbox{ as }\sigma\rightarrow\infty,
\end{align*}} 
 such that  the difference above \eqref{Difference:E} is replaced by
{\small\begin{align}\label{replace} 
 |\cE^{s}_{\tau_0}-(E^{s}_{\tau_0})^2|
&=\Bigl|\frac{2}{\pi}\sin(\pi(s-1))\int^\infty_{\tau_0}(\tau^2-\tau_0^2)^{s-1}  \Re\Bigl(\cH_3^{>k} +\sum_{2\leq j\leq s}\cH_{2j}^{>k}
 \notag
\\
&+\tau^2\sum_{j>s}\tilde T_{2j} \Bigr)(i\sigma) \dtau
 -\sum_{l=1}^{[s-1]} \tau_0^{2(s-1-l)}\begin{pmatrix} s-1\\ l \end{pmatrix} 
\Re \bigl( \cH_{3}^{2l+2}+\sum_{2\leq j\leq s}\cH_{2j}^{2l+2}  \bigr)\Bigr|.
\end{align} }  
 
%  $\Re (\tilde T_3 +\sum_{2\leq j<s}\tilde T_{2j} )$ in the integrand above by their remainder terms $$, while keep the higher order terms  $\Re ( \sum_{j>s}\tilde T_{2j} )$ in the integrand all the same.
We are going to follow the strategy and make use of the estimates in Section 6, \cite{KT} to control the differences \eqref{Difference:E} and \eqref{replace}.
%We will estimate the high order terms $\tilde T_{2j}$, $j>s$ and the low order terms $\tilde T_{2j}$, $2\leq j<s$ separately to show the equivalence between $\cE^{s}_{\tau_0}$ and $(E^{s}_{\tau_0})^2$: $\tilde T_3$ is    high or low order term when $s\in (\frac12, \frac32)$ or $s>\frac32$ respectively.
 The rest of this section is organised as follows: 
 \begin{itemize}
 \item
 We  establish the estimates for high order terms $\tilde T_{2j}$, $j>s$ (and $\tilde T_3$ if $s\in(\frac12,\frac32)$) and low order terms $\tilde T_3, \tilde T_{2j}$, $2\leq j<s$ for $s>\frac32$  in Subsections \ref{subs:High}  and \ref{subs:Low} respectively;
 \item We complete the proof of Theorem \ref{thm:Es} in Subsection \ref{subs:energy}.
 \end{itemize}

 \subsection{High order terms }\label{subs:High}
 We are going to derive the estimates for   high order terms in this section, which can be viewed as a more accurate version of the  estimates  in Section \ref{subs:Tr} on the imaginary axis setting:
\begin{itemize}
\item We derive first some preliminary estimates for $q_2, q_3, q_4$, $P$ and then the estimates for the integrals $\<XY>^j, \<XBBY>_{2j}, \<XBY>$  in Subsection \ref{subs:FundEst};
\item We make use of these  estimates   to control the high order terms    in Subsection \ref{subss:high}, which can  control the difference \eqref{Difference:E} when $s\in(\frac12,\frac32)$.
\end{itemize}
We will use the  estimates \eqref{product} and \eqref{fg:free} freely through this section.
 \subsubsection{Preliminary estimates}\label{subs:FundEst}  
Recall the notations in Subsection \ref{subs:fw} and we are going to estimate $q_2, q_3, q_4$, $P$, $\<XY>^j$, $\<XBBY>_{2j}$, $\<XBY>$,  in terms of 
{\small$$C_\tau=1+\frac1\tau\|q'\|_{l^\infty_\tau DU^2}
\hbox{ and } \|\bq\|_{l^p_{\tau} DU^2}=\bigl\| \bigl( \||q|^2-1\|_{l^p_{\tau} DU^2},\, \|q'\|_{l^p_{\tau} DU^2}\bigr)\bigr\|_{\ell^p}.$$  }
\begin{lem}\label{lem:C0}
Assume \eqref{ImaginAxis} and \eqref{smallq}: $(\lambda,z)=(i\sigma, i\tau/2)\in\cR$, $\tau\geq\tau_0\geq 2$, $\zeta=\lambda+z=i\omega$, $\omega\in [\frac\tau2,\tau)$  and   $q\in X^s$, $s>\frac12$ such that $\bigl| |q|^2-1\bigr|\leq \frac{1}{64}{\tau_0}^2$.  

Then the following estimates hold for $P$ defined in \eqref{P} and
 for $p\in [1,\infty]$:
\begin{equation*}\begin{split}
& \|P \|_{l^\infty_\tau V^2} \lesssim C_\tau,
\,\, 
\|q_\kappa\|_{l^p_\tau DU^2} +\|\frac{1}{\tau^2}(|q|^2-1)\|_{l^p_\tau U^2} \lesssim \frac{C_\tau}{  \tau} \|\bq\|_{l^p_{\tau} DU^2},\, \kappa=2,3,4.
\end{split}\end{equation*}  
 \end{lem}

\begin{proof}  
We firstly derive straightforward from the assumption \eqref{smallq} that
\begin{align*}
\|q\|_{l^\infty_\tau V^2}
\lesssim \|q\|_{L^\infty}+\|q'\|_{l^\infty_\tau DU^2}
\lesssim \tau(1+\tau^{-1}\|q'\|_{l^\infty_\tau DU^2})=\tau C_\tau.
\end{align*}  
Recalling  the partition of unity \eqref{unity}, we have
%\footnote{It is straightforward to calculate 
%$\Bigl|\frac{1}{\omega^{-2}|q(x)|^2+1}-\frac{1}{\omega^{-2}|q(y)|^2+1}\Bigr|
% =\Bigl| \frac{\omega^{-2}(|q(y)|^2-|q(x)|^2)}{(\omega^{-2}|q|^2(x)+1)(\omega^{-2}|q|^2(y)+1)}\Bigr|
%\\
% \leq \omega^{-2}(|q(y)|^2-|q(x)|^2).$
%=\omega^{-2}\bigl(q(y)(\bar q(y)-\bar q(x))
%+\bar q(x)(q(y)-q(x))\bigr)$.
%}
\begin{align*}
&\Bigl\| \frac{1}{ \omega^{-2}|q|^2+1}\Bigr\|_{l^\infty_\tau V^2}
\leq 1+\mathop{\sup}\limits_k
 \Bigl\| \frac{\chi_{\tau,k}}{ \omega^{-2}|q|^2+1}\Bigr\|_{ V^2}
%\\
%& =1+\mathop{\sup}\limits_k \mathop{\sup}\limits_{k-\frac1\tau=t_0<t_1<\cdots<t_N=k+\frac1\tau}
%\Bigl(\sum_{j=0}^{N-1} \bigl(\frac{\chi_{\tau,k}(t_{j+1})}{ \omega^{-2}|q(t_{j+1})|^2+1}- \frac{\chi_{\tau,k}(t_j)}{ \omega^{-2}|q(t_j)|^2+1} \bigr)^2\Bigr)^{\frac12}
\\
& \leq1+2\mathop{\sup}\limits_k 
\mathop{\sup}\limits_{\frac{k-1}\tau=t_0<t_1<\cdots<t_N=\frac{k+1}\tau}
\Bigl( \sum_{j=0}^{N-1}\bigl( (\chi_{\tau,k}(t_{j+1})  -  \chi_{\tau,k}(t_j) \bigr)^2
\\
&\qquad
+\bigl(\omega^{-2}(|q(t_{j+1})|^2-|q(t_j)|^2)\bigr)^2\Bigr)^{\frac12}
\lesssim1+\tau^{-1}\|q'\|_{l^\infty_\tau DU^2}=C_\tau,
\end{align*}
where we used $\bigl| |q|^2-1\bigr|\leq \frac{1}{64}{\tau_0}^2$ and $|q(t_{j+1})-q(t_j)|\lesssim  \|q'\|_{l^\infty_\tau DU^2}$.
Similarly we have the same estimate for  $\frac{1}{\tau^2}(|q|^2-1)$ and hence for the polynomial $P$.
Therefore the estimates for $q_\kappa$, $\kappa=2,3,4$ and $\|\frac{1}{\tau^2}(|q|^2-1)\|_{l^p_\tau U^2}$  follow from \eqref{product} and \eqref{fg:free}.
\end{proof}

We  claim the following estimates similar as \eqref{q4} (recalling $\chi_{\tau,k}$ in the partition of unity \eqref{unity} and the assumption \eqref{smallq}):
{\small\begin{equation}\label{Estimate:q4,chi}\begin{split}
&\Bigl\| \chi_{\tau,k} \Bigl(e^{\int^x_{\frac k\tau} q_4 dx'}-1\Bigr)\Bigr\|_{V^2}
\lesssim C_\tau \tilde c_{\tau,k}, 
\quad \Bigl\|\chi_{\tau,k}e^{\int^x_{\frac k\tau} q_4 dx'}\Bigr\|_{V^2}\lesssim 1,
\\
&\Bigl|e^{\int^{\frac k\tau}_{\frac{k'}\tau} q_4 dx}-1\Bigr|
\lesssim  e^{\frac12(k-k')}C_\tau \sum_{k''=k'}^k \tilde c_{\tau,k''}, 
\quad \Bigl| e^{\int^{\frac k\tau}_{\frac{k'}\tau} q_4 dx}\Bigr| \lesssim e^{\frac14 (k-k')},
\quad k'\leq k-1,
\end{split}\end{equation}  
with $\tilde c_{\tau,k}=\frac1\tau\bigl(\|\tilde\chi_{\tau,k}(|q|^2-1)\|_{DU^2}
+\|\tilde\chi_{\tau,k}q'\|_{DU^2}\bigr)$, $\tilde\chi=\chi(\frac{1}{12}\cdot)$.
Indeed,  we write  
\begin{align*} 
q_4=a+ib,
\quad a=\Re q_4=\frac{-2(|q|^2-1)/\omega}{\omega^{-2}|q|^2+1},
\quad b=\Im q_4=\frac{2\Im (q\bar q')/(\omega^2)}{\omega^{-2}|q|^2+1}.
\end{align*}
Since we derive from  \eqref{smallq} that
$\bigl\| \bar\chi_{\tau,k}(x) \int^x_{\frac k\tau} a \dx'  \bigr\|_{V^2}
\leq \frac8\tau   \|  a\|_{L^\infty} \leq \frac12$ with $\bar\chi=\chi(\frac12\cdot)$ supported on $[-\frac83,\frac83]$ such that $\bar\chi\chi=\chi$,
we have the following estimate from Lemma \ref{lem:C0}:
\begin{align*}
&\bigl\|  \chi_{\tau,k}(x)\bigl(e^{\int^x_{\frac k\tau} a \dx'}-1 \bigr) \bigr\|_{V^2}
 \lesssim \Bigl\|\bar\chi_{\tau,k}(x) \int^x_{\frac k\tau} a \dx'  \Bigr\|_{V^2}
 \lesssim  \frac{C_\tau}\tau \bigl\|\tilde\chi_{\tau,k}(|q|^2-1)\bigr\|_{ DU^2}.
\end{align*}
On the other hand, since $|e^{ic}|\leq 1$, $c\in\R$, we   have
%$\| e^{ib} \|_{l^\infty_1V^2}\lesssim 1$ and
\begin{align*}
\Bigl \| \chi_{\tau,k}(x) \Bigl(e^{i\int^x_{\frac k\tau} b\dx'}-1\Bigr) \Bigr\|_{V^2}
&\lesssim \Bigl\|\bar\chi_{\tau,k}(x) \int^x_{\frac k\tau} b\dx'   \Bigr\|_{V^2}
 \lesssim \frac{C_\tau}\tau \bigl\|\tilde\chi_{\tau,k}q'\bigr\|_{ DU^2}.
\end{align*}
Then the first line of the estimate \eqref{Estimate:q4,chi} follows. 
Similarly we have the second line of \eqref{Estimate:q4,chi} for $k'\leq k-1$: We derive straightforward from \eqref{smallq} that
$\bigl|e^{\int^{\frac k\tau}_{\frac{k'}\tau}q_4\dx}\bigr|\leq \bigl| e^{\int^{\frac k\tau}_{\frac{k'}\tau}a\dx} \bigr|\leq e^{\frac14(k-k')}$,
and hence by Lemma \ref{lem:C0}
\begin{align*}
\Bigl|e^{\int^{\frac k\tau}_{\frac{k'}\tau}q_4\dx}-1\Bigr|
\leq e^{\frac14 (k-k')} \Bigl| \int^{\frac k\tau}_{\frac{k'}\tau}q_4\dx\Bigr|
\lesssim   \bigl(e^{\frac14 (k-k')}  (k-k')\bigr)\, C_\tau\sum_{k''=k'}^k \tilde c_{\tau,k''}.
\end{align*}}
Therefore we have the following estimates for $\<XY>^j, \<XBBY>_{2j}, \<XBY>$:
\begin{lem}\label{lem:J}
Assume the same hypothesis as in Lemma \ref{lem:C0}. 
Then  we have the following estimates:
\begin{equation}\label{EstimateJS}\begin{split}
&|\<XY>^j|
\lesssim (\|q_2\|_{l^2_\tau DU^2}\|q_3\|_{l^2_\tau DU^2})^j,
\quad 
|\<XBBY>_{2j}|
\lesssim 
\max\{\|q_{2}\|_{l^{2j}_\tau DU^2}, \|q_3\|_{l^{2j}_\tau DU^2}\}^{2j},
\\
&|\<XBY>|
\lesssim \frac{C_\tau}{\tau }
 \|\bq\|_{l^3_\tau DU^2} 
\|q_2\|_{l^3_\tau DU^2}\|q_3\|_{l^3_\tau DU^2}.
\end{split}\end{equation}
%and hence there exists a constant $C$ such that  in terms of $c_\tau, C_\tau$ in \eqref{cq}
%\begin{equation}\label{EstimateXY}\begin{split}
% |\<XY>^j|+|\<XBBY>_{2j}|
%\leq (CC_\tau c_\tau)^{2j}, 
%\quad 
%|\<XBY>|\leq  (CC_\tau c_\tau)^3.
%\end{split}\end{equation}  
\end{lem}

\begin{proof} 
We  follow the decomposition idea in the proof of Lemma \ref{Lem:wn} to derive that
{\small\begin{align*}
&\Bigl\| \int_{x<y<t} (e^{\varphi(y)-\varphi(x)}-e^{-\tau(y-x)}) g(y) h(x)\dx\dy\Bigr\|_{U^2_t}
\\
&\lesssim \sum_{k} \Bigl\|\int_{y-\frac 3\tau}^y (\chi_{\tau,k} g)(y) e^{-\tau(y-x)}\Bigl(e^{\int^y_{\frac k\tau} q_4}-1 +  e^{\int^y_{\frac k\tau} q_4}(e^{\int^{\frac k\tau}_x q_4}-1)\Bigr) (\tilde \chi_{\tau,k} h)(x)\dx\Bigr\|_{DU^2_y}
\\
&\quad
+\sum_{k'\leq k-1} e^{- (k-k')}\Bigl\| \int_{-\infty}^{y-\frac 3\tau} (\chi_{\tau,k} g)(y) \Bigl(e^{\int^y_{\frac k\tau} q_4}-1 +  e^{\int^y_{\frac k\tau} q_4}(e^{\int^{\frac k\tau}_{\frac{k'}\tau} q_4}-1)
 \\
 &\qquad\qquad 
 +  e^{\int^y_{\frac k\tau} q_4}e^{\int^{\frac k\tau}_{\frac{k'}\tau}q_4}(e^{\int^{\frac{k'}\tau}_{x} q_4}-1)\Bigr)  
 (\chi_{\tau,k'} h)(x)\dx\Bigr\|_{DU^2_y}, 
\end{align*}   }
and hence we take $g=q_2$ and $h=q_3$ to arrive at the  estimate for $|\<XBY>|$ in  \eqref{EstimateJS} by virtue of the claim 
\eqref{Estimate:q4,chi}. 

Similarly, since $e^{\varphi(y)-\varphi(x)}=e^{-\tau(y-x)}e^{\int^y_{\frac k\tau} q_4}e^{\int^{\frac k\tau}_{\frac{k'}\tau}q_4}e^{\int^{\frac{k'}\tau}_x q_4}$, we have the following estimate for the operator $S$ (defined in \eqref{SS}) from the claim \eqref{Estimate:q4,chi}:
{\small\begin{align*}
\|S\|_{V^2\mapsto U^2}\lesssim \|q_2\|_{l^2_\tau DU^2}\|q_3\|_{l^2_\tau DU^2},
\quad S(f)(t)=\int_{x<y<t}e^{\varphi(y)-\varphi(x)} q_2(y) (q_3 f)(x)\dx\dy,
\end{align*}}
and hence the estimate for $\<XY>^j$ in  \eqref{EstimateJS} follows:
\begin{align*}
|\<XY>^j|=|\lim_{t\rightarrow\infty}(S^j 1)(t)|\leq \|S^j 1\|_{V^2}\leq (C\|q_2\|_{l^{2}_\tau DU^2}\|q_3\|_{l^{2}_\tau DU^2})^j.
\end{align*}
We apply the above estimate for the operator $S$ iteratively to arrive at  
{\small\begin{align*}
&\|S_{2j}\|_{V^2\mapsto U^2}\lesssim \max\{\|q_2\|_{l^{2j}_\tau DU^2}, \|q_3\|_{l^{2j}_\tau DU^2}\}^{2j}, \hbox{ for the operator }
\\
&S_{2j}(f)(t)=\int_{t_1<\cdots<t_{2j}<t}
\prod_{n=1}^{2j-1} e^{\delta_n(\varphi(t_{n+1})-\varphi(t_n))}
 q_{\kappa_1}(t_{2j})\cdots (q_{\kappa_{2j}} f)(t_1) dt_1\cdots dt_{2j},
\end{align*}}
where $\kappa_n\in \{2, 3\}$, $\delta_n\in \{1,\cdots, j\}$,
and hence the estimate for  $\<XBBY>_{2j}$ in  \eqref{EstimateJS} follows.
%{}
%Finally, the estimates in \eqref{EstimateXY} follow from \eqref{EstimateJS} and Lemma \ref{lem:C0}.  
  \end{proof}

 \subsubsection{Estimates for high order terms}\label{subss:high}
 Recall  \eqref{bq}, \eqref{Es},  \eqref{Q} and \eqref{cq}:
 {\small\begin{align*}
 &\bq=(|q|^2-1, q'), \quad E^{s}_{\tau}=\|\bq\|_{H^{s-1}_\tau},
 \quad C_{\tau}=1+\frac1\tau\|q'\|_{l^\infty_\tau DU^2},
 \\
&Q=\frac{1}{\tau}(|q|^2-1, q', \bar q'),
\quad c_\tau=\frac1\tau\|\bq\|_{l^2_\tau DU^2}= \|Q\|_{l^2_\tau DU^2},
 \end{align*}} 
and the scaling invariance property \eqref{ftau}: {\small\begin{equation}\label{Qtau}\begin{split}
 &E^{s}_{\tau_0}%=\|\bq\|_{H^{s-1}_{\tau_0}}
 = \tau_0^{s-\frac12}\|\bq_{\tau_0}\|_{H^{s-1}},
\quad
 \|f\|_{l^p_{\tau} DU^2}= \|f_{\tau_0}\|_{l^p_{\tilde\tau} DU^2},  
% \\
% &   
%c_{\tau}\sim \frac1\tau\|U\|_{l^2_{\tau} DU^2},
%\quad U=\tau Q=(|q|^2-1, q', \bar q'),\quad
\quad f_\tau=\frac1\tau f(\frac\cdot\tau),
 \quad \tilde\tau:=\frac{\tau}{\tau_0}.
 \end{split}\end{equation} }

 We are going to give the estimates for high order terms in the expansion of $\ln\Tc^{-1}$ in Lemma \ref{Lem:IA}: $\tilde T_{2j}$, $j>s$ if $s>\frac32$ or $\tilde T_3$ and $\tilde T_{2j}$, $j>s$ if $s\in(\frac12, \frac32)$, which will be used to control the energy difference $|\cE^{s}_{\tau_0}-(E^{s}_{\tau_0})^2|$ in \eqref{Difference:E} and \eqref{replace}. Indeed, after rescaling, the estimates in Sections 5, 6 in \cite{KT} work well here and we simply make use of them to derive the estimates.
We refer the interested readers there for more detailed analysis.
 \begin{prop} \label{prop:smalls}
Assume \eqref{ImaginAxis} and \eqref{smallq}.
%Let $E^{s}_{\tau}$, $c_\tau$, $C_\tau$ be defined in \eqref{Es}, \eqref{cq}.
For $j\geq 1$, $T_{2j}=\<XY>^j$ and for $j\geq 2$, $\tilde T_{2j}$ is finite linear combination of integrals $\<XBBY>_{2j}$. $\tilde T_3$ is finite linear combination of cubic terms in \eqref{ReT3}.

Then there exist a constant $C$ and a constant $C_j$ depending on  $j\geq2$ such that
{\small\begin{equation}\label{Estimate:T3,Tj,smalls}\begin{split}
&|\mathbf{1}_{j\geq 1}T_{2j}(i\sigma)|
+|\mathbf{1}_{j\geq 2}\tilde T_{2j}(i\sigma)|\leq  (CC_\tau \|Q\|_{l^2_\tau DU^2})^{2j}, 
\\
&|\mathbf{1}_{j\geq 2}\tilde T_{2j}(i\sigma)|\leq C_j
\bigl( C_\tau  \|Q\|_{l^{2j}_\tau DU^2} \bigr)^{2j}, 
\\
&|\tilde T_3(i\sigma)|
 \leq  \bigl(CC_\tau \|Q\|_{l^{3}_\tau DU^2} \bigr)^3.
\end{split}\end{equation}   
}
If we assume furthermore (with a possibly larger $C$ depending on $s$)
\begin{equation}\label{assumption:c0}
c_{\tau_0}=\frac{1}{\tau_0} \|\bq\|_{l^2_{\tau_0} DU^2} \leq \frac 1C,
\end{equation} 
we have for non-half integer $s>\frac12$   that
{\small\begin{equation}\label{EstimateT2j}\begin{split}
&\int^\infty_{\tau_0}(\tau^2-\tau_0^2)^{s-1}
 \tau^2 \sum_{j\geq 2s-1} 
 \Bigl( \bigl|T_{2j}(i\sigma) \bigr|
 +\mathbf{1}_{j\geq2}\bigl|\tilde T_{2j}(i\sigma) \bigr|\Bigr) \dtau
 \leq \frac{C c_{\tau_0}^{2[2s]-2}}{ [2s]+1-2s}(E^{s}_{\tau_0})^2,  
\end{split}\end{equation}
 \begin{equation}\label{EstimatetildeT2j}\begin{split} 
 &\int^\infty_{\tau_0}(\tau^2-\tau_0^2)^{s-1}
 \tau^2\sum_{s\leq j<2s-1}\bigl| \mathbf{1}_{j\geq2}\tilde T_{2j}(i\sigma) \bigr| \dtau
 \leq\frac{ Cc_{\tau_0}^{2[s]}}{ [s]+1-s}(E^{s}_{\tau_0})^2,
\end{split}\end{equation} }
and  in particular when $s\in(\frac12,\frac32)$,   we have
{\small\begin{equation}\label{Estimatesmalls}\begin{split}
&\int^\infty_{\tau_0}(\tau^2-\tau_0^2)^{s-1}
 \tau^2\Bigl( \bigl|\tilde T_3(i\sigma)\bigr|+
 \sum_{j\geq 2}  \bigl|\tilde T_{2j}(i\sigma)\bigr| \Bigr) \dtau
  \leq \frac{ Cc_{\tau_0}}{(s-\frac12)(\frac32-s)}(E^{s}_{\tau_0})^2.
\end{split}\end{equation} }
\end{prop}
%%%%%%%%%%%%%%%
 \begin{proof}   
The estimates for $T_{2j}(i\sigma), \tilde T_{2j}(i\sigma)$ in  \eqref{Estimate:T3,Tj,smalls} follow directly  from Lemma \ref{lem:C0} and  Lemma  \ref{lem:J}
\footnote{ 
If $|T_{2j}|\leq A$ then $|\tilde T_{2j}|\leq (2A)^j$, $j\geq 2$.
Indeed, following   the proof of Proposition 5.10 in \cite{KT}, by multiplying $q_2, q_3$ by $\eta$, we arrive from the expansions \eqref{Expansion:Tc} and \eqref{Expansion:lnT} at
 $ \ln(1 +\sum_{j\geq 1}\eta^{2j} T_{2j})=\eta^2T_2+\sum_{j\geq 2}\eta^{2j}\tilde T_{2j}. $
 We introduce a partial order $\preceq$ in the set of   holomorphic functions near zero, where $g\preceq h$ means that the absolute value of each coefficient in the Taylor series of $g$ at zero is bounded by the corresponding coefficient in the Taylor series of $h$. 
In particular, $\ln(1+\zeta)\preceq \frac{\zeta}{1-\zeta}:=f(\zeta)$ and hence $\sum_{j=2}^\infty\tilde T_{2j}\eta^{2j}\preceq f\circ f(A\eta)\preceq\sum_{j=1}^\infty 2^{j-1}A^j\eta^j$, $|T_{2j}|\leq A=CC_\tau c_\tau$ such that  $|\tilde T_{2j}|\leq (2A)^j$, $j\geq 2$. }.
By virtue of the integrals in \eqref{ReT3}, we derive from Lemmas  \ref{lem:C0} and \ref{lem:J} that
{\small \begin{align*} 
 |\tilde T_3(i\sigma)|
&\lesssim 
C_\tau\Bigl\|\frac{|q|^2-1}{\tau^2}\Bigr\|_{l^3_\tau U^2}^2
\Bigl\|\frac{|q|^2-1}{\tau}\Bigr\|_{l^3_\tau DU^2}
+|\<XBY>|
\\
& +\Bigl( \Bigl\|\frac{|q|^2-1}{\tau^2}\Bigr\|_{l^3_\tau U^2}
+\bigl\|\frac{q'}{\tau}\bigr\|_{l^3_\tau DU^2}\Bigr) 
\sum_{ h\in O}\|h\|_{l^3_\tau DU^2}^2
  \leq  \bigl(C C_\tau \|Q\|_{l^{3}_\tau DU^2} \bigr)^3.
\end{align*} }

If $C_\tau\leq  2$, by change of variables $\tau\to \tilde \tau=\frac{\tau}{\tau_0}$ we bound the integral  in terms of $|T_{2j}(i\sigma)|$ in \eqref{EstimateT2j} as follows:
\begin{align*}
&\int^\infty_{\tau_0}(\tau^2-\tau_0^2)^{s-1}
 \tau^2  \bigl( \bigl|T_{2j}(i\sigma)\bigr| +\mathbf{1}_{j\geq2}\bigl|\tilde T_{2j}(i\sigma) \bigr|\bigr)\dtau
 \leq (2C)^{2j}\int^\infty_{\tau_0}(\tau^2-\tau_0^2)^{s-1}
 \tau^2 \|Q\|_{l^2_\tau DU^2}^{2j} \dtau
 \\
 &=(2C)^{2j}\tau_0^{2s-1}\int^\infty_{1}(\tilde\tau^2-1)^{s-1}
   \|Q_{\tau_0}\|_{l^2_{\tilde\tau} DU^2}^{2j-2} 
 \|\bq_{\tau_0}\|_{l^2_{\tilde\tau}DU^2}^2d\tilde\tau,
\end{align*}
 with the equality ensured by \eqref{Qtau}.
 Proposition 5.13 in \cite{KT} implies that in the   regime $-\frac12<s-1\leq\frac{j-1}{2}\leq j-1$, the above integral is bounded by 
 \begin{align*}
 \frac{1}{j+1-2s}(2C)^{2j} \tau_0^{2s-1}  
 \|\frac{1}{\tau_0}\bq_{\tau_0}\|_{l^2_1 DU^2}^{2j-2} \|\bq_{\tau_0}\|_{H^{s-1}}^{2} 
=\frac{1}{j+1-2s}(2C)^{2j} c_{\tau_0}^{2j-2}(E^{s}_{\tau_0})^2,
 \end{align*}
 with the equality ensured by \eqref{Qtau}.

Therefore under the smallness  assumption \eqref{assumption:c0} such that
$C_\tau\leq 1+c_\tau\leq 1+c_{\tau_0}\leq 2$,  
  \eqref{EstimateT2j} holds.
% Since  $T_{2j}, \tilde T_{2j}$ are both homogeneous of degree $2j$ in  $(q_2, q_3)$, $j\geq 2$, the same bound (with a larger constant $C$) as for the summation $\sum_{j\geq2s-1}|T_{2j}(i\sigma)|$ holds for $\sum_{j\geq2s-1}|\tilde T_{2j}(i\sigma)|$,   and hence the estimate \eqref{EstimateT2j} follows.

Similarly, if $C_\tau\leq 2$, we bound the integral  in terms of $|\tilde T_{2j}(i\sigma)|$ in \eqref{EstimatetildeT2j} as
{\small\begin{align*}
&\int^\infty_{\tau_0}(\tau^2-\tau_0^2)^{s-1}
 \tau^2  \bigl|\tilde T_{2j}(i\sigma) \bigr| \dtau
% \leq C_j 2^{2j}\int^\infty_{\tau_0}(\tau^2-\tau_0^2)^{s-1}
 %\tau^2 \|Q\|_{l^{2j}_\tau DU^2}^{2j} \dtau
% \\
\leq C_j 2^{2j}\tau_0^{2s-1}\int^\infty_{1}(\tilde\tau^2-1)^{s-1}
   \|Q_{\tau_0}\|_{l^{2j}_{\tilde\tau} DU^2}^{2j-2} 
 \|\bq_{\tau_0}\|_{l^{2j}_{\tilde\tau}DU^2}^2d\tilde\tau.
\end{align*}} Proposition 6.2 in \cite{KT} and \eqref{Qtau} implies \eqref{EstimatetildeT2j}.

Finally we can do the same as for $\tilde T_{2j}$ above to $\tilde T_3$ (with $2j$ replaced by $3$ in the above) for $s\in (\frac12,\frac32)$.
Hence \eqref{Estimatesmalls} follows from \eqref{EstimateT2j} and \eqref{EstimatetildeT2j}.
  \end{proof}

\subsection{Low order terms}\label{subs:Low}
Let $s>\frac32$ and we aim to get the estimates for the low order terms $\tilde T_3$, $\tilde T_{2j}$, $2\leq j<s$ in this section, which will be used to control the energy difference $|\cE^{s}_{\tau_0}-(E^{s}_{\tau_0})^2|$ in \eqref{replace}.
Indeed, as in Section 6 \cite{KT}, we have to do integration by parts to expand low order terms until $k$-th order, $k=[2s]-1$.

 We will first derive  the general formula for finite expansions    in Subsection \ref{subss:FE} and then   apply it to the low order terms    in Subsection \ref{subs:TEst}. 

\subsubsection{Finite expansion}\label{subss:FE}
Let us do integration by parts in the following integral:
{\small \begin{equation*}\label{IP:varphi} 
\begin{split}
& \int_{x<y} e^{\varphi(y)-\varphi(x)} g(y) h(x) \dx\dy 
 =\int_{x<y} e^{-\tau(y-x)} e^{\int^y_x q_4} g(y) h(x) \dx\dy
 \\
& 
=  \frac{1}{\tau}\int_{\R} gh \dx
+\frac{1}{\tau}\int_{x<y} e^{\varphi(y)-\varphi(x)}
 \bigl( g(y)(q_4 h)(x)-g(y)h'(x)\bigr) \dx\dy
% \hbox{ if } g\in DU^2, \,   h\in U^2,
 \\
& 
=  \frac{1}{\tau}\int_{\R} gh\dx
+\frac1\tau\int_{x<y}e^{\varphi(y)-\varphi(x)}
\bigl( (q_4 g)(y)h(x)+g'(y) h(x)\bigr)\dx\dy,
\end{split}\end{equation*} }
which leads us to define the two operators $D_{\pm}$:
 \begin{equation}\label{DM}
 D_{\pm}(g)=(q_4\pm \d_x)g. 
 \end{equation}
Then we have  
 \begin{lem} \label{lem:IP} 
 For any $k\in\N$, we have the following formal finite-order expansion: 
 \begin{equation*}
\begin{split}
& \int_{x<y} e^{\varphi(y)-\varphi(x)} g(y)  h(x) \dx\dy
=  \sum_{\ell=1}^{k} b_\ell  
+  b^{\geq k+1},  
\end{split}\end{equation*}
  where  
{\small\begin{align*}
&b_\ell=\frac{1}{\tau^\ell} \int_{\R} (D_+^{m_\ell} g)( D_-^{n_\ell}  h) \dy
=  \int_{\R} \Bigl((\frac{D_+}{\tau})^{m_\ell} \frac g\tau\Bigr)
\Bigl( (\frac{D_-}{\tau})^{n_\ell}  h\Bigr) \dy,
%=  \int_{\R} \frac{1}{\tau^{m_\ell}}D_+^{m_\ell} (\frac1\tau q_2) \frac{1}{\tau^{n_\ell}}D_-^{n_\ell}  q_3  \dy,
\, \, m_\ell+n_\ell+1=\ell, 
\\
&b^{\geq k+1}(z)=  
\frac{1}{\tau^k} \int_{x<y}  e^{\varphi(y)-\varphi(x)}  (D_+^m g)(y)
  (D_-^n  h)(x) \dx\dy
  \\
  &\qquad
  \qquad
  = \int_{x<y}  e^{\varphi(y)-\varphi(x)}  \Bigl((\frac{D_+}{\tau})^{m} g\Bigr)(y)
\Bigl((\frac{D_-}{\tau})^{n} h\Bigr)(x) \dx\dy, 
\, m+n=k.
\end{align*}   }
 \end{lem}  
In order to study the applications of the operators $D_{\pm}$ on $q_2, q_3$,
we  make use of the structures and the preliminary estimates for $q_2, q_3, q_4, P$ derived in Lemma \ref{lem:C0}.
   Recall $\bq=(|q|^2-1, q')$ and   $Q=\frac1\tau (|q|^2-1, q', \bar q')$.
  In the following we will be flexible in the notations concerning $Q$ in the sense that  the notation $Q\cdots Q$ will be understood as one element in the set $\{Q^{\alpha_1}\cdots Q^{\alpha_n}\,|\, \alpha_\beta=1,2,3\}$.
 
 Notice   that $q_2, q_3, q_4$ are, up to the multiplication by $\frac1\omega q$, $\frac1\omega\bar q$, $\frac{1}{\omega^{-2}|q|^2+1}$, $\omega\in[\frac\tau2,\tau)$, linear combinations of components of $Q$, and the applications (perhaps several times) of the operators $D_+, \d_x$ or of the multiplication operators  $\mathcal{M}_{q_\kappa}$   on $q_{\kappa'}$, $\kappa, \kappa'=2,3,4$   are, up to the multiplication by the polynomials  in $\frac1\omega q, \frac1\omega\bar q, \frac{1}{\omega^{-2}|q|^2+1}$, 
$$
\hbox{applications of the derivative }\d_x\hbox{ on }Q, Q', Q'',\cdots
\hbox{ or the multiplication of }  Q.
$$

For notational simplicity, we introduce the set $O_{M}$, $M\geq 1$, which concerns $(M-1)$-times applications  of the operators $\frac1\tau D_{\pm}$ or $\frac{1}{\tau}\d_x$ or $\cM_{\frac1\tau Q}$ on $Q$, as follows (noticing that $O_{1}=O$ defined in \eqref{O}):
{\small\begin{equation}\label{OM}\begin{split}
O_{M}=\Bigl\{
h_{(M,\alpha)}=P\cdot \Bigl(\prod_{\gamma=1}^{\alpha-1} \frac{Q^{(\ell_\gamma)}}{\tau^{1+\ell_\gamma}}\Bigr) \frac{Q^{(\ell_{\alpha})}}{\tau^{\ell_\alpha}}
\,\Big|\, P \hbox{ is any polynomial of form \eqref{P}},
\\
\qquad\alpha=1,\cdots,M,
\quad \ell_1+\cdots+\ell_\alpha+\alpha=M \Bigr\}.
\end{split}\end{equation}}
In the following  $h_{(M,\alpha)}$ will always denote an element in $O_M$ which is homogeneous of degree $\alpha$ in $Q$ and sometimes we will denote simply $h_{(M)}\in O_M$ without pointing out the precise homogeneity. 
Then the operators $ \frac1\tau D_{\pm}, \frac1\tau\d_x, \cM_{\frac1\tau Q}$ map $O_M$ to $O_{M+1}$ and $\cM_{\frac{Q^{(m)}}{\tau^{1+m}}}$ maps $O_M$ to $O_{M+m}$.

We can rewrite the finite expansion in Lemma \ref{lem:IP}   with   $\bar g=h=q'$ by use of the notations  $h_{(m)}\in O_m$ as follows:
{\small\begin{align*}\label{expan:ex} 
&\frac{1}{\tau^2}\int_{x<y}e^{\varphi(y)-\varphi(x)} \bar q'(y) q'(x)\dx\dy
=\sum_{\ell=1}^k a_\ell+a^{\geq k+1},  \hbox{where}
\\
& a_\ell
%=\int_{\R}(\bar q')^{(m_\ell)} (q')^{(n_\ell)}\dx
=\int_{\R}h_{(\ell+1, \alpha)}\dx,
\quad\alpha\geq2,
%, \, m_\ell+n_\ell=\ell-1,
\\
& a^{\geq k+1}
%=\frac{1}{\tau^k}\int_{x<y}e^{-\tau(y-x)} (\bar q')^{(m)} (y)(q')^{(n)}(x)\dx\dy
=\int_{x<y}e^{\varphi(y)-\varphi(x)} h_{(m)}(y)h_{(n)}(x)\dx\dy,
\, m+n=k+2.
\end{align*}}
 Motivated by this finite expansion formula, we derive the following estimates.
  \begin{lem}\label{lem:LP}  Assume the same hypothesis in Lemma \ref{lem:C0}.
  Recall $c_\tau$, $C_\tau$, $E^{s}_{\tau}$ defined in \eqref{cq}, \eqref{Es}.
  Let   $O_{m}$ be the set defined in \eqref{OM}. 
 
 Then for any $M\geq 1$, there exists a constant $C(M)$  such that  the following holds for any $h_{(M,\alpha)}\in O_{M}$ homogeneous of degree $\alpha$, $\alpha\in [1,M]$ in $Q$:
{\small\begin{equation}\label{hm}
\|h_{(M,\alpha)}\|_{l^p_\tau DU^2}
\leq C(M) C_\tau \sum_{ \ell_1+\cdots+\ell_\alpha=M-\alpha, \, \,\sup_{\beta}\ell_\beta\leq [M/2]}^{M-1}
\prod_{\gamma=1}^{\alpha}
\bigl\|\frac{Q^{(\ell_\gamma)}}{\tau^{\ell_\gamma}}\bigr\|_{l^{p\alpha}_\tau DU^2}.
\end{equation} } 
Furthermore, the following estimates hold when $k=[2s]\leq 2s$, $\alpha\geq 2$:
 {\small\begin{equation}\label{aell}\begin{split}
 \tau^{\ell+2}\Bigl|\int_{\R}h_{(\ell+1,\alpha)}\dx\Bigr|
 &\lesssim  \tau_0^{\ell-2s+1}C_{\tau_0}c_{\tau_0}^{\alpha-2}(E^{s}_{\tau_0})^2,
 \,\,\tau\geq\tau_0, 
\,\,  1\leq \ell\leq k-1,
 \end{split}\end{equation}}  
{\small\begin{equation}\label{Jk}\begin{split} 
&\int^\infty_{\tau_0}(\tau^2-\tau_0^2)^{s-1} 
 \Bigl|\tau^2 \int_{x<y} \bigl( e^{\varphi(y)-\varphi(x)}-e^{-\tau(y-x)}\bigr)
  h_{(m,\alpha_1)}(y) h_{(n,\alpha_2)}(x)\dx\dy\Bigr|\dtau
\\
&\lesssim \frac{1}{|\sin(2\pi s)|}C_{\tau_0}^2\sum_{\alpha=\alpha_1+\alpha_2-1}^{k-1}c_{\tau_0}^{\alpha}(E^{s}_{\tau_0})^2,
\quad m+n=k, 
\end{split}\end{equation}}
{\small\begin{equation}\label{Ek}\begin{split} 
&\int^\infty_{\tau_0}(\tau^2-\tau_0^2)^{s-1}  \Bigl|\tau^2 \int_{t_1<\cdots<t_{n}} 
\prod_{i=1}^{n-1}e^{\delta_i(\varphi(t_{i+1})-\varphi(t_i))}
 h_{(m_1, \alpha_1)}(t_1)\cdots h_{(m_n,  \alpha_n)}(t_{n}) dt\Bigr|\dtau
\\
&\lesssim \frac{1}{|\sin(2\pi s)|}C_{\tau_0}^n\sum_{\alpha=\alpha_1+\cdots+\alpha_n-2}^{k+n-3}c_{\tau_0}^{\alpha} (E^{s}_{\tau_0})^2,
\quad 1\leq\delta_i\leq n/2,
\quad m_1+\cdots+m_n=k+1.
\end{split}\end{equation}} 
   \end{lem} 
\begin{proof}
 The estimate \eqref{hm} comes from the estimates in \eqref{product}, \eqref{fg:free} and the estimate for $P$ in Lemma \ref{lem:C0}:
{\small\begin{align*}
\|h_{(M,\alpha)}\|_{l^p_\tau DU^2}
& \lesssim C_\tau \prod_{\gamma=1}^{\alpha-1} \Bigl\| \frac{Q^{(\ell_\gamma)}}{\tau^{1+\ell_\gamma}}\Bigr\|_{l^{\frac{p\alpha}{\alpha-1}}_\tau U^2}
\Bigl\| \frac{Q^{(\ell_{\alpha})}}{\tau^{\ell_\alpha}}\Bigr\|_{l^{p\alpha}_\tau DU^2},
\quad \ell_1+\cdots+\ell_\alpha=M-\alpha
\\
&\quad\lesssim C_\tau\sum_{M-\alpha\leq \ell_1+\cdots+\ell_\alpha\leq M-1} 
\prod_{\gamma=1}^{\alpha}
\Bigl\|\frac{Q^{(\ell_\gamma)}}{\tau^{\ell_\gamma}}\Bigr\|_{l^{p\alpha}_\tau DU^2}
\hbox{ by use of  \eqref{fg:free}}.
\end{align*}
In the above, by integration by parts we can always choose $\sup_{1\leq\beta\leq\alpha}\ell_\beta\leq [M/2]$.}

We now turn to the proof of \eqref{aell}.
We first bound
$\tau^{\ell+2}\Bigl|\int_{\R} h_{(\ell+1,\alpha)}\dx\Bigr|$, $\tau\geq\tau_0$, $\alpha\geq 2$ by
{\small\begin{align*}
  \tau_0^{\ell+2-\alpha}\Bigl| \int_{\R}P\cdot \Bigl(\prod_{\gamma=1}^{\alpha-1} \frac{\bq^{(\ell_\gamma)}}{\tau_0^{1+\ell_\gamma}}\Bigr) \frac{\bq^{(\ell_{\alpha})}}{\tau_0^{\ell_\alpha}}\dx\Bigr|
=  \tau_0^{\ell+2-\alpha}\Bigl| \int_{\R}P(\frac{\cdot}{\tau_0})\cdot  
\prod_{\gamma=1}^{\alpha}  \bq_{\tau_0}^{(\ell_\gamma)}    \dx\Bigr|, 
\quad f_\tau=\frac1\tau f(\frac\cdot\tau),
\end{align*}}
 for some $\ell_1+\cdots+\ell_\alpha=\ell+1-\alpha$. We can do as above for $h_{(M,\alpha)}$ to derive that
{\small\begin{align*}
&\tau^{\ell+2}\Bigl|\int_{\R} h_{(\ell+1,\alpha)}\dx\Bigr|\lesssim C_{\tau_0}\tau_0^{\ell+2-\alpha}
\sum_{\ell+1-\alpha\leq \ell_1'+\cdots+\ell_\alpha'\leq \ell, \,\,\ell_\beta'\leq[(\ell+1)/2]} 
\prod_{\gamma=1}^{\alpha}
 \bigl\| \bq_{\tau_0}^{(\ell_\gamma')} \bigr\|_{l^{\alpha}_{1} DU^2} .
\end{align*}} 
By the proof of Proposition 6.5 \cite{KT}, the pointwise bound \eqref{aell} holds: {\small\begin{equation*}\begin{split}
 \tau^{\ell+2}\Bigl|\int_{\R}h_{(\ell+1,\alpha)}\dx\Bigr|
 &\lesssim  C_{\tau_0} \tau_0^{\ell+2-\alpha} 
 \|\bq_{\tau_0}\|_{l^2_1 DU^2}^{\alpha-2}   
 \|\bq_{\tau_0}\|_{H^{s-1}}^2,
 \quad  1\leq \ell\leq 2s-1
 \\
& \sim \tau_0^{\ell-2s+1}C_{\tau_0}c_{\tau_0}^{\alpha-2}(E^{s}_{\tau_0})^2,
\hbox{ by the scaling property \eqref{Qtau}}.
 \end{split}\end{equation*}}

   We now consider the following integral
{\small\begin{align*}
\tau^2\int_{x<y}e^{\varphi(y)-\varphi(x)} h_{(m)}(y) h_{(n)}(x) \dx\dy, \, m+n=k+1,
\end{align*}}
which, by  the estimate for $\<XY>$ in Lemma \ref{lem:J} and   Lemma \ref{lem:LP}, is bounded by
{\small\begin{align*}
\tau^2  C_\tau^2\sum_{\alpha=2}^{k+1}
\sum_{\ell_1+\cdots+\ell_\alpha=k+1-\alpha}^{k-1}
\prod_{\gamma=1}^\alpha 
\bigl\|\frac{Q^{(\ell_\gamma)}}{\tau^{\ell_\gamma}}\bigr\|_{l^{\alpha}_\tau DU^2}.
\end{align*}}  
Therefore, exactly as in the proof of Proposition \ref{prop:smalls}, we do change of variable $\tau\to \tilde\tau=\frac{\tau}{\tau_0}$ and make use of the scaling property \eqref{Qtau}, such that
{\small\begin{align*}
&\int^\infty_{\tau_0}(\tau^2-\tau_0^2)^{s-1}  \Bigl|\tau^2 \int_{x<y} e^{\varphi(y)-\varphi(x)} h_{(m)}(y) h_{(n)}(x)\dx\dy\Bigr| \dtau\lesssim \tau_0^{2s-1} \int^\infty_1 (\tilde\tau^2-1)^{s-1} 
\\
&\cdot
C_{\tau_0}^2\sum_{\alpha=2}^{k+1} 
\sum_{\ell_1+\cdots+\ell_\alpha=k+1-\alpha,\,\,\ell_\beta\leq[k/2]}^ {k-1} 
\Bigl(\prod_{\gamma=1}^{\alpha-2}
\Bigl\|\frac{Q_{\tau_0}^{(\ell_\gamma)}}{\tilde\tau^{\ell_\gamma}}\Bigr\|_{l^\alpha_{\tilde\tau}DU^2}\Bigr)
\Bigl\|\frac{\bq_{\tau_0}^{(\ell_{\alpha-1})}}{\tilde\tau^{\ell_{\alpha-1}}}\Bigr\|_{l^\alpha_{\tilde\tau}DU^2} 
\Bigl\|\frac{\bq_{\tau_0}^{(\ell_\alpha)}}{\tilde\tau^{\ell_\alpha}}\Bigr\|_{l^\alpha_{\tilde\tau}DU^2} d\tilde\tau,
\end{align*}}
which, by Propositions 6.4 and 6.5 in \cite{KT}, is bounded by 
{\small\begin{align*}
\frac{\tau_0^{2s-1}  
C_{\tau_0}^2}{|\sin(2\pi s)|} \sum_{\alpha=2}^{k+1}  
 \| \frac1{\tau_0}\bq_{\tau_0}\|_{l^\alpha_{1}DU^2}^{\alpha-2}
 \| \bq_{\tau_0}\|_{H^{s-1}}^2
 =\frac{C_{\tau_0}^2\sum_{\alpha=0}^{k}c_{\tau_0}^\alpha}{2s-[2s]}    
(E^{s}_{\tau_0})^2,
 \quad\hbox{if } \frac{k}{2}<s<\frac {k+1}2.
\end{align*}}

Similarly, by use of the estimates for $\<XBY>, \<XBBY>_{2j}$ in Lemma \ref{lem:J} and Lemma \ref{lem:LP}  we derive the estimates  \eqref{Jk}
and \eqref{Ek}  respectively.  \end{proof}

\subsubsection{Estimates for low order terms $\tilde T_3,  \tilde T_{2j}$}\label{subs:TEst} 
 In this section we are going to do finite expansions in Lemma \ref{lem:IP} to the low order terms  $\tilde T_3, \tilde T_{2j}$, $2\leq j<s$, $s>\frac32$, 
keeping in mind the estimates in Lemma \ref{lem:LP}.

%%%%%%%

 \begin{prop}\label{prop:larges}
Assume \eqref{ImaginAxis} and \eqref{smallq} and let $s>\frac32$ be away from half integers such that $k:=[2s]\geq 3$.
Then
{\small    \begin{itemize}
 \item  we can expand $\tau^2\tilde T_3(i\sigma)$ until $k$th-order: 
\begin{equation}\label{EstimateT3}\begin{split} 
&\tau^2\tilde T_3(i\sigma)=\sum_{\ell=3}^{k} \cH_{3}^\ell\tau^{-\ell+1}
+\cH_3^{>k}(i\sigma),
\quad \cH_3^\ell \hbox{ independent of }\tau,
\quad \sigma\rightarrow\infty, 
\\
&\hbox{such that } \tau_0^{2s-\ell} |\cH_{3}^\ell| 
\leq  C C_{\tau_0} \sum_{\alpha=1}^{k-2} c_{\tau_0}^\alpha   (E^{s}_{\tau_0})^2,
 \\
&\hbox{ and }\int^\infty_{\tau_0}(\tau^2-\tau_0^2)^{s-1}
 \bigl|\cH_3^{>k}(i\sigma)  \bigr| \dtau
 \leq  \frac{C  C_{\tau_0}^3 \sum_{\alpha=1}^{k-1} c_{\tau_0}^\alpha }{|\sin(2\pi s)|} (E^{s}_{\tau_0})^2;
\end{split}\end{equation} 

\item for $2\leq j<s$,   we can expand  $\tau^2\tilde T_{2j}(i\sigma)$ until $k$th-order:
\begin{equation}\label{EstimateTildeT2j}\begin{split}
&\tau^2\tilde T_{2j}(i\sigma)
=\sum_{\ell=2j}^k \cH_{2j}^\ell\tau^{-\ell+1}+\cH_{2j}^{>k}(i\sigma),
\quad \cH_{2j}^\ell \hbox{ independent of }\tau,
\hbox{ as }\sigma\rightarrow\infty,
\\
&\hbox{such that } \tau_0^{2s-\ell}|\cH_{2j}^\ell|\leq  C C_{\tau_0} \sum_{\alpha=2j-2}^{k-2} c_{\tau_0}^\alpha  (E^{s}_{\tau_0})^2,
 \\
&\hbox{ and }\int^\infty_{\tau_0}(\tau^2-\tau_0^2)^{s-1}
 \bigl|\cH_{2j}^{>k}(i\sigma)  \bigr| \dtau \leq \frac{C C_{\tau_0}^{2j} \sum_{\alpha=2j-2}^{k-1} c_{\tau_0}^\alpha}{ |\sin(2\pi s)|} (E^{s}_{\tau_0})^2. 
\end{split}\end{equation} 
\end{itemize}}
\end{prop} 
\begin{proof}   We will follow exactly the procedure in Section 6 \cite{KT} and hence be   sketchy.

Recall that $\tilde T_3$ is linear combinations of integrals of type in \eqref{ReT3}.
 We do integration by parts as in Lemma \ref{lem:IP} to the integrals  in \eqref{ReT3}   from the left and right sides alternatively to expand $\tau^2\tilde T_3(i\sigma)$ until $(k-1)$th-order:  We do integration by parts $(k-1)$-times to the integral $\<XBY>$ while $(k-2)$-times to the last two integrals in \eqref{ReT3},
 and we also  notice the expansion $\Im\zeta=\omega=\tau+O(\tau^{-1})$ as $\tau\rightarrow\infty$, to arrive at the expansion in \eqref{EstimateT3}.

Here, $\cH_{3}^\ell$,  $3\leq \ell\leq k$ is the leading order (in terms of $\tau^{-1}$)  of a  linear combination of  integrals of the following forms  with bounded coefficients:  
{\small\begin{align*}  
&  \tau^2\int_{\R}\Bigl( (D_+^{m_\ell}q_2)(D_-^{n_\ell}q_3)
-(\d_x^{m_\ell} q_2)\bigl((-\d_x)^{n_\ell}q_3\bigr)\Bigr)\dx, 
 \quad m_\ell+n_\ell\leq\ell-2,
\\
&   \tau^3\int_{\R} \bigl( \frac{(|q|^2-1)\hbox{ or }q'\hbox{ or }\bar q'}{\tau^2}h_{1}\bigr)^{(m_\ell-1)} h_{2}^{(n_\ell)}\dx,
\quad h_1, h_2\in O \hbox{ defined in \eqref{O}}.
\end{align*}}
%where $E_{3}^2$ also includes the integral $ \int_{\R}  \bigl(  |q|^2-1 \bigr)^2\, h\dx$, $h\in O=O_1$.
Then  $\cH_{3}^\ell$ is linear combination of integrals of type $\tau^{m+1}\int_{\R} h_{(m,\alpha)}\dx$, $3\leq m\leq\ell$ with $h_{(m,\alpha)}\in O_{m}$ homogeneous of degree $\alpha\geq 3$ in $Q$.
Hence  the estimate for $\cH_3^\ell$ in  \eqref{EstimateT3} follows from \eqref{aell} in Lemma \ref{lem:LP}.

Here $\cH_{3}^{>k}$ is  linear combination of  integrals such as $\tau\int_{\R}h_{(k,\alpha)}\dx$, $\alpha\geq3$   (which appears still because of the expansion of $\omega$ in terms of $\tau$) and the following  integrals   with bounded coefficients, $m+n=k-1$, $m_1+m_2+m_3=k+1$ and $ h_{(m)}, h_{(m, \alpha)}\in O_m$:
{\small\begin{align*}   
&\tau^2\int_{x<y} 
\Bigl( (e^{\varphi(y)-\varphi(x)}-e^{-\tau(y-x)}) 
( \frac{1}{\tau^{m}} D_+^{m}q_2)(y) ( \frac{1}{\tau^{n}}D_-^{n}q_3)(x)\dx\dy,
\\
&
\tau^2 \int_{x<y} e^{-\tau(y-x)}
\Bigl( ( \frac{1}{\tau^{m}}D_+^{m}q_2)(y) ( \frac{1}{\tau^{n}}D_-^{n}q_3)(x) -( \frac{1}{\tau^{m}}\d_y^{m}q_2)(y)\bigl(  \frac{1}{\tau^{n}}(-\d_x)^{n}q_3\bigr)(x)
\Bigr)\dx\dy, 
\\
&\tau^2 \int_{x<y} e^{-\tau(y-x)}
 h_{(m+1,\alpha_m)}(y) h_{(n+1, \alpha_n)}(x)\dx\dy, 
 \quad \alpha_m+\alpha_n\geq 3,
\\
& \tau^2\int_{x<t<y} e^{-\tau(y-x)}
 h_{(m_1)}(y) h_{(m_2)}(t) h_{(m_3)}(x)\dx dt \dy.
\end{align*}   }
The estimate for $\cH_3^{>k}$ in \eqref{EstimateT3} then follows from Lemma \ref{lem:LP}. 

Similarly, since $\tilde T_{2j}$ is a linear combination of integrals $\<XBBY>_{2j}$ reading as \eqref{tildeS}, \eqref{EstimateTildeT2j} follows from Lemmas \ref{lem:IP} and \ref{lem:LP}.\end{proof}

 %%%%%%%%%%%%%%%%%%%%%%
 \subsection{The energies}\label{subs:energy} 
 We restrict ourselves on the imaginary axis \eqref{ImaginAxis}: $(\lambda,z)=(i\sigma, i\tau/2)\in\cR$.
 For any $q\in X^s$, $s>\frac12$, by \eqref{ctau,Etau}, there exists $\tau_0\geq C$ such that the smallness assumption \eqref{tau0}: $c_{\tau_0}=\frac{1}{\tau_0}\|\bq\|_{l^2_{\tau_0}DU^2}< \frac{1}{2C}$ holds.
Consequently the  condition \eqref{smallq}: $| |q|^2-1|\leq \frac{1}{64}{\tau_0}^2$ holds if $\tau_0, C$ above have been chosen large enough.
 Indeed, by view of \eqref{qLinfty} such that $\|q\|_{L^\infty}\lesssim 1+\tau_0 c_{\tau_0}^{1/2}+\tau_0 c_{\tau_0}$, \eqref{smallq} follows from \eqref{tau0} for large $C$ and $\tau_0$.
From now on we fix $\tau_0$.

 \subsubsection{The expansion of $G(i\tau/2)$}
 Let us first consider rigorously the expansion of the real part of the expansion of $\ln\Tc^{-1}(i\sigma)$  in Lemma \ref{Lem:IA}:
 \begin{equation}\label{Expan:ReTc}\begin{split}
\Re(4z^2\ln\Tc^{-1})|_{\lambda=i\sigma}
&=\Re(4z^2\tilde T_2)|_{\lambda=i\sigma}
+\Re(4z^2\tilde T_f)|_{\lambda=i\sigma}
\\
&+\sum_{j>s}\mathbf{1}_{j\geq 2}\Re(4z^2\tilde T_{2j})|_{\lambda=i\sigma},
\qquad \tilde T_f:=\tilde T_3+\sum_{j=2}^{[s]}\tilde T_{2j}.
 \end{split}\end{equation}

 Recall the expansion for $\Re(4z^2\tilde T_2)|_{\lambda=i\sigma}$ in \eqref{Estau0} with $N=[s-1]$:
{\small \begin{align*} 
& \Re(4z^2\tilde T_2)(i\sigma)= \sum_{l=0}^{N} (-1) ^l
\cH_2^{2l+2}\,\tau^{-2l-1}+ \cH_2^{>2N+2}(i\sigma),
\hbox{ as }\tau\rightarrow\infty,
\\
&\hbox{ with }
\tau_0^{2(s-1-l)}\cH_2^{2l+2}\leq (E^{s}_{\tau_0})^2,
\quad \cH_2^{>2N+2}(i\sigma)=o(\tau^{1-2s}).
 \end{align*}}
  If $s>\frac32$ is away from half integers, then by Proposition \ref{prop:smalls} and Proposition \ref{prop:larges} with $k=[2s]$ under the smallness assumption \eqref{tau0},  
 \begin{itemize}
 \item We can expand $\tilde T_f=\tilde T_3+\sum_{j=2}^{[s]}\tilde T_{2j}$ as
 {\small\begin{equation}\label{tildeTf}\Re(4z^2\tilde T_f)(i\sigma)
 =\sum_{\ell=3}^{k} \cH_{f}^\ell\tau^{-\ell+1}+ \cH_f^{>k}(i\sigma)
 \hbox{ as }\sigma\rightarrow\infty.\end{equation}}
 Here $\cH_{f}^\ell=-\Re( \cH_{3}^\ell+\sum_{j=2}^{[s]}  \cH_{2j}^{\ell})$ with
 $\cH_f^{2l+1}=0$ and  $\tau_0^{2(s-1-l)}|\cH_{f}^{2l+2}|\leq  Cc_{\tau_0} (E^{s}_{\tau_0})^2$,
 and $|\cH_f^{>k}(i\sigma)|
 \leq |\cH_3^{>k}(i\sigma)|
 +\sum_{j=2}^{[s]}|\cH_{2j}^{>k}(i\sigma)|=o(\tau^{1-2s})$;
 
 \item $\sum_{j>s} \bigl|4z^2\tilde T_{2j}(i\sigma)\bigr|=o(\tau^{1-2s})$.
 \end{itemize}
 
 \smallbreak
 
 To conclude, noticing that $k\geq 2N+2$ and $k>2N+2$ only if $k\in 2\Z+1$,
by   \eqref{Expan:ReTc} and \eqref{tildeTf} above,  we have the following expansion  for   $G(i\frac\tau2)=\Re(4z^2\ln\Tc^{-1})(i\sigma)$ when $s>\frac12$ away from half integers:
{\small\begin{equation}\label{Expansion:G}
  G(i\frac\tau2)
 =\Re(4z^2\ln\Tc^{-1})(i\sigma) 
 = \sum_{l=0}^{N} (-1)^l \cH^{2l+2}\tau^{-2l-1}+\cH^{>2N+2}(i\frac\tau2),
 \,\hbox{ as }\tau\rightarrow\infty,
 \end{equation}}
 where (noticing that $\cH^{k}=0$ if $k>2N+2$ such that $k\in 2\Z+1$)
\footnote{%%%%
If $s=m$ an integer then \eqref{Est:G} follows from the proof of Proposition \ref{prop:larges}; if $s=m+\frac12$, then we can replace $\cH_f^{>k}(i\sigma)$ in $\cH^{>2N+2}(i\tau/2)$ by  
(noticing $\cH_f^{>k}=\cH_f^{>k-1}$, $\cH_{f}^{k}=0$ when $k\in2\Z+1$)
{\small \begin{equation}\label{decomposition:Tf}
 \Bigl( (\cH_f^{>k})( i\sigma;Q) 
 -(\cH_f^{>k})( i\sigma; Q_{<\tau})\Bigr)
 + \bigl(\cH_f^{>k-1}( i\sigma;Q_{<\tau}) \bigr)
\end{equation}}
such that $\cH^{>2N+2}(i\tau/2)=o(\tau^{1-2s})$ holds:
Here $Q_{<\tau}=\frac1\tau((|q|^2-1)_{<\tau}, q'_{<\tau}, \bar q'_{<\tau})$ denotes the low frequency part of $Q=\frac1\tau(|q|^2-1 , q', \bar q')$ and hence there exists at least one high frequency $Q_{\geq\tau}$ in the first part of the decomposition \eqref{decomposition:Tf} while there is only low frequency part $Q_{<\tau}$ in the second part of \eqref{decomposition:Tf}, from which we derive $\cH^{>2N+2}(i\tau/2)=o(\tau^{1-2s})$ from the proof of Proposition \ref{prop:larges} (see also Section 6 \cite{KT}).
}%%%%%.
\begin{equation}\label{Est:G}\begin{split}
&\cH^{2l+1} =\cH_2^{2l+1}+ (-1)^l \cH_{f}^{2l+1}
\hbox{ with }|\tau_0^{2(s-1-l)}\cH^{2l+1} |\leq   C (E^{s}_{\tau_0})^2,
\\
&\cH^{>2N+1}(i\frac\tau2)
=\cH_2^{>2N+1}
+  \cH_f^{>k}  
+ \sum_{j>s}\Re (4z^2\tilde T_{2j})(i\sigma)=o(\tau^{1-2s}).
\end{split}\end{equation}  

By Proposition \ref{prop:G} and the above expansion \eqref{Expansion:G} for the non negative superharmonic function $G$ on the upper half plane, the trace of $G$ on the real line exists as a finite Radon measure $\mu$ such that the measure $(1+\xi^2)^N\mu$ is finite. Furthermore,  by view of 
$G=\Re(z^2\ln\frac{1}{z-z_m})+$ harmonic function in a small enough neighborhood of $z_m$ (with $\lambda_m$ the zeros of $\Tc^{-1}(\lambda)$),   $\cH^{2l+1}$ indeed reads as in \eqref{cE2j}. %since %  we have the following trace formular for $G(i\tau/2)$:
%{\small \begin{align*}
% G(i\frac \tau2)
% &=\frac{1}{\pi}\int_{\R+i0} \frac{\tau/2}{(\tau/2)^2+\xi^2} 
% \frac 12\sum_{\pm}(2\xi)^2 \ln\bigl| \Tc^{-1}(\pm\sqrt{\xi^2+1})\bigr| \dxi
%- \sum_{m}  2z_m^2\ln\Bigl| \frac{i\tau/2-\bar z_m}{i\tau/2-z_m} \Bigr|.
% \end{align*}    }

\subsubsection{The energies}
If $s\in (\frac 12,\frac32)$, by view of the difference \eqref{Difference:E} for $ \bigl| \cE^{s}_{\tau_0} - (E^{s}_{\tau_0})^2\bigr|$, 
under the smallness condition \eqref{tau0}, we derive the equivalence relation \eqref{Equiv} from Proposition \ref{prop:smalls}.
 
Similarly for $s\geq \frac32$,  by the above finite expansion \eqref{Expansion:G} for $G(i\tau/2)$   and  Propositions \ref{prop:smalls} and \ref{prop:larges},  we also have the  inequality \eqref{Equiv} since  we derive from \eqref{replace} that
{\small\begin{align*}
&|\cE^{s}_{\tau_0}-(E^{s}_{\tau_0})^2|
 \leq \sum_{l=0}^N \tau_0^{2(s-1-l)}\begin{pmatrix} s-1\\ l \end{pmatrix} 
 \bigl|  \cH_{f}^{2l+2}\bigr|
\\
&+ \Bigl|\frac{2}{\pi}\sin(\pi(s-1))\int_{\tau_0}^\infty (\tau^2-\tau_0^2)^{s-1}   
\bigl(\cH_f^{>k}+ \sum_{j>s}\Re(4z^2\tilde T_{2j})\bigr) ( i\sigma) 
\dtau\Bigr|\leq Cc_{\tau_0}(E^{s}_{\tau_0})^2.
\end{align*} }
Here we noticed that for the estimates in Propositions \ref{prop:smalls} and \ref{prop:larges}, if $s=m$ is an integer then the singularity $(2s-[2s])^{-1}$ is compensated by the coefficient $\sin(\pi(s-1))$, while if $s=m+\frac12$ we can  do the frequency decomposition as in \eqref{decomposition:Tf}. 
Therefore for any $q\in X^s$, $s>\frac12$ there exists $\tau_0\geq C\geq2$ such that the  energy $\cE^{s}_{\tau_0}$ is well-defined in  \eqref{cElarge} satisfying \eqref{Equiv}.

Furthermore we derive the trace formula \eqref{cEtrace} by Proposition \ref{prop:G}.
 For general $\tau'\geq 2$, we can still define our energy $\cE^{s}_{\tau'}$ as in \eqref{cElarge}  since $\Tc^{-1}(\lambda)$ is holomorphic on $\cR=\{(\lambda,z)\,|\,\lambda^2=z^2+1,\,\lambda\not\in\Ic,\,\Im z>0\}$ and  $G(i\frac\tau2)$ is integrable on the finite interval $\tau\in [\tau',\tau_0]$ if $\tau'\leq \tau_0$. 
The analyticity of $\cE^{s}_{\tau'}(q)$ in $q\in X^s$ follows from the analyticity of the renormalised transmission coefficient $\Tc^{-1}(\lambda; q)$ in Theorem \ref{thm:Tc}.

\setcounter{equation}{0}%%%%%%%%%%%%%%%%%%%%%%%%%%%%%%%%%%%%%%%%\setcounter{equation}{0}%%%%%%%%%%%%%%5 
 \section{The metric space}\label{sec:metric}
 In this section we study the  metric space $X^s$ defined in \eqref{Xs}:
 $$
X^s= \{ q\in H^s_{\loc}(\R)\,:\, |q|^2-1\in H^{s-1}(\R), \quad q'\in H^{s-1}(\R) \}/ {\mathbb{S}^1},
\quad s\geq 0,
 $$
and  its endowed metric  defined in \eqref{ds}:
 $$
d^s(q,p)=
\Bigl(\int _{ \R}
 \inf_{|\lambda|=1} \|\sech( \cdot-y) (\lambda q-p)\|_{H^{s}(\R)}^2 \dy
 %+\||q|^2-|p|^2\|_{H^{s-1}(\R)} ^2
 \Bigr)^{\frac 12}.
 $$
Recall the energy $E^{s}(q)$ associated to $q\in X^s$ given by \eqref{Es2} and \eqref{Sobolev}:
\begin{equation*} 
E^{s}(q)=\Bigl( \|q'\|_{H^{s-1}_2(\R)}^2+\||q|^2-1\|_{H^{s-1}_2(\R)}^2\Bigr)^{1/2}.
\end{equation*}  

This section is devoted  to the proof of Theorem \ref{thm:metric} and will be divided into two subsections, with the first subsection devoted to the study of the metric structure (see Theorem \ref{thm:metric1} below), and the second one to the analytic structure (see Theorem \ref{thm:analytic} below). 

\subsection{The metric structure} 
We study in this subsection the metric structure of the metric space $(X^s, d^s)$.
\begin{thm}\label{thm:metric1}
  Suppose that $s\ge 0$. Then $(X^s, d^s)$ is a separable complete metric space.
  Moreover, there exists a constant $c$ depending on $\Lambda>0$ such that for any $q, p\in X^s$ with $E^s(q), E^s(p)\leq \Lambda$, 
  \[ |E^s(q) - E^s(p)| \le   c d^s(q,p). \] 
   $1+C_0^\infty(\R)$ is a dense subset. Every metric ball is contractible.
  If $s>0$ then every closed metric ball is weakly sequentially  compact.
    \end{thm}

  By weakly sequentially compact we mean that if $(q_j)$ is a sequence
  in $B=\overline{B^s_r(q)}$ with $B^s_r(q)=\{p\in X^s\,|\, d^s(p,q)<r\}$, then there is a subsequence and $p \in B$
  so that $q_{j_k} \to p$ and $|q_{j_k}|^2-1 \to |p|^2-1$ as
  distributions. If $s>0$ then the boundedness and the convergence
  $q_{j_k} \to p $ as a distribution imply that $q_{j_k}\to p $ in
  $L^2(K)$ for every compact interval $K$ and hence
  $|q_{j_k}|^2-1 \to |p|^2 -1$ in $L^1(K)$ for every compact interval $K$,
  and hence as distribution. Thus only the convergence of $q_{j,k} \to p$
  as a distribution and the weakly compactness of closed balls has to be proven.
  
  Before proving Theorem \ref{thm:metric1}, we claim the following lemma stating the relation between the energy and the metric, whose proof is postponed to the end of this subsection.
\begin{lem}\label{metricspace} 
  If $q \in X^s$ then with an absolute constant $c$ we have
  \begin{equation}\label{diameter}   d^s(1,q) \le c E^s(q). \end{equation} 
  If $p \in X^s$ and $q \in H^s_{loc} $ so that $d^s(p,q) < \infty$, then
  $ q\in X^s$ and
\begin{equation}\label{Es,ds,2}
 E^s(q) \le E^s(p) + c(1+E^s(p))^{\frac12}d^s(p,q)+ c(d^s(p,q))^2. 
 \end{equation}  
 \end{lem}

\begin{proof}[Proof of Theorem \ref{thm:metric1}]
We organize the proof into a   series of steps. 
The cutoff functions $\eta$ and $\rho$ will be chosen appropriately in each step and may vary from step to step.

\noindent{\bf Step 1.}
Suppose that  $q, p\in X^s$ with $E^{s}(q), E^{s}(p)< \infty$. Then $d^s(q,p)=0$
  iff there exists $\lambda\in \C$ with $|\lambda|=1$ such that $p = \lambda q $.
Since  $p=\lambda q$ implies $d^s(q,p)=0$ trivially, we assume that $d^s(q,p)=0$.
Then as $ \|\sech( \cdot-y) f\|_{H^{s }}
\geq  C(y,a,b) \|f\|_{H^{s }(I)}$ for any $y\in\R$ and any interval $I=[a,b]$, there exists $\lambda\in \C$ with $|\lambda|=1$ so that 
  \[ \|\lambda q -p\|_{H^s(I)} = 0. \]
  Hence $ p = \lambda q$.

\noindent{\bf Step 2. Triangle inequality.}   If $q,p,r \in H^s_{\loc} $ with
  $ d^s(q,p)<\infty$ and $d^s(p,r) <\infty$, then  we simply integrate the square of the following triangle inequality 
  \[ \inf_{|\mu|=1}  \Vert \sech(\cdot-y) (q-\mu r) \Vert_{H^s}
  \le \inf_{|\lambda_1|=1}  \Vert \sech(\cdot-y) (q-\lambda_1 p) \Vert_{H^s}
  + \inf_{|\lambda_2|=1} \Vert \sech(\cdot-y) (p-\lambda_2 r) \Vert_{H^s}, \]  
to derive the triangle inequality
\[d^s(q,r) \le d^s(q,p) + d^s(p,r). \]

\noindent{\bf Step 3. $(X^s,d^s)$ is a metric space.}

We deduce from the triangle inequality in Step 2 and \eqref{diameter} that whenever $p,q \in X^s$ then $d^s(p,q) < \infty$ and $(X^s, d^s)$ is a metric space.

\noindent{\bf Step 4. $(X^s,d^s)$ is a complete metric space.} 

Let $(q_n)$ be   representatives of a Cauchy sequence in $(X^s, d^s)$ and let $y \in \R$. There exists $q\in H^s_\loc$ and a sequence  $(\lambda_n(y))$ with $|\lambda_n(y)|=1$  such that
\[ \sech(\cdot-y) \lambda_n(y)  q_n \to \sech(\cdot-y) q \qquad \text{ in } H^s.     \]
Clearly $q$ does not depend on $y$. 
This implies pointwise convergence of the integrant with respect to $y$ in the definition of the distance function and
\[ d^s(q_n, q) \le \sup_{m\ge n} d^s(q_n,q_m)  \to 0 \qquad \text{ with } n \to \infty. \] 
Lemma \ref{metricspace} implies $q\in X^s$.

\noindent{\bf Step 5. A dense subset.} 

  We claim that   $1+C^\infty_0(\R)  \subset X^s$ is dense.
  Let $ q $ satisfy $E^s(q) < \infty$. We fix a monoton function $\eta \in C^\infty$ with $\eta(x) = 1$ for $ x > \frac12$ and $\eta(x) = 0 $ for $ x \le -\frac12$.  Since 
\begin{align*}
& \Vert (1- \eta(R\pm x)) (|u|^2-1) \Vert_{H^{s-1}}+ \Vert(1- \eta(R\pm x)) u_x \Vert_{H^{s-1}}   \to 0   
\end{align*}
as $R\to\infty$, 
%and the quantities
%    \[ \Vert \eta(R-x) (|u|^2-1) \Vert_{H^{s-1}}+ \Vert \eta(R-x) u_x \Vert_{H^{s-1}}  \to 0 \]
 %  $ \hbox{ as }R\to-\infty,$
   given $\varepsilon>0$ there exists $R_0$ so that all these quantities above are at most of size $\varepsilon$ for $ R > R_0$.

    Let $\hat u_\pm = \int_{\R} \eta'(x\mp R) u \dx$ with $R>R_0$. We claim that there exists an absolute   constant $c$ such that
    \[ ||\hat  u_\pm |-1|  \le  c\varepsilon. \]  
    This estimate follows from Lemma \ref{l:bound} below.
     Multiplying by a complex constant of modulus $1$ we may assume that 
     $\hat u_- \in [\frac12,2]$. 
     We choose $\omega\in [-\pi,\pi)$ so that
      $\hat u_{+} = e^{i\omega} |\hat u_{+}| $.  
  We define
    \[ u_R = \eta(R-x) \eta(R+x) u(x) +
    (1-\eta(R-x))  e^{ i \omega   ( \ln( 3)/ \ln ( 2+ |x/R|^2)) }  
    + (1-\eta(R+x) ). \] 
      It is not hard to see that 
      \[ \lim_{R\to \infty} d^s(u_R , u) = 0. \]
      Clearly $u_R-1$ vanishes for $ x  <-2R$ and it decays as $x \to \infty$.
      After convolving $u_R-1$ with a Dirac sequence we may assume that (without changing the notation) that in addition $ u_R \in C^\infty$. 
      By a standard cutoff argument we may assume $u_R-1\in C^\infty_0(\R)$.
      %Finally we define for $r>0$
     % \[ u_{r,R} = 1 + e^{-|x|^2/r^2} (u_R-1). \] 
     % Then
    %  \[  \lim_{r\to 0} d^s(u_{r,R}, u_R)  = 0, \]
    %  and $u_{r,R} -1 \in C^\infty_0(\R)$.  

  \noindent{\bf Step 6. Weak compactness.}
   
Let $s\ge 0$, $ q \in X^s$, $r< \infty$ and $q_n \in X^s$ so that $ d^s(q,q_n) \le r$. We claim that there exists a weakly convergent subsequence with a limit $p$ in the same closed ball. This follows by an easy modification of Step 5.
Lemma \ref{metricspace} implies that the weak limit is in $X^s$.  
 If $ s >0$ we obtain
\begin{equation} \label{weaklower}  E^s(p) \le \liminf_{n\to \infty} E^s(q_n). \end{equation} 
Indeed, if $q_n$ converges weakly to $p$ then up to choosing $\lambda_n$
\[ \eta (|q_n|^2 -1) \to \eta(|p|^2-1)  \]
in $H^{s-1}$ by compactness. We easily deduce \eqref{weaklower}.
\end{proof}

In the remainder of this subsection we will give  two technical lemmas (Lemma \ref{lem:commute} and Lemma \ref{l:bound})  and their proofs, as well as the proof of Lemma \ref{metricspace}.
 
\begin{lem}\label{lem:commute}
Let  $0<\delta<\frac14$ and  
  suppose that $\eta^{(k)} \le C \eta$ for all $k \ge 0$,  $|\eta'|\le \delta \eta$. Let $s\in \R$, then
  \[   \Vert \eta\langle D\rangle^{s} f  \Vert_{L^2}
  \le c \Vert \langle D\rangle^s ( \eta f) \Vert_{L^2}, \]
  where the operator $\langle D\rangle^s$ is defined by the Fourier multiplier $\langle \xi\rangle^s=(1+|\xi|^2)^{s/2}$.
\end{lem}

\begin{proof}  
The operator $\langle D\rangle^s$ is defined by convolution with (up to a multiplication by a power of $2\pi$)
$$
g(x)=\int_{\R}e^{i\xi x}\langle \xi\rangle^s d\xi,
$$
which is understood as inverse Fourier transform \textit{resp.} as oscillatory integral if $x\neq 0$.
To be more specific, let us assume $x>0$. Then we move the contour of the integration to 
$\{\xi+i\tau\,|\,\xi\in\R, \,-1<\tau<1\}$ in the complex space:
\begin{align*}
g(x)=\int_{\R} e^{i(\xi+i\tau)x} (1+(\xi+i\tau)^2)^{s/2} d\xi
=e^{-\tau x}\int_{\R}e^{i\xi x}(\xi^2+2i\tau\xi+1-\tau^2)^{s/2} d\xi.
\end{align*} 
We take $\tau=1$. We take the smooth cutoff function $\rho$ with $\rho=1$ around $0$ to decompose the  integration into the part close to $0$ and the part away from $0$:
\begin{align*}
e^{x}g(x)=\int_{\R}e^{i\xi x}\xi^{s/2}[\rho(\xi+2i)^{s/2}] d\xi
+\int_{\R}e^{i\xi x}\xi^s[(1-\rho)(1+2i/\xi)^{s/2}]d\xi.
\end{align*}
Similarly for $x<0$ we take $\tau=-1$.
Then by the theory of oscillatory integrals,
\begin{align*}
 |g(x)|&\leq Ce^{-|x|}\Bigl( (1+|x|)^{-1-s/2} \,+\, \bigl(1+\chi_{\{|x|\leq 1\}}|x|^{-1-s}\bigr)\Bigr),
% \\
%& \leq Ce^{-|x|} \Bigl(1+\chi_{\{|x|\leq 1\}}|x|^{-1-s}+\chi_{\{|x|\geq 1\}}|x|^{-1-s/2}\Bigr).
\end{align*}
and the exponential decay holds:
\begin{equation*}
|\d_x^k g(x)|\leq Ce^{-|x|}(1+|x|^{-1-s/2}),
\quad \forall |x|\geq 1.
\end{equation*}

We denote $g=g^s$ to emphasize the dependence of the above function $g(x)$ on $s$ and we decompose $g^s(x)$ into 
$$
g^s(x)=g_1^s(x)+g_2^s(x),
\quad g_1^s(x)=\rho(x)g^s(x),
$$
such that
\begin{equation}\label{g1g2}\begin{split}
%&\|g_1^s f\|_{H^\sigma}\leq C\|f\|_{H^{\sigma+\max\{s/2,s\}}},\quad\forall\sigma\in\R,
&|g_1^s|\leq C\chi_{\{|x|\leq 1\}}(1+|x|^{-1-s}),
\\
&|\d_x^k g_2^s(x)|\leq C_k e^{-\frac{|x|}{2}},
\hbox{ and hence }\|g_2^s\ast f\|_{H^N}\leq C_N\|f\|_{-N},\quad \forall N\in\N.
\end{split}\end{equation}
The claimed inequality is equivalent to 
\begin{align*}
\|\eta\langle D\rangle^s\eta^{-1}\langle D\rangle^{-s}f\|_{L^2}\leq c\|f\|_{L^2},
\quad \hbox{i.e.}\quad \|\eta  g^s\ast(\eta^{-1} g^{-s}\ast f) \|_{L^2}\leq c\|f\|_{L^2},
\end{align*}
and by duality it suffices to consider the case $s\geq 0$.
We do the above decomposition for $g^s$ and it remains to show 
$$
\|\eta  g_j^s\ast(\eta^{-1} g_l^{-s}\ast f) \|_{L^2}\leq c\|f\|_{L^2},\quad j,l=1,2.
$$
When $j=l=2$, then  the integral kernel of the operator on the LHS reads as
\begin{align*}
k_2^s(x,y)=\eta(x)\int_{\R} g_2^s(x-z)\eta^{-1}(z) g_2^{-s}(z-y) dz.
\end{align*}
By $|\d_x^k \eta(x)|\leq c_k e^{\delta|x-z|}\eta(z)$, the estimate follows from \eqref{g1g2}:
\begin{align*}
\int_{\R}e^{-\frac12|x-z|} e^{\delta|x-z|}e^{-\frac12  |z-y|}dz\leq Ce^{-\frac14|x-y|}.
\end{align*}
It is also straightforward to check the other cases by use of \eqref{g1g2}:
\begin{align*}
&\|\eta  g_1^s\ast(\eta^{-1} g_2^{-s}\ast f) \|_{L^2}\leq c\|g_2^{-s}\ast f\|_{H^{s}}\leq c\|f\|_{L^2},
\\
&\|\eta  g_2^s\ast(\eta^{-1} g_1^{-s}\ast f) \|_{L^2}\leq c\|g_1^{-s}\ast f\|_{H^s}\leq c\|f\|_{L^2},
\\
&\|\eta  g_1^s\ast(\eta^{-1} g_1^{-s}\ast f) \|_{L^2}\leq c\|g_1^{-s}\ast f\|_{H^{s}} \leq c\|f\|_{L^2}.
\end{align*}
  \end{proof}

We turn to another technical lemma.  Let $ \eta(x) = (1+x^2)^{-1}$. Then 
\begin{equation}\label{sech,eta} \left| \int_{\R} \sech^2(x-y)(|q|^2-1)(x) \dx  \right|
   \le C \Vert \eta(\cdot-y)  \langle D\rangle^{-1} (|q|^2-1) \Vert_{L^2}, 
   \end{equation}
   since
   \[ \eta^{-1}(\cdot-y) \sech^2(\cdot-y) \in \mathcal{S}(\R) \subset H^1(\R). \]  
   \begin{lem}\label{l:bound} 
    Let $\eta_0$ be a nonnegative Schwartz function. Then 
     \begin{equation}  \int_{\R\times \R} \eta_0(x) \eta_0(y) |q(x)-q(y)|^2 \dx \dy
     \le c \Vert \eta \langle D\rangle^{-1} q_x \Vert_{L^2}^2, \label{wD} \end{equation}  
and with $\kappa=\int_{\R}\sech^2(x)\dx$,
\begin{equation} \label{e:deviation}  
   \Big| \,    \Big|\frac1\kappa\int_{\R} \sech^2(x-y) q(x) \dx\Big|-1 \Big| \le c\left(  \Vert \eta(\cdot-y) \langle D\rangle^{-1} q_x \Vert_{L^2} + \Vert \eta(\cdot-y) \langle D\rangle^{-1} (|q|^2-1) \Vert_{L^2}\right) .   \end{equation} 
\end{lem}

\begin{proof}
Straightforward calculation yields
     \[
\begin{split} 
     \int_{\R \times \R} \eta_0(x) \eta_0(y) |q(x)&-q(y)|^2 dx dy
     =  2\Re \int_{x<y}  \eta_0(x) \eta_0(y) \int_{x<z_1,z_2<y} q'(z_1) \bar q'(z_2) dz_1 dz_2 dx dy
     \\ &  =2\Re \int_{\R\times\R} \int_{-\infty}^{\min\{z_1,z_2\} } \eta_0(x) dx \int_{\max\{z_1,z_2\}}^\infty \eta_0(y) dy q'(z_1) \bar q'(z_2) dz_1 dz_2.
\end{split} 
       \]
       Let
       \[ \rho(z_1,z_2) = 2 \int_{-\infty}^{\min\{z_1,z_2\} } \eta_0(x) dx \int_{\max\{z_1,z_2\}}^\infty \eta_0(y) dy.\]
       The assertion \eqref{wD} follows once we prove with $ \kappa_0 =2 \int_{\R} \eta_0 $,
       \[ \Vert \rho(.,.) \Vert_{L^2} + \Vert \partial_{z_1} \rho \Vert_{L^2} +
       \Vert \partial_{z_2} \rho \Vert_{L^2} + \Vert \partial^2_{z_1z_2} \rho- \kappa_0 \delta_{z_1-z_2} \eta_0(z_1)  \Vert_{L^2} \le  c. \] 
Indeed,
       \[ \partial_{z_1}  \rho(z_1,z_2) =
       \left\{ \begin{array}{rl}\displaystyle  2\eta_0(z_1) \int_{z_2}^\infty \eta_0(y) dy   & \hbox{ if }z_1 < z_2 \\
   \displaystyle    -2   \int_{-\infty}^{z_2} \eta_0(x) dx \eta_0(z_1) & \hbox{ if }z_2 < z_1
       \end{array}\right. \]
       and 
       \[ \partial_{z_1 z_2}^2  \rho(z_1,z_2) =  -2\eta_0(z_1) \eta_0(z_2)
       \hbox{ if }z_1\neq z_2. \] 
              At the diagonal $\{(z_1, z_2)\,|\,z_1=z_2\}$ we have
       \[ \partial_{z_1} \rho(z,z_+) - \partial_{z_1} \rho(z,z_-)
       =2\eta_0(z)\int^\infty_{z_+}\eta_0+2\eta_0(z)\int^{z_-}_{-\infty}\eta_0
       = 2\eta_0(z) \int_{\R} \eta_0\]   
       and hence
      \[ \partial_{z_1z_2}^2 \rho(z_1,z_2) = -2 \eta_0(z_1) \eta_0(z_2) 
      +2 \delta_{z_1-z_2} \eta_0(z_1)\int_{\R}\eta_0. \] 

We turn to the proof of \eqref{e:deviation}. 
            Let   $\kappa = \int_{\R} \sech^2(x) \dx $. 
            Then
      \[
\begin{split} 
      \left| \int_{\R} \sech^2(x-y) q(x) \dx  \right| \le & 
      \int_{\R} \sech^2(x-y) |q(x)| \dx
\\       \le &  \kappa^{1/2} \left(  \int_{\R} \sech^2(x-y) |q|^2(x) \dx \right)^{\frac12} 
\\ \le &  \kappa \Bigl( 1 + \kappa^{-1} \int_{\R} \sech^2(x-y) (|q|^2-1)(x) \dx\Bigr) 
\\ \le & \kappa + C\Vert \eta(x-y) \langle D\rangle^{-1} (|q|^2-1) \Vert_{L^2}^{\frac12},
\end{split}
\]
where in the last step we used \eqref{sech,eta}.
This implies the desired estimate \eqref{e:deviation} if for some $\varepsilon>0$, $\left| \frac1\kappa\int_{\R} \sech^2(x-y) q(x) dx\right|  \ge 1+ \varepsilon$.
%, with the constant $c$ in \eqref{e:deviation} depending on $\varepsilon$.  

      We hence fix $\varepsilon$ ($\varepsilon= \frac12$ being legitimate) and consider the case 
      \begin{equation}\label{boundedness} 
       \left|  \frac1\kappa\int_{\R} \sech^2(x-y)q(x) \dx \right| \le (1+  \varepsilon).  \end{equation} 
 Using Fubini and \eqref{wD}  we have
      \begin{equation} \label{e:diff}
\begin{split} 
      &  \left\|\frac1{\kappa} \int_{\R} \sech^2(x-y) q \dx - q \right\|^2_{L^2(\sech^2(.-y))}
      \\
&=   \int_{\R} \sech^2(x'-y)   \Big|\frac1{\kappa}\int_{\R} \sech^2(x-y) q(x) \dx - q(x') \Big|^2  \dx'
\\ 
&\le   \frac1{\kappa} \int_{\R^2} \sech^2(x'-y)  \sech^2(x-y)   |q(x) - q(x') |^2 \dx \dx'
  \\   
  &   \le   c \Vert \eta(\cdot-y) \langle D\rangle^{-1} q_x \Vert_{L^2}^2, \end{split} \end{equation} 
      and hence by triangle inequality, 
      \[ \begin{split}& \left|   \left| \frac1{\kappa}\int_{\R} \sech^2(x-y) q(x) \dx \right|^2 - 1 \right|  \le    \left|\frac1\kappa \int_{\R} \sech^2 (x-y) (|q|^2-1)(x)\dx\right|
        \\ &\qquad + \left|\frac{1}{\kappa} \int_{\R} \sech^2(x-y) |q|^2(x) \dx
        - \left| \frac1\kappa  \int_{\R} \sech^2(x-y) q(x) \dx \right|^2  
         \right|,
        \end{split} 
        \] 
 where the last term is estimated using \eqref{e:diff} and by writing it as
        $ |A^2-B^2| = |A+B| |A-B| $.  In this last step we made use of \eqref{boundedness}.  
\end{proof}

We complete this subsection with verifying the relation between metric and energy stated in Lemma \ref{metricspace}. 
\begin{proof}[Proof of Lemma \ref{metricspace}]
  We claim that there exists a constant $c>0$ so that  
  \[ d^s(q,1) \le c E^s(q). \]
  This is the first claim of the lemma. 
 We begin with the most difficult case $s=0$ and   fix $y\in \R$. Then, with $\kappa = \int_{\R} \sech^2(x) \dx $, 
\[\begin{split} 
& \inf_{|\lambda|=1}  \int_{\R} \sech^2(x-y) |q-\lambda|^2(x) \dx
 \\
 &   =  \int_{\R} \sech^2(x-y) (|q|^2+1) \dx 
  - 2 \sup_{|\lambda|=1} \Re \lambda \overline{  \int_{\R}  \sech^2(x-y) q(x) \dx}
   \\ 
   &  =  \int _{\R}\sech^2(x-y)(|q|^2-1) \dx -  
   2 \Big( \Big|\int_{\R} \sech^2(x-y) q(x) \dx\Big| - \kappa  \Big)
   \\ & \le c\left(  \Vert \eta(x-y) \langle D\rangle^{-1}(|q|^2-1 ) \Vert_{L^2}
   +  \Vert \eta(x-y) \langle D\rangle^{-1} q_x \Vert_{L^2} \right),
   \end{split}
\]
where the bound on the first term follows by \eqref{sech,eta}, and the second term by \eqref{e:deviation}.  
%The problem with this estimate is that it is quadratic on the LHS, and linear on the RHS.        
Let $\delta\in (0,1)$ be a small constant to be determined later and we define the set
\[Y_\delta=\{y\in\R\,|\,\Vert \eta(\cdot-y) \langle D\rangle^{-1} (|q|^2-1)\Vert_{L^2} +  \Vert \eta(\cdot-y) \langle D\rangle^{-1} q_x \Vert_{L^2} \ge \delta > 0\}. \]
Then
\begin{align*}
d^0(q,1)&=\Bigl(\int_{\R} \inf_{|\lambda|=1}  \int_{\R} \sech^2(x-y) |q-\lambda|^2(x) \dx\dy\Bigr)^{\frac12}
\\
&\leq \Bigl(\int_{Y_\delta^C} \inf_{|\lambda|=1}  \int_{\R} \sech^2(x-y) |q(x)-\lambda|^2  \dx\dy\Bigr)^{\frac12}
\\
&\quad
+c\Bigl( \int_{Y_\delta}\left(  \Vert \eta(x-y) \langle D\rangle^{-1}(|q|^2-1 ) \Vert_{L^2}
   +  \Vert \eta(x-y) \langle D\rangle^{-1} q_x \Vert_{L^2} \right) \dy\Bigr)^{\frac12},
\end{align*}
where the last term is bounded by $c_\delta E^0(q)$ for some constant $c_\delta$ depending on $\delta$.
%These estimates settle the case when for some $\delta>0$ 
%\[\Vert \eta(x-y) \langle D\rangle^{-1} (|q|^2-1)\Vert_{L^2} +  \Vert \eta(x-y) \langle D\rangle^{-1} q_x \Vert_{L^2} \ge \delta > 0, \]
%up to an integration in $y$. 

   It remains to consider those $y\in Y_\delta^C$.
%   \[   \Vert \eta(x-y) \langle D\rangle^{-1} (|q|^2-1 ) \Vert_{L^2} + \Vert \eta(x-y) \langle D\rangle^{-1} q_x \Vert_{L^2} \le  \delta, \]
 %  for a fixed small $\delta$
We regularize $q$ by taking the convolution with the Schwartz function $\frac1\kappa\sech^2(x)$, and we can always decompose (as in Section \ref{sec:lwp})
%We use the notation q^1 instead of q_1 (which is defined in Section 3).
$$q=b+q^1,
\hbox{ with }b\in L^2,\quad q^1=q\ast(\frac1\kappa\sech^2)\in X^\sigma, \,\forall \sigma\geq 0,$$
such that  
   \[ \Vert \sech^2(x-y)\, b \Vert_{L^2} + \Vert \sech^2(x-y)\, (q^1)_x \Vert_{L^2}
   \le c   \Vert \eta(x-y) \langle D\rangle^{-1} q_{x} \Vert_{L^2}. \]
Fix $y\in Y_\delta^C$ and in the following we  will simply denote $\eta(x-y)$ by $\eta$.  
We use Lemma \ref{l:bound} and then choose $\delta$ small enough such that 
   \[ ||q^1(y)|-1| \le c \left(  \Vert \eta \langle D\rangle^{-1} (|q|^2-1)\Vert_{L^2}
   + \Vert \eta \langle D\rangle^{-1} q_x \Vert_{L^2}\right)
   \leq c\delta<\frac12, \]
   and thus $|q^1(y)|\neq 0$.
%   for $|x-y|\le R $.  
Moreover (as in \eqref{pointwisebound}),
   \[ \Vert q^1 \Vert_{L^\infty([y-R,y+R])} 
   \le  c(1 + \Vert \eta \langle D\rangle^{-1} (|q|^2-1) \Vert_{L^2} + \Vert \eta \langle D\rangle^{-1} q_x  \Vert_{L^2}) \le c, \]
   and
   \[ \Bigl|\int_{\R} \eta (|q^1|^2-1)\dx\Bigr|  \le \left| \int  \eta (|q|^2-1)\dx \right|
   + \int \eta (|q|^2- |q^1|^2) \dx, \]
  where the second term on the righthand side above is bounded by
   \[ \left( \Vert \eta^{1/2} q \Vert_{L^2} + \Vert \eta^{1/2} q^1 \Vert_{L^2}\right) \Vert \eta^{1/2} b \Vert_{L^2}\leq c \Vert \eta^{1/2} b \Vert_{L^2}. \] 
Since
\begin{align*}
& \Big\Vert \sech^2(x-y)  \bigl( q(x)-\frac{ q^1(y)}{|q^1(y)|}\bigr) \Big\Vert_{L^2_x}
  \le \Vert \sech^2(x-y) (q- q^1)(x) \Vert_{L^2_x} 
  \\
&\qquad\qquad\qquad\qquad
 + \|\sech^2(x-y)(q^1(x)-q^1(y))\|_{L^2_x}
 +\Vert \sech^2(x-y)(|q^1(y)|-1) \Vert_{L^2_x} ,
 \end{align*}
we take the square and then integrate with respect to $y$ in the set $Y_\delta^C$, to   complete the proof of $d^0(q,1)\leq cE^0(q)$. 
  To deal with general $s>0$ we slightly modify  the  steps.

%The second claim is
%\[  \mathbf{E^{s}(q) \leq E^{s}(p) + c\sqrt{1+E^s(q)}d^s(q,p) +  cd^s(p,q)^2 .} \] 
  To prove the second claim \eqref{Es,ds,2} we consider the case when the right hand side is finite. 
 By the triangle inequality, 
 \[ \Vert |q|^2-1 \Vert_{H^{s-1}} \le \Vert |p|^2-1 \Vert_{H^{s-1}} + \Vert  |q|^2- |p|^2 \Vert_{H^{s-1}}. \]
 We now verify
\begin{equation}\label{Es,ds}
\Vert |q|^2- |p|^2 \Vert_{H^{s-1}} \le c\sqrt{1+E^s(p)+E^s(q)}d^s(p,q).  
\end{equation} 
Since for any $|\lambda|=1$,
\begin{align*}
  |q|^2- |p|^2  
 &=   \Re\bigl( (\lambda q+p)(\overline{\lambda q-p})\bigr) 
 =\Re\Bigl( \bigl(\lambda (b_q+q^1)+(b_p+p^1)\bigr)(\overline{\lambda q-p})\Bigr),
\end{align*}
where $q=b_q+q^1$ and $p=b_p+p^1$ are decompositions above,      we have for $s\geq 0$
\begin{align*}
\|  |q|^2- |p|^2  \|_{H^{s-1}}
\lesssim (\|(b_q,b_p)\|_{H^s}+\|( (q^1)_x, (p^1)_x)\|_{H^{s-1}}+\|(q^1, p^1)\|_{L^\infty}) \|\lambda q-p\|_{H^s},
\end{align*}
and hence \eqref{Es,ds} follows:
\begin{equation*} 
\begin{split} 
  \Vert |q|^2-|p|^2 \Vert^2_{H^{s-1}}  
  & \, \le
  c \int_{\R} \Vert \sech(x-y) (|q|^2-|p|^2)(x) \Vert^2_{H^{s-1}_x}\, dy  
  \\ 
  &  \le c(1+E^s(p)+E^s(q))  \int_{\R} \inf_{|\lambda|=1} \Vert \sech(x-y) (\lambda q-p) \Vert_{H^s}^2dy
  \\
&=c(1+E^s(p)+E^s(q))(d^s(q,p))^2.
  \end{split} 
\end{equation*} 
Then   for any $\epsilon>0$ small enough, there exists $c_\epsilon>0$ such that
\begin{equation}\label{q-1}
\begin{split} 
\Vert |q|^2-1 \Vert_{H^{s-1}}
% \le &  
%\Vert |p|^2-1 \Vert_{H^{s-1}}
%+ \Vert |p|^2-|q|^2 \Vert_{H^{s-1}}
%\\
 & \le  E^s(p) + c(1+E^s(p)+E^s(q))^{\frac12} d^s(p,q)
\\ & \le  \epsilon E^s(q)+E^s(p) +  c(1+E^s(p))^{\frac12}d^s(p,q)  + c_\epsilon (d^s(p,q))^2. 
\end{split}
\end{equation}

 Let $\eta \in \mathcal{S}(\R)$. We calculate
  \[
\begin{split} 
  \int_{\R} \Vert \eta(x-y) f(x) \Vert_{L^2_x}^2 dy
  =  \int_{\R} \int_{\R} | \eta(x-y) f(x) |^2 dx\, dy  
   =   \Vert \eta \Vert_{L^2}^2 \Vert f \Vert_{L^2}^2  .
\end{split}
\]   
We take $\|\eta\|_{L^2}=1$, such that
  {\small  \[
\begin{split} 
  \Vert q' &\Vert_{H^{s-1}}
=  
    \, \Bigl( \int_{\R} \Vert \eta(x-y) \langle D\rangle^{s-1} q'(x) \Vert^2_{L^2_x} dy\Bigr)^{1/2}
  \\   \le &\,  
  \Bigl(\int_{\R} \Vert \eta(x-y) \langle D\rangle^{s-1} p' \Vert_{L^2}^2 dy\Bigr)^{1/2}
  + \Bigl(\int_{\R}\inf_{\lambda} \Vert \eta(x-y)\langle D\rangle^{s-1}(q'-\lambda p') \Vert_{L^2}^2  dy\Bigr)^{1/2}
 \\ =
 &\,  \Vert p' \Vert_{H^{s-1}}
+ \Bigl(\int_{\R} \inf_{\lambda} \Vert \eta(x-y)\langle D\rangle^{s-1} (q'-\lambda p') \Vert^2_{L^2}dy\Bigr)^{1/2}.
  \end{split} 
\] } 
  Then by choosing $\epsilon$ sufficiently small in \eqref{q-1}, \eqref{Es,ds,2} follows from Lemma \ref{lem:commute} (taking $\eta=C\sech(\delta x)$ with $C>0$ such that $\|\eta\|_{L^2}=1$).
\end{proof} 
 
\subsection{The analytic structure} 
In this subsection we focus on the analytic structure of the metric space $(X^s, d^s)$.
\begin{thm}\label{thm:analytic}   
  Let $\eta\in C^\infty_0([-1,1])$ with $\eta=1$ on $[-1/2,1/2]$. 
Let $E^s(q)<\infty$. There exist $r$ and $L$ depending only on $E^s(q)$ such that the map
\begin{align}\label{analytic}
&B_r^s(q)\ni p\mapsto ((a_n)_n, b)\in l^2_d\times\tilde H^s,\hbox{ with }
\\
&\|(a_n)_n\|_{l^2_d}=\bigl(\sum_{n}|a_n-a_{n-1}|^2\bigr)^{\frac12},\notag
\\
&\tilde H^s=\{b\in H^s\,|\,\langle \eta((x-4Ln)/L)b,\eta((x-4Ln)/L)q\rangle_{H^s}\in \R,\quad \forall n\in\Z\}\notag
\end{align}
is a biLipschitz map to its image. 
If $d^s(q, q_1)<r$ then the coordinate change in the intersection is an analytic diffeomorphism with uniformly bounded derivatives. 
\end{thm}

\begin{proof}

We define
\begin{align*}
\|f\|_{H^{-1}(I)}=\sup\Bigl\{\int_{\R} fg\dx\,\Big|\, g\in C^\infty_0(I),\, \|g\|_{H^1}=1\Bigr\}.
\end{align*}
We denote
\begin{align*}
f^a(x)=\int_{-\frac12}^{\frac12} f(x+y) \dy,\quad \forall f\in L^1_\loc(\R).
\end{align*}
\begin{prop}\label{prop:lower}
There exists $\varepsilon>0$ such that
$$
\frac12\leq |q^a(x)|\leq 2
\hbox{ and }\|q\|_{L^2([x-\frac12,x+\frac12])}\geq \frac12,
$$
if 
\begin{align*}
\|q_x\|_{H^{-1}([x-\frac12,x+\frac12])}+\||q|^2-1\|_{H^{-1}([x-1,x+1])}\leq\varepsilon.
\end{align*}
In particular, if the interval $I$ satisfies 
$$
|I|\geq 6((E^s(q))^2/\varepsilon^2+1),
$$
then
$$
\|q\|_{H^s(I)}\geq E^s(q)\varepsilon^{-1}.
$$
\end{prop}
\begin{proof} 
Without loss of generality we take $x=0$ and we consider
\begin{align*}
\|q-q^a(0)\|_{L^2([-\frac12, \frac12])}^2
=\int^{\frac12}_{-\frac12} |q(y)-q^a(0)|^2\dy
=\int^{\frac12}_{-\frac12} \Bigl| q(y)-\int^{\frac12}_{-\frac12}q(x)\dx\Bigr|^2\dy,
\end{align*}
which reads as
\begin{align*}
& \int^{\frac12}_{-\frac12}\Bigl( \int^{\frac12}_{-\frac12}\int_x^y q_x(z)\dz\dx\Bigr)
\Bigl(\int_{-\frac12}^{\frac12}\int_{x'}^y\bar q_x(z')\dz'\dx'\Bigr)\dy 
\\
&:=\int^{\frac12}_{-\frac12}\int^{\frac12}_{-\frac12}
k(z,z') q_x(z)\bar q_x(z')\dz\dz'.
\end{align*}
In the above,
{\small
\begin{align*}
&k(z,z')=\int_{A(z,z')} \sign(y-x)\sign(y-x') \dy\dx\dx',
\\
&A(z,z')=\Bigl(\Bigl\{-\frac12<x<z<y<\frac12\Bigr\}\cap \Bigl\{-\frac12<x'<z'<y<\frac12\Bigr\}\Bigr)
\\
&\qquad \cup \Bigl\{-\frac12<x<z<y<z'<x'<\frac12\Bigr\}
\cup \Bigl\{ -\frac12<x'<z'<y<z<x<\frac12\Bigr\}
\\
&\qquad \cup \Bigl(\Bigl\{ -\frac12<y<z<x<\frac12\Bigr\}\cap \Bigl\{ -\frac12<y<z'<x'<\frac12\Bigr\}\Bigr),
\end{align*} 
such that
\begin{align*}
&k(z,z')=\mathbf{1}_{z<z'} (z+\frac12)(\frac12-z') 
+\mathbf{1}_{z'<z}  (z'+\frac12)(\frac12-z) 
\\
&\qquad\quad=\frac14-zz'-\frac12|z-z'|,\quad \forall (z,z')\in [-\frac12,\frac12]\times[-\frac12,\frac12],
\end{align*}}
which is symmetric and Lipschitz continuous with 
$$k(\pm\frac12,z')=k(z,\pm\frac12)=0,$$ and smooth away from the diagonal with uniformly bounded derivatives of all orders. 

We claim that 
\begin{equation}\label{utildeu}
\|q-q^a(0)\|_{L^2([-\frac12, \frac12])}
\leq c\|q_x\|_{H^{-1}([-\frac12,\frac12])},
\end{equation}
and it is equivalent to say that the integral operator with the integral kernel $k(z,z')$ maps from $H^{-1}$ to $H_0^1$.
That is, the integral operator with the integral kernel 
$$
\d_z k(z,z')=-z'+\frac12(\mathbf{1}_{z<z'}-\mathbf{1}_{z'<z})
$$
maps from $H^{-1}$ to $L^2$.
This is equivalent to the adjoint operator mapping from $L^2$ to $H_0^1$, which, by 
$\d_zk(z,\pm\frac12)=0$, is equivalent to the fact that
$$
\d_{zz'}k(z,z')=-1+\delta_{z-z'}
$$
is the integral kernel of an operator mapping from $L^2$ to $L^2$: This is obvious. 

Let $\eta\in C^\infty_0([-1,1])$ such that $0\leq \eta\leq 1$, $\eta=\frac12$ on $[-\frac12,\frac12]$ and $\int\eta=1$.
Then
\begin{align*}
\Bigl|\int_{\R} \eta(y)\bigl( |q(x+y)|^2-1\bigr) \dy\Bigr|\leq c\||q|^2-1\|_{H^{-1}([x-1,x+1])},
\end{align*}
and hence by $\frac12\int_{-\frac12}^{\frac12} |q(x+y)|^2\dy\leq \int \eta|q(x+\cdot)|^2=\int\eta(|q(x+\cdot)|^2-1)+\int\eta$,
$$
|q^a(x)|+\|q\|_{L^2([x-\frac12,x+\frac12])}\leq c\bigl(1+\||q|^2-1\|_{H^{-1}([x-1,x+1])}\bigr)^{\frac12}.
$$
Thus with a different test function, still denoted by $\eta$, with $\eta\in C^\infty_0([-\frac12,\frac12])$ and $\eta=1$ on $[-\frac14,\frac14]$, $\int\eta=1$, $0\leq\eta\leq1$, we have for any $x\in\R$,
{\small\begin{align*}
& \bigl| |q^a(x)|^2-1\bigr|
\leq \Bigl| \int_{\R} \eta(y)\bigl( |q(x+y)|^2-1\bigr) \dy\Bigr|
+\int_{\R} \eta(y)\bigl|q^a(x)+q(x+y)\bigr|\bigl| q^a(x)-q(x+y)\bigr|\dy
\\
&\qquad\leq c\||q|^2-1\|_{H^{-1}([x-\frac12,x+\frac12])}+c\bigl(1+\||q|^2-1\|_{H^{-1}([x-1,x+1])}\bigr)
\|q_x\|_{H^{-1}([x-\frac12,x+\frac12])},
\end{align*}}
where we used \eqref{utildeu} to control $\|q^a(x)-q(x+\cdot)\|_{L^2([-\frac12,\frac12])}$.
Therefore there exists $\varepsilon<0$ such that if $\|q_x\|_{H^{-1}([x-\frac12,x+\frac12])}+\||q|^2-1\|_{H^{-1}([x-1,x+1])}\leq\varepsilon$ then 
$$\frac12\leq |q^a(x)|\leq 2
\hbox{ and }\|q\|_{L^2([x-\frac12, x+\frac12])}\geq \frac12.$$

By Tschebycheff's inequality we have
\begin{align*}
\#\{n: \|q_x\|_{H^{-1}([2n-1, 2n+1])}+\||q|^2-1\|_{H^{-1}([2n-1,2n+1])}\geq\varepsilon\}
\leq \varepsilon^{-2} (E^s(q))^2,
\end{align*}
and hence if $|I|\geq 6(\varepsilon^{-2}(E^s(q))^2+1)$ then for all $s\geq 0$,
\begin{align*}
\|q\|_{H^s(I)}^2\geq \|q\|_{L^2(I)}^2\geq \frac14\cdot\frac23|I|\geq \varepsilon^{-2}(E^s(q))^2+1.
\end{align*}
\end{proof}

After these preparation we turn to the crucial construction. 
Let
\[ L = 6 (\varepsilon^{-2} (E^s(q)^2+1)) \]
in the sequel. We replace $\sech(x) $ by $\sechL(x)=\frac{e^L}{2}\varphi\ast(\sech(\max\{L,  |x|\}))$ for some fixed smooth compactly supported function $\varphi$.
This function is close to $1$ on an interval of length $2L$. 
For $q\in X^s$,  by Proposition \ref{prop:lower}, we have  
\begin{equation}\label{u:lower}
 \Vert \sechL(x-y) q(x) \Vert_{H^s_x}(y) \ge  \frac12E^s(q)/\varepsilon,\quad \forall y\in\R. 
 \end{equation}

We replace $\sech$ by $\sech_L$ in the definition of the distance. This leads to an equivalent metric with constants of size $e^L/2$.  
Expand the quantity in the integrand in the definition of $d^s(q,p)$ as
\begin{align*}
\|\sechL(x-y)(\lambda q-p)\|_{H^s_x}^2
&=\|\sechL(x-y)q\|_{H^s_x}^2
+\|\sechL(x-y)p\|_{H^s_x}^2
\\
&\quad-2\Re\bigl[ \bar \lambda \langle \sechL(x-y)p, \sechL(x-y)q\rangle_{H^s_x}\bigr].
\end{align*}
Let
$$
\mu=\mu(y):=\langle \sechL(x-y)p, \sechL(x-y)q\rangle_{H^s_x}.
$$
If  $\mu\neq 0$, we take $\lambda=\lambda(y)=\frac{\mu(y)}{|\mu(y)|}$ such that
\begin{align*}
\inf_{|\lambda|=1}\|\sechL(x-y)(\lambda q-p)\|_{H^s_x}^2
&=\|\sechL(x-y)q\|_{H^s_x}^2
+\|\sechL(x-y)p\|_{H^s_x}^2-2|\mu(y)|.
\end{align*}

%{\color{red}  
Suppose that $d^s(q,p) \le \frac1{32}E^s(q)/\varepsilon$. 
Then for any $y_0\in\R$ and the interval $I_0=[y_0-\frac12, y_0+\frac12]$, we have
\begin{align*}
&e^{-\frac12}\inf_{|\lambda|=1}\|\sechL(x-y_0)(\lambda q(x)-p(x))\|_{H^s_x}^2
\\
&\leq \int_{I_0}\inf_{|\lambda|=1}e^{-|y-y_0|}\|\sechL(x-y_0)(\lambda q(x)-p(x))\|_{H^s_x}^2\dy
\\
&\leq \int_{I_0}\inf_{|\lambda|=1}\|\sechL(x-y)(\lambda q(x)-p(x))\|_{H^s_x}^2 \dy
\leq d^s(q,p)\leq \frac1{32}E^s(q)/\varepsilon.
\end{align*} 
This together with \eqref{u:lower} implies that  given $y$ there exists $\lambda $ with modulus $1$ so that
\begin{equation}\label{lambdau-v:lower}
 \Vert \sechL(\lambda q - p) \Vert_{H^s} \le \frac14E^s(q)/\varepsilon
\leq \frac12\|\sechL(x-y)q\|_{H^s},
\end{equation}
and hence
\begin{equation}\label{mu:lower}
\begin{split} 
  &|\mu| =    |\langle \sechL(x-y)p, \sechL(x-y)q\rangle_{H^s_x}|
  \\  & \ge   \Vert \sechL(x-y) q \Vert_{H^s}^2 - \left| \langle \sechL(x-y)(p-\lambda q), \sechL(x-y)q\rangle_{H^s_x}\right|
\\   &\ge \Vert \sechL(x-y) q \Vert_{H^s}^2 - \frac12 \Vert \sechL(x-y)q\Vert_{H^s_x}^2 =  \frac12 \Vert \sechL(x-y) q \Vert_{H^s}^2.
\end{split}
\end{equation}

We are going to study the map
$$
p(x)\mapsto \lambda(y)=\frac{\mu(y)}{|\mu(y)|},
$$
in a small ball around $q$ with the radius depending only on $E^s(q)$.
\begin{lem}\label{lem:lambda'}
If $d^s(q,p)< \frac1{32}E^s(q)/\varepsilon$, then 
$$
\|\lambda_y\|_{H^N}\leq c_N d^s(q,p), \quad \forall N\in\N.
$$
\end{lem}

\begin{proof}
We calculate
\begin{align*}
\lambda_y=\frac{\mu_y}{|\mu|}-\frac12\frac{|\mu|^2\mu_y+\mu^2\bar\mu_y}{|\mu|^3}
=\frac12\frac{\mu(\bar\mu\mu_y-\mu \bar\mu_y)}{|\mu|^3}.
\end{align*}
If $p=\lambda q$, then $\lambda_y=0$.

We differentiate $\mu$ to get
\begin{align*}
 \mu' &= \frac{d}{dy} \langle \sechL(x-y) p , \sechL(x-y) q \rangle 
 \\
&= -\langle \sechL'(x-y) p, \sechL(x-y) q \rangle 
- \langle \sechL(x-y) p, \sechL'(x-y) q \rangle,
\end{align*}  
and for notational simplicity we denote $\sechL(x-y)=\rho$ and $\sechL'(x-y)=\rho'$ such that $\mu=\langle \rho p, \rho q\rangle$ and $\mu'=-\langle\rho' p, \rho q\rangle-\langle\rho p, \rho'q\rangle$ in the following of the proof.
We expand $p=\lambda q+(p-\lambda q)$ with $|\lambda(y)|=1$ to  see that the difference $\bar\mu\mu'-\mu\bar\mu'$ is the summation of 
\begin{align*}
&\langle\rho(\overline{p-\lambda q}), \rho \bar q\rangle\mu'
-\bar\lambda\langle\rho\bar q, \rho\bar q\rangle
\bigl( \langle \rho'(p-\lambda q), \rho q\rangle
+\langle \rho(p-\lambda q), \rho' q\rangle\bigr)
\\
&-\langle\rho(p-\lambda q), \rho q\rangle\bar\mu'
+\lambda\langle\rho q, \rho q\rangle
\bigl( \langle \rho'(\overline{p-\lambda q}), \rho \bar q\rangle
+\langle \rho(\overline{p-\lambda q}), \rho'\bar q\rangle\bigr)
\end{align*}
and
\begin{align*}
&-\bar\lambda\langle\rho\bar q, \rho\bar q\rangle
\bigl( \lambda\langle\rho'q, \rho q\rangle
+\lambda\langle\rho'q, \rho q\rangle\bigr)
+\lambda\langle \rho q, \rho q\rangle
\bigl( \bar\lambda\langle \rho'\bar q, \rho\bar q\rangle
+\langle\rho\bar q, \rho'\bar q\rangle\bigr)
\\
&=-\|\rho q\|_{H^s}^22\Re\langle\rho'q, \rho q\rangle
+\|\rho q\|_{H^s}^22\Re\langle\rho'q, \rho q\rangle=0.
\end{align*}
Therefore
\begin{align*}
|\lambda_y|\leq \frac12\sum_{\rho_1,\rho_2,\rho_3,\rho_4\in\{\rho,\rho'\}}\frac{\|\rho_1(p-\lambda q)\|_{H^s}\|\rho_2 q\|_{H^s}\bigl(|\langle\rho_3 p,\rho_4 q\rangle|+|\langle\rho q, \rho q\rangle|\bigr)}{|\mu|^2},
\end{align*}
and by \eqref{lambdau-v:lower}, \eqref{mu:lower},
$$
|\lambda_y|\leq c \|\rho(p-\lambda q)\|_{H^s},
\hbox{ and hence }\|\lambda_y\|_{L^2}\leq cd^s(q,p).
$$ 
We easily obtain the claimed estimates for higher order derivatives.
\end{proof}

Next we study what happens when we modify the weight.
Let $\eta\in C^\infty_0([-1,1])$ with $\eta=1$ on $[-1/2,1/2]$ and  we define
\begin{align*}
\tilde\mu=\tilde\mu(y)=\langle \eta((\cdot-y)/L)p, \eta((\cdot-y)/L)q\rangle_{H^s},
\quad \tilde\lambda=\tilde\lambda(y)=\frac{\tilde\mu(y)}{|\tilde\mu(y)|}.
\end{align*}
\begin{lem} \label{lem:tildelambda}
Assume the same hypotheses as in Lemma \ref{lem:lambda'}, then
$$
\|\tilde\lambda-\lambda\|_{H^N}\leq c_Nd^0(q,p), \quad \forall N\in\N,
$$
if the righthand side is bounded by a constant depending on the energy.
\end{lem}

\begin{proof}
  Since
  \[ |\tilde \lambda - \lambda| = \left|\frac{|\tilde\mu(y)| \mu(y) - |\mu(y)| \tilde \mu(y) }{ |\mu(y) ||\tilde \mu(y)|} \right|, \]
  this amount to bounding the difference
  \[ A= |\tilde \mu(y) | \mu(y) - |\mu(y)| \tilde \mu(y). \] 
  This vanishes again for $ \lambda q$ and we can continue as in the proof of Lemma \ref{lem:lambda'} since
  \[ |A|\le A_1 +A_2 \]
  where
  \[
\begin{split} 
  A_1  = &  \Big| \bigl|\langle \eta (x-y) p , \eta(x-y) q \rangle\bigr|\langle \sech_L (x-y) p , \sech_L(x-y) q \rangle \\ & - \bigl|\langle \eta (x-y)\lambda q , \eta(x-y) q \rangle\bigr|
  \langle \sech_L (x-y) \lambda q , \sech_L(x-y) q \rangle \Big|
  \\ \le &    \bigl|\langle \eta (x-y) (p-\lambda q) , \eta(x-y) q \rangle\bigr|
  \bigl|\langle \sech_L (x-y) q , \sech_L(x-y) q \rangle \bigr|
  \\ & +  \bigl|\langle \eta (x-y) q , \eta(x-y) q \rangle\bigr|
  \bigl |\langle \sech_L (x-y) (p-\lambda q) , \sech_L(x-y) q \rangle\bigr|
\end{split} 
  \]
  \[
\begin{split} 
  A_2 = & \Big| |\langle \sech_L (x-y) p , \sech_L(x-y) q \rangle|\langle \eta (x-y) p , \eta(x-y) q \rangle\\ &  - |\langle \sech_L (x-y) \lambda q , \sech_L(x-y) q \rangle|\langle \eta (x-y) \lambda q , \eta(x-y) q \rangle \Big|.
  \end{split} \]
  
  Bounding the derivatives is done in the same fashion as for $\lambda$.

\end{proof} 

There exists $r<\frac{1}{32}E^s(q)/\varepsilon$ small enough such that for any $p\in B^s_r(q)$, we can construct a function $\theta=\theta(x)$ using  the function $\tilde\lambda=\tilde\lambda(y)$ as follows:
\begin{enumerate}
\item We choose a sequence $(a_n)_{n\in \Z}$ so that
  \[ e^{ia_n}  = \tilde \lambda(4nL) \]
  and
  \[ \sum_n |a_{n-1}-a_n|^2 <  CL d^s(q,p) <\frac12 \]
  where the latter is satisfied for small enough $r$. The sequence is unique up to the addition of a multiple of $2\pi$. 
  
\item We fix a smooth partition  of unity
  \[ \sum_{n} \rho((x-4Ln)/L) = 1 \hbox{ with }\, \rho=1 \, \hbox{ on }\, [-1,1]=\Supp(\eta) \] 
  and define  
  \[ \theta(x) = \sum a_n \rho((x-4Ln)/L) .\]
\item We define the map
  \[ p \to e^{-i\theta}p -q =:b.\] 
  \end{enumerate} 

This defines the map \eqref{analytic} in Theorem \ref{thm:analytic}:
$$
B^s_r(q)\ni p\mapsto ((a_n)_n, b)\in l^2_d\times\tilde H^s.
$$ 
Indeed, it suffices to show $b\in \tilde H^s$. Since
\begin{align*}
\|\sechL(x-y)b(x)\|_{H^s_x}
&\leq \|\sechL(x-y)((\chi(y))^{-1}p(x)-q(x))\|_{H^s_x}
\\
&\quad+\|\sechL(x-y)((\chi(y))^{-1}-(\tilde\chi(x))^{-1})p(x)\|_{H^s_x}
\\
&\quad+\|\sechL(x-y)((\tilde\chi(x))^{-1}-e^{-i\theta(x)})p(x)\|_{H^s_x},
\end{align*}
by Lemma \ref{lem:commute}, Lemma \ref{lem:lambda'} and Lemma \ref{lem:tildelambda}, we derive that
\begin{align*}
\|b\|_{H^s}^2\lesssim \int\|\sechL(x-y)b\|_{H^s_x}^2\dy\lesssim d^s(q,p).
\end{align*}
Furthermore, it is straightforward to calculate
\[
\begin{split} 
&\langle \eta((x-4nL)/L)b, \eta((x-4nL)/L)q\rangle
\\
= & 
e^{-ia_n} \langle \eta((x-4nL)/L) p, \eta((x-4nL)/L)q\rangle
-\langle \eta((x-4nL)/L)q, \eta((x-4nL)/L)q\rangle
\\ = &  |\tilde \mu(4nL)|-\|\eta((\cdot-4nL)/L)q\|_{H^s}^2 \in \R.
\end{split} 
\]

Let us consider the map
$$
l^2_d\times\tilde H^s\ni ((a_n)_n, b)\mapsto e^{i\theta}(q+b),
$$
where the function $\theta$ is constructed as above from $\theta(4n L)=a_n$.
Since
\begin{align*}
|e^{i\theta}(q+b)|^2-|e^{i\tilde\theta}(q+\tilde b)|^2
=|b|^2-|\tilde b|^2+2\Re\bigl[\bar q(b-\tilde b)],
\end{align*}
we derive that
\begin{align*}
\bigl\| |e^{i\theta}(q+b)|^2-|e^{i\tilde\theta}(q+\tilde b)|^2\bigr\|_{H^{s-1}}
\leq c \|b-\tilde b\|_{H^s}.
\end{align*}
Furthermore, we also derive (expanding the norm) 
\begin{align*}
&\int_{\R} \Bigl\|\sech(x-y) \Big( e^{i\tilde\theta(y)-i\theta(y)} e^{i\theta(x)}(q+b) - e^{i\tilde\theta(x)}(q+\tilde b)\Big)\Bigr\|_{H^s_x}^2 \dy
\\
&\leq c\bigl( \|b-\tilde b\|^2_{H^s}+\|\bigl((a_n)-(\tilde a_n)\bigr)_n\|^2_{l^2_d}\bigr).
\end{align*}

On the other side,   the map
$$
B_r^s(q)\ni p\mapsto \tilde\lambda(y)=\frac{\langle \eta((x-y)/L)p, \eta((x-y)/L)q\rangle_{H^s_x}}{|\cdot|}\in |D|^{-1}H^N
$$
and the map
$$
p\mapsto b=e^{-i\theta} p-q
$$
are Lipschitz continuous. This proves the biLipschitz continuity.

Finally, the following two maps describing the coordinate changes are smooth:
\begin{align*}
&((a_n)_n, b)\mapsto \tilde\lambda_1(y)
=\frac{\langle \eta((x-y)/L)(e^{i\theta(x)}(q+b)), \eta((x-y)/L)q_1\rangle_{H^s_x}}{|\cdot|}\in |D|^{-1}H^N,
\\
&((a_n)_n, b)\mapsto b_1=e^{i\theta-i\theta_1}(q+b) -q_1 .
\end{align*}

\end{proof}

\appendix
\section{Calculation of the quadratic term}\label{appendix}
\setcounter{equation}{0} 
We prove Lemma \ref{Lem:IA} here: We derive  the expansion \eqref{Expansion:lnT,T2} of $\ln\Tc^{-1}$ from the  expansion \eqref{Expansion:lnT} on the imaginary axis  \eqref{ImaginAxis}. 
 It suffices to show
\begin{equation}\label{PhiXY}
\Phi+T_2=\tilde T_2+\tilde T_3,
\quad \hbox{ when }(\lambda,z)=(i\sqrt{\tau^2/4-1}, i\tau/2),
\,\, \zeta=\lambda+z,
\, \tau\geq 2,
\end{equation}
where $\Phi$, $T_2=\<XY>$, $\tilde T_2$  are given in \eqref{Phi}, \eqref{S}, \eqref{T2} respectively:
{\small\begin{align*}
&\Phi:=-\frac{i}{2z}\int_{\R}\frac{(|q|^2-1)^2}{|q|^2-\zeta^2}\dx 
+\frac1{2z\zeta}\int_{\R} \frac{q'\bar q (|q|^2-1)}{(|q|^2-\zeta^2)} \dx,
\\
&\<XY>=\int_{x<y}  e^{\varphi(y)-\varphi(x)}q_3(x) q_2(y)  \dx \dy,
\\
&\tilde T_2(i\sigma)
 = \frac{1}{4z^2}\int_{x<y}e^{2iz(y-x)}
\Bigl( (|q|^2-1)(y)(|q|^2-1)(x)+ q'(y)\bar q'(x)\Bigr) \dx\dy
\\
& \quad -i\frac{z+\zeta}{4z^3\zeta^2}\int_{\R}\Im(q'\bar q)(|q|^2-1)\dx 
+\frac{1}{8z^3\zeta} 
\int_{x<y}e^{2iz(y-x)}\Bigl( q'(y) 
 \bar q'(x)- \bar q'(y)q'(x)\Bigr)dxdy.
\end{align*}}
Here, $\tilde T_3$ identifies the summation of the cubic terms in $\ln \Tc^{-1}$ and reads as the finite     linear combination of the integrals of type \eqref{ReT3}, that is,  of the cubic terms from the following set:
{\small\begin{equation}\label{H}\tag{H}\begin{split}
 \Bigl\{
& \int_{\R}\Bigl(\frac{|q|^2-1}{\tau^2}\Bigr)^2 h\dx,
\,\<XBY>,
\, \int_{x<y}e^{2iz(y-x)} h_1(y)\Bigl(\frac{|q|^2-1}{\tau^2}h_2\Bigr)(x)\dx\dy, 
\\
&  \int_{x<y}e^{2iz(y-x)}\Bigl(\frac{|q|^2-1}{\tau^2} h_1\Bigr)(y)h_2(x)\dx\dy,
\\
&\, \int_{x<y}e^{2iz(y-x)} h_1(y) \int^y_x\frac{q'\hbox{ or }\bar q'}{\tau} \dm h_2(x)\dx\dy\,\Big|\, h, h_1, h_2\in O\Bigr\}.
\end{split}\end{equation}   }
Here  $\<XBY>=\int_{x<y}(e^{\varphi(y)-\varphi(x)}-e^{2iz(y-x)})q_2(y)q_3(x)\dx\dy$ is defined  in \eqref{J} and the set $O=\bigl\{ P\cdot\frac1\tau (|q|^2-1), \, P\cdot \frac1\tau q', \, P\cdot\frac1\tau \bar q'\bigr\}$ (defined in \eqref{O}) where $P$  is   polynomial of form \eqref{P}: $P=P(\frac{1}{\omega^{-2}|q|^2+1}, \frac{1}{\tau}q, \frac{1}{\tau}\bar q, \frac{1}{\tau^2}(|q|^2-1))$. 

For notational simplicity,  we will always denote $H$ to be the finite summation of some   cubic terms from the above set \eqref{H}, which may change from line to line. 
Thus  the goal equality \eqref{PhiXY} reads as
{\small\begin{equation}\label{PhiH}
 \Phi+\int_{x<y}e^{2iz(y-x)}q_2(y)q_3(x)\dx\dy=\tilde T_2+H,
 \end{equation} }
and we are going to decompose the quantity $\Phi+\int_{x<y}e^{2iz(y-x)}q_2(y)q_3(x)\dx\dy$ into the quadratic and cubic terms in the following. 
We will use freely the equality in \eqref{RelationZeta}: $\frac{1}{|q|^2-\zeta^2}=-\frac{1}{2z\zeta}+\frac{|q|^2-1}{2z\zeta(|q|^2-\zeta^2)}$.

We can first rewrite $\Phi$ as the finite summation of quadratic and cubic terms:
{\small\begin{equation}\label{PhiIm}\begin{split}
&\Phi=\frac{i}{4z^2\zeta}\int_{\R}(|q|^2-1)^2\dx
-\frac{i}{4z^2\zeta}\int_{\R}\frac{(|q|^2-1)^3}{|q|^2-\zeta^2}\dx
%+\frac1{2z\zeta}\int_{\R} \frac{\bar q q'(|q|^2-1)}{(|q|^2-\zeta^2)} \dx
\\
&\qquad-\frac{1}{4z^2\zeta^2}\int_{\R}q'\bar q(|q|^2-1)\dx
+\frac{1}{4z^2\zeta^2}
\int_{\R}\frac{q'\bar q (|q|^2-1)^2}{ |q|^2-\zeta^2}\dx
\\
&\quad=\frac{i}{4z^2\zeta}\int_{\R}(|q|^2-1)^2\dx
-\frac{1}{4z^2\zeta^2}\int_{\R}q'\bar q(|q|^2-1)\dx+H.
\end{split}\end{equation}
Recall  $q_2, q_3$  defined in \eqref{q1234} such that the product $q_2(y)q_3(x)$ reads as 
{\small\begin{align*}\label{Product:q2q3} 
q_2(y)q_3(x)= 
&\frac{q'(y)\bar q'(x) }{4z^2} 
- \frac{ q'(y)  }{4z^2 }
\frac{(|q|^2-1)\bar q'}{ |q|^2-\zeta^2}(x)
+  \frac{\zeta}{2z  }  
 \frac{ (|q|^2-1)q'}{ |q|^2-\zeta^2} (y) 
  \frac{ \bar q' }{ |q|^2-\zeta^2} (x)\notag
\\
& 
+ {i\zeta}\Bigl( \frac{  q' }{|q|^2-\zeta^2}(y) 
\frac{ (|q|^2-1)\bar q  }{|q|^2-\zeta^2}(x)
- \frac{(|q|^2-1)q }{ |q|^2-\zeta^2}(y)
 \frac{\bar q'  }{|q|^2-\zeta^2}(x)\Bigr)
\\
&+  \frac{(|q|^2-1)q }{ |q|^2-\zeta^2}(y) \frac{ (|q|^2-1)\bar q  }{ |q|^2-\zeta^2}(x).\notag
\end{align*} }
We decompose $\int_{x<y}e^{2iz(y-x)}q_2(y)q_3(x)\dx\dy$ into 
{\small\begin{equation}\label{XYH}\begin{split} 
& \int_{x<y}e^{2iz(y-x)}q_2(y)q_3(x)\dx\dy
=\frac{1}{4z^2}\int_{x<y}e^{2iz(y-x)} q'(y)\bar q'(x)  \dx\dy
+G+H,
\end{split}\end{equation}
 \begin{align*}
\hbox{with }&G=  \frac{i}{4z^2\zeta} 
\int_{x<y}e^{2iz(y-x)}\Bigl( q'(y) 
\bigl( (|q|^2-1)\bar q  \bigr)(x)
- \bigl( (|q|^2-1)q \bigr)(y)
 \bar q'  (x)\Bigr)\dx\dy
\\
&\qquad\qquad+  \frac{1}{4z^2\zeta^2} 
\int_{x<y}e^{2iz(y-x)}\bigl( (|q|^2-1)q \bigr)(y)
\bigl( (|q|^2-1)\bar q  \bigr)(x)\dx\dy.
\end{align*}}  
By virtue of $q(y)-q(x)=\int^y_x q'(m) dm,$ 
\begin{align*}
G=&\frac{i}{4z^2\zeta} 
\int_{x<y}e^{2iz(y-x)}\Bigl( (q'\bar q )(y) 
  (|q|^2-1)(x)
-   (|q|^2-1) (y)
(q \bar q'  )(x)\Bigr)\dx\dy
\\
&+  \frac{1}{4z^2\zeta^2} 
\int_{x<y}e^{2iz(y-x)} (|q|^2-1)(y)
\bigl( (|q|^2-1)|q|^2  \bigr)(x)\dx\dy
+H,
\end{align*}
and hence
\begin{align*}
G=&\frac{i}{4z^2\zeta} 
\int_{x<y}e^{2iz(y-x)}\Bigl( \Re(q'\bar q )(y) 
  (|q|^2-1)(x)
-   (|q|^2-1) (y)
\Re(q' \bar q)(x)\Bigr)\dx\dy
\\
&-\frac{1}{4z^2\zeta} 
\int_{x<y}e^{2iz(y-x)}\Bigl( \Im(q'\bar q )(y) 
  (|q|^2-1)(x)
+ (|q|^2-1) (y)
\Im(q'\bar q )(x)\Bigr)\dx\dy
\\
&+  \frac{1}{4z^2\zeta^2} 
\int_{x<y}e^{2iz(y-x)} (|q|^2-1)(y)
  (|q|^2-1)(x)\dx\dy
+H.
\end{align*}
Noticing $2\Re(q'\bar q )=(|q|^2-1)'$ and the integration by parts
  \begin{align*}
  \int_{x<y}e^{2iz(y-x)}g(y)h(x)\dx\dy
& =\frac{i}{2z}\bigl(\int_{\R}gh\dx- \int_{x<y}e^{2iz(y-x)}g(y)h'(x)\dx\dy\bigr)
 \\
 &
 =\frac{i}{2z}\bigl(\int_{\R}gh\dx+\int_{x<y}e^{2iz(y-x)}g'(y)h(x)\dx\dy\bigr), 
\end{align*}
we derive that
\begin{align*}
&G=\frac{1}{2z\zeta} 
\int_{x<y}e^{2iz(y-x)} (|q|^2-1) (y)
  (|q|^2-1)(x) \dx\dy
  -\frac{i}{4z^2\zeta}\int_{\R}(|q|^2-1)^2\dx
\\
&+\frac{i}{4z^3\zeta} 
\int_{x<y}e^{2iz(y-x)}\Bigl( \Im(q'\bar q )(y) 
\Re(q'\bar q)(x)
-\Re(q'\bar q) (y)
\Im(q'\bar q )(x)\Bigr)\dx\dy
\\
&
-\frac{i}{4z^3\zeta}\int_{\R}\Im(q'\bar q)(|q|^2-1)\dx+  \frac{1}{4z^2\zeta^2} 
\int_{x<y}e^{2iz(y-x)} (|q|^2-1)(y)
  (|q|^2-1)(x)\dx\dy
+H.
\end{align*}
Hence by $\frac{1}{2z\zeta}+\frac{1}{4z^2\zeta^2}=\frac{1}{4z^2}$
\begin{align*}
G=&\frac{1}{4z^2} 
\int_{x<y}e^{2iz(y-x)} (|q|^2-1) (y)
  (|q|^2-1)(x) \dx\dy
  -\frac{i}{4z^2\zeta}\int_{\R}(|q|^2-1)^2\dx
\\
&+\frac{1}{8z^3\zeta} 
\int_{x<y}e^{2iz(y-x)}\Bigl( (q'\bar q )(y) 
 (q\bar q')(x)
- (q\bar q') (y)
 (q'\bar q )(x)\Bigr)\dx\dy
\\
&
-\frac{i}{4z^3\zeta}\int_{\R}\Im(q'\bar q)(|q|^2-1)\dx +H,
\end{align*}
which, by virtue of $q(y)-q(x)=\int^y_x q'(m) dm$ and $|q|^2=(|q|^2-1)+1$ again, reads as 
\begin{align*}
&G=\frac{1}{4z^2} 
\int_{x<y}e^{2iz(y-x)} (|q|^2-1) (y)
  (|q|^2-1)(x) \dx\dy
  -\frac{i}{4z^2\zeta}\int_{\R}(|q|^2-1)^2\dx
\\
&+\frac{1}{8z^3\zeta} 
\int_{x<y}e^{2iz(y-x)}\Bigl( q'(y) 
 \bar q'(x)
- \bar q'(y)q'(x)\Bigr)\dx\dy 
-\frac{i}{4z^3\zeta}\int_{\R}\Im(q'\bar q)(|q|^2-1)\dx +H.
\end{align*}
To conclude, we arrive at \eqref{PhiH} and hence \eqref{PhiXY} by summing up  \eqref{PhiIm} and \eqref{XYH} (noticing again  $\int_{\R}\Re(q'\bar q)(|q|^2-1)\dx=0$).
}

\end{document}